\documentclass[a4paper,12pt,reqno,twoside]{amsart}

\usepackage[utf8]{inputenc} 		
\usepackage[T1]{fontenc}

\usepackage{dirtytalk}
\usepackage{amssymb}
\usepackage{amsmath}
\usepackage{graphicx}
\usepackage[colorlinks,citecolor=blue,urlcolor=blue]{hyperref}
\usepackage{cite}
\usepackage[hmargin=2.5cm,vmargin=2.5cm]{geometry}
\usepackage{mathrsfs} 
\usepackage{ellipsis}

{
	\theoremstyle{plain}
	
}

\theoremstyle{definition}
\newtheorem{exa}{Example}[section]

\newcommand{\be}{\begin{eqnarray}}
\newcommand{\ee}{\end{eqnarray}}

\newcommand{\beq}{\begin{equation}}
\newcommand{\eeq}{\end{equation}}
\newcommand{\beqn}{\begin{equation*}}
\newcommand{\eeqn}{\end{equation*}}

\newcommand{\round}[1]{\lfloor#1\rfloor}

\newtheorem{thm}{Theorem}[section]
\newtheorem{prop}[thm]{Proposition}
\newtheorem{cor}[thm]{Corollary}
\newtheorem{lem}[thm]{Lemma}

\theoremstyle{definition}
\newtheorem{claim}[thm]{Claim}
\newtheorem{remark}[thm]{Remark}

\newcommand\cA{{\mathcal A}}
\newcommand\cB{{\mathcal B}}
\newcommand\cC{{\mathcal C}}
\newcommand\cD{{\mathcal D}}
\newcommand\cE{{\mathcal E}}
\newcommand\cF{{\mathcal F}}

\newcommand\cI{{\mathcal I}}
\newcommand\cJ{{\mathcal J}}
\newcommand\cK{{\mathcal K}}
\newcommand\cL{{\mathcal L}}
\newcommand\cM{{\mathcal M}}
\newcommand\cN{{\mathcal N}}

\newcommand\cP{{\mathcal P}}

\newcommand\cT{{\mathcal T}}

\newcommand\bN{{\mathbb N}}

\newcommand\bR{{\mathbb R}}

\newcommand\bZ{{\mathbb Z}}

\newcommand\fD{{\mathfrak D}}

\newcommand\fF{{\mathfrak F}}

\newcommand{\id}{{\mathrm{id}}}
\newcommand{\ve}{\varepsilon}

\def\bfC{\mathbf{C}}

\def\bfE{\mathbf{E}}

\def\bfP{\mathbf{P}}

\usepackage{enumitem}  
\newlength{\lmar} 
\begin{document}

\title[Multivariate Berry--Esseen theorem for dynamical systems]{A multivariate 
Berry--Esseen theorem for time-dependent expanding
dynamical systems}

\author[Juho Lepp\"anen]{Juho Lepp\"anen}

\thanks{\textsc{Department of Mathematics, 
		Tokai University, Kanagawa, 259-1292, Japan}}
\thanks{ \textit{E-mail address}: leppanen.juho.heikki.g@tokai.ac.jp}
\thanks{\textit{Date}:  March 30, 2025}
\thanks{This research was supported by 
	JSPS via the project LEADER. The author is thankful to the referees for their helpful comments, in particular to one referee for identifying a mistake in the original proof of the main theorem related to the decomposition in Section~\ref{sec:Ri_decomp}.}


 





\begin{abstract} 
We adapt Stein's method to obtain Berry--Esseen type error bounds in the multivariate central limit theorem for non-stationary processes generated by time-dependent compositions of uniformly expanding dynamical systems. In a particular case of random 
dynamical systems with a strongly mixing base transformation, we derive an error estimate of order 
$O(N^{-1/2})$ in the quenched multivariate CLT, provided that 
the covariance matrix \say{grows linearly} with the number of summands $N$. The error in the normal approximation is estimated 
for the class of all convex sets. 
\end{abstract}

\maketitle




\section{Introduction}

Let $(\xi_n)_{n \ge 1}$ be a sequence of centered real-valued random variables. The central limit theorem (CLT) states that, under suitable conditions on the moments and dependence structure of $(\xi_n)$, the normalized sums $\sigma_N^{-1} S_N$, where 
$S_N = \sum_{n=1}^N \xi_n$ and $\sigma^2_N = \text{Var}(S_N)$, 
converge weakly to the standard normal distribution $\mathcal{N}_1$ as $N \to \infty$. Initially established for independent and identically distributed (i.i.d.) variables, the CLT has been extended to martingales \cite{levy1935}, strongly mixing sequences \cite{rosenblatt1956central}, and chaotic dynamical systems \cite{bunimovic1974central, keller1980}, among other dependent processes.
The accuracy in the approximation of $\sigma_N^{-1} S_N$ by $\cN_1$ is quantified by the
Berry--Esseen theorem \cite{berry1941accuracy, esseen1942liapunov},
which, in the case of i.i.d. variables $\xi_n$ with $\bfE[ |\xi_1|^3 ] < \infty$, asserts that
\begin{align}\label{eq:univ_berry}
	\sup_{ x \in \bR } | \bfP(  \sigma^{-1}_N S_N \le x ) -  \cN_1(  (-\infty, x]  ) | = O(  \bfE( |\xi_1|^3 ) N^{-1/2} ).
\end{align}
Such error bounds have also been extended to various dependent processes. In the case of dynamical systems, early results in this direction include \cite{RE83} for piecewise expanding interval maps and \cite{CP90} for subshifts of finite type.

Multivariate extensions of \eqref{eq:univ_berry} were obtained in the classical works \cite{N76, B86}. Let $W = \sum_{n=1}^N Y^n$, where $Y^n$ are $\mathbb{R}^d$-valued random vectors with $\text{Cov}(W) = \mathbf{I}_{d \times d}$. Set  
$
\beta_3 = \sum_{n=1}^N \mathbf{E}[ \| Y^n \|^3 ] < \infty,
$
where $\Vert x \Vert$ denotes the Euclidean norm of a vector $x \in \mathbb{R}^d$. For i.i.d. summands $Y^n$, Bentkus \cite{B86, bentkus2003dependence} established the estimate  
\begin{align}\label{eq:bentkus}
d_c( \mathcal{L}(W), \mathcal{N}_d ) = O ( d^{1/4}  \beta_3 ) \quad \text{as } N \to \infty,
\end{align}
for the non-smooth metric  
\begin{align}\label{eq:metric}
d_c( \mathcal{L}(W), \mathcal{N}_d ) := 
\sup_{C \in \mathcal{C}} 
| \mathbf{P}(W \in C) - 
\mathcal{N}_d
( C ) |,
\end{align}
where $\mathcal{N}_d$ denotes the $d$-dimensional standard multivariate normal distribution, and $\mathcal{C}$ is the class of all convex subsets of $\mathbb{R}^d$. The result is a natural extension of \eqref{eq:univ_berry} to the multivariate setting. To date, \eqref{eq:bentkus} remains the best known bound in terms of $d$ for general independent variables. 
G\"{o}tze \cite{gotze1991rate} used Stein's method combined with induction to derive  
$
d_c( \mathcal{L}(W), \mathcal{N}_d ) = O ( d \beta_3 ) 
$
for independent (not necessarily identically distributed) summands. More recently, building on the arguments of Bentkus and G\"{o}tze in \cite{bentkus2003dependence, gotze1991rate}, Rai{\v{c}} \cite{raic2019multivariate} established a certain generalization of \eqref{eq:bentkus} in the case of independent summands. 

Beyond the independent case, variants of \eqref{eq:bentkus} were derived for bounded locally dependent 
random vectors by Rinott and Rotar \cite{rinott1996multivariate}, Fang and R\"{o}llin \cite{fang2015rates},
and Fang \cite{fang2016multivariate}, with 
applications to normal approximation for certain graph related statistics. In particular,  \cite{fang2016multivariate} established 
$d_c( \cL(W), \cN_d ) = O( d^{1/4} N \beta^3 )$ in the case of 
decomposable random vectors with $\Vert Y^n \Vert_\infty \le \beta$, where the dependence structure is described in terms of certain dependency neighborhoods.
In this bound the constant grows 
(polynomially) as the \say{size} of the dependency neighborhood increases.

In this work, our main objective is to develop a version of Fang's approach \cite{fang2016multivariate} suitable for 
multivariate normal approximation with respect to non-smooth metrics such as $d_c$,
in the case of processes 
generated by dynamical systems with good mixing properties. We study the problem in a setting of time-dependent dynamical systems
of the form $Y^n = \psi_n \circ T_n \circ \cdots \circ T_1$, where 
$T_n : M \to M$ is a full-branch Gibbs--Markov map of a bounded metric space $M$, and 
$\psi_n : M \to \bR^d$ is a regular function. A distinctive feature of the approach
described here is that it 
allows us to essentially reduce the problem of 
estimating $d_c( \cL(W), \cN_d)$ to a set of correlation decay conditions.

Fourier analytic techniques \cite{CP90, gouezel2004berry, fernando2021edgeworth} and martingale approximations \cite{dedecker2008mean, liu2024wasserstein, antoniou2019rate} have been successfully adapted to obtain univariate Berry--Esseen bounds in the spirit of \eqref{eq:univ_berry}, along with other quantitative refinements of the CLT, for a wide range of 
measure-preserving hyperbolic systems. Extensions of these techniques to time-dependent systems described by compositions 
$T_n \circ \cdots \circ T_1$, where the maps $T_n$ vary either deterministically or randomly, have also been explored in several works, some of which are mentioned in the following.  
In the sequential (nonrandom) setting, the CLT was studied for piecewise uniformly expanding systems in one and higher dimensions in \cite{conze2007, haydn2017, B94-1, B94-2}. A bound similar in spirit to \eqref{eq:univ_berry} was obtained in \cite{heinrich1996mixing} in a self-normed CLT for sequential compositions of piecewise uniformly expanding interval maps, with rate $O(N^{-1/2})$ under the assumption of linear growth of variance. More recent works on error bounds in univariate CLTs for piecewise uniformly expanding and hyperbolic sequential systems include \cite{H20, dolgopyat2024rates, dedecker2022rates, liu2023wasserstein}. In the very recent work \cite{DL25}, Berry--Esseen-type bounds were obtained for sequential dispersing billiards. 
Quenched CLTs were established for random subshifts of finite type and expanding maps in \cite{K98}, and more recently for various random hyperbolic dynamical systems in \cite{dragicevic2018almost, A15, K08, ALS09, NTV18}, among others.
In the recent work \cite{dragicevic2020limit}, quenched Berry--Esseen bounds were derived for a broad class of piecewise uniformly expanding and hyperbolic systems, assuming ergodicity of the base transformation.

Certain correlation-decay criteria for a rate of convergence 
$d_{\cK}( \cL(W),  \cN_d  ) =  O(N^{-1/2})$ in the multivariate CLT
with respect to the Kantorovich (or Wasserstein-$1$) distance 
$$d_{\cK}( \cL(W),  \cN_d  ) = \sup_{ \Vert h \Vert_{ \text{Lip}  \le 1 }  }  | \bfE[  h(W)  ] -  \cN_d[h] |$$
of Lipschitz continuous test functions
were given  in \cite{pene2005rate} based on an approach due to Rio \cite{rio1996sur}.
The result applies to a broad class of hyperbolic measure-preserving
dynamical systems, including Sinai billiards \cite{pene2005rate}, Axiom A diffeomorphisms \cite{xia2011multidimensional}, 
and Pomeau--Manneville type interval maps \cite{leppanen2017functional}.
Very recently, martingale approximations were used in \cite{P24} to prove 
rates of convergence
with respect to $d_{\cK}$ in the multivariate functional CLT for nonuniformly hyperbolic maps and flows.
The present work is partly based on \cite{hella2020stein, leppanen2020sunklodas}, where an adaptation of Stein's method for smooth metrics such as $d_{\cK}$ was developed in a 
dynamical systems setting. We are not aware of any previous error bounds 
on the distance $d_c$ in CLTs for dynamical systems.
We emphasize that, due to the inductive step 
in \cite{fang2016multivariate, raic2019multivariate, gotze1991rate}
that is used to estimate 
\eqref{eq:metric} through Stein's method, the results of this paper are not a direct consequence of \cite{hella2020stein, leppanen2020sunklodas} but require the development of new ideas. To conclude, we mention that in the different dynamical systems setting of Poisson approximation related to hitting time statistics for shrinking sets, 
Stein's method has been implemented in \cite{gordin2012poisson,denker2004poisson, haydn2016entry, haydn2013entry}.

 \subsection*{Organization and notation}
 The paper is organized as follows. In Section \ref{sec:results}, we define the model
 to be studied in the rest of the paper and state our results. 
 We also provide an outline of the strategy used to prove the main 
 result on normal approximation.
 In Section \ref{sec:prelim}, we review preliminaries related to Stein's method, smoothing, 
 and correlation decay properties of the dynamical system under consideration. 
 In Section \ref{sec:proof}, we prove our main result. Appendix \ref{sec:ml_proof} contains 
 proofs of two correlation decay estimates
 stated in Section 
 \ref{sec:prelim}.
 
 Throughout the paper, we denote by 
 $\Vert x \Vert$ the Euclidean norm of a vector $x \in \bR^d$, and by $\Vert A \Vert_s = \sup \{ \Vert Ax \Vert \: : \Vert x \Vert =1  \}$
 the spectral norm 
 of a matrix $A \in \bR^{d \times d}$. 
 For $1 \le r,s \le d$, we write $x_r$ for the $r$th component of $x$, and 
 $[A]_{r,s}$ for the $(r,s)$-entry of $A$.
 Moreover, $\lambda_{\min}(A)$ and 
 $\lambda_{\max}(A)$ denote respectively 
 the minimum and maximum eigenvalue of $A$. For a function $f : X \to \bR$ 
 defined on a measure space $(X, \cB, \mu)$, we write $\mu(f) = \int_X f \,d \mu$.

\section{Setting and statement of main result}\label{sec:results}

\subsection{A time-dependent expanding dynamical system}\label{sec:model}

Let $(M,d)$ be a metric space with $\text{diam}(M) \le 1$. We endow $M$ with its Borel sigma-algebra $\cB$. 
Suppose that $\lambda$ is a probability measure on $\cB$. We denote by $\cM$ the collection of all 
transformations $T : M \to M$ which admit a countable\footnote{In this paper, 
countable means finite or countably infinite.}  
measurable partition $\cA_1(T)$ of $M$, such that for each $a \in \cA_1(T)$, the map $T :a  \to M$ is a measurable bijection.

We consider sequences  $(T_n)$  of maps in $\cM$.
Time-dependent compositions 
along the given sequence are denoted as follows:
\begin{align*}
	\cT_{\ell , k} &:= T_k \circ \cdots \circ T_\ell, \quad 
	\cT_{k} := \cT_{1, k},
\end{align*}
where the convention is that $\cT_{\ell , k} = \id_M$ whenever $k < \ell$.
For each $k, n \ge 1$, define 
$$
\cA( \cT_{k, k + n - 1 } ) = \bigvee_{i=0}^{n-1} \cT_{k, k + i-1}^{-1} \cA_1( T_{k + i} ).
$$
That is, $\cA( \cT_{k, k} ) = \cA_1(T_k)$ and, for $n \ge 2$,  $\cA( \cT_{k, k + n -1} )$ consists of \say{cylinder} sets of the form 
$$
A_k \cap \cT_{k,k}^{-1} A_{k+1} \cap \cdots \cap  \cT_{k, k + n - 2}^{-1} A_{k + n -1}, \quad A_{i} \in \cA_1(  T_{i} ).
$$
For each $1 \le j \le k$, define 
\begin{align*}
	\Lambda_{j, k} = \inf_{a \in \cA( \cT_{j,k} )} \inf_{ \substack{  x,y \in a \\ x \neq y }} \frac{d(  \cT_{j,k} x,  \cT_{j,k} y  ) }{d(x,y)}.
\end{align*}
Given $\psi : M \to \bR$ and $\alpha \in (0,1]$, set 
\begin{align*}
	| \psi  |_\alpha = \sup_{x \neq y} \frac{ |  \psi (x) -  \psi (y) | }{d(x,y)^\alpha}, \quad 
	\Vert u \Vert_\infty = \Vert u \Vert_\infty + 	| u |_\alpha,
\end{align*}
and if $\psi \ge 0$,
\begin{align*}
	|\psi |_{\alpha, \ell} = | \log \psi |_\alpha = \sup_{x \neq y} \frac{ |  \log \psi (x) -  \log \psi (y) | }{d(x,y)^\alpha},
\end{align*}
where we adopt the conventions $\log 0 = - \infty$ and $\log 0 - \log 0 = 0$.

\begin{remark} For any $\psi : M \to \bR_+$, 
	\begin{align}\label{eq:psi_lb_ub}
		e^{ - |  \psi  |_{\alpha , \ell } } \int_M \psi \, d \lambda \le \psi \le 	e^{ |  \psi  |_{\alpha , \ell } } \int_M \psi \, d \lambda.
	\end{align}
	Note that, by the mean value theorem,
	\begin{align*}
		|\psi(x) - \psi(y)| 
		&\le \Vert \psi \Vert_\infty |  \log \psi (x) - \log \psi (y) |
		\le e^{ |  \psi  |_{\alpha , \ell } } \int_M \psi \, d \lambda 
		\cdot  |  \psi  |_{\alpha , \ell } d(x,y)^\alpha,
	\end{align*}
	which gives
	\begin{align}\label{eq:from_lip_to_ll}
		| \psi |_\alpha  \le  | \psi |_{\alpha , \ell} e^{  |  \psi |_{\alpha , \ell} } \int_M \psi \, d \lambda.
	\end{align}
	In the opposite direction, we have that
	\begin{align}\label{eq:from_ll_to_lip}
		| \psi |_{\alpha , \ell } \le (  \inf_M \psi )^{-1}  |  \psi  |_{\alpha }.
	\end{align}
\end{remark}

We assume that the sequential compositions $\cT_{\ell , k}$ are uniformly expanding with bounded distortions
in the following sense:

\noindent\textbf{Assumptions (UE).}
\begin{itemize}[leftmargin=\dimexpr\lmar+\parindent+\labelsep] 
	\item[(UE:1)]  There exist $p \ge 1$ and $\Lambda >1$ such that
	\begin{align*}
		\Lambda_{j, p+j-1}  \ge  \Lambda  \quad \forall j \ge 1.
	\end{align*}	
	\item[(UE:2)] There exists $K' \ge 1$ such that, for all  $j \ge 1$, and all $1 \le \ell < p$,
	\begin{align*}
		d(  x,   y ) \le K' d(  \cT_{j, j + \ell - 1}x, \cT_{j, j + \ell - 1} y  ) \quad \forall x,y \in a, \, \forall a \in \cA( \cT_{j, j + \ell - 1} ).
	\end{align*}
	\item[(UE:3)] There exists $K > 0$ such that 
	\begin{align*}
		\zeta_a^{(j, j + k - 1)} = \frac{ d ( \cT_{j, j + k - 1}  )_* ( \lambda |_a ) }{d  \lambda } \quad \text{satisfies} \quad 
		| \zeta_a^{(j, j + k - 1)}  |_{\alpha , \ell  } \le K.
	\end{align*}
	for all  $a \in \cA(\cT_{j, j + k - 1})$, and all $j,k \ge 1$. \smallskip 
\end{itemize}

Basic examples of maps satisfying (UE:1-3) are given by \say{folklore} piecewise smooth expanding maps of the unit interval.

\begin{exa}[Piecewise expanding interval maps]\label{exa:pw}
	Let $\lambda$ be the Lebesgue measure on $I := [0,1]$.
	Denote by $\cE_{a,B}$ the family of all maps $T: I \to I$ with the following properties:
\begin{itemize}
\item[(i)] There exists a countable (mod $\lambda$) partition $\cA_1(T) = \{I_j\}$ of $I$ into open sub-intervals $I_j$ 
such that $T$ can be extended to a $C^2$ diffeomorphism $T_j : \bar{I}_j \to I$
on the closure $\bar{I}_j$ of each $I_j$; \smallskip 
\item[(ii)] $\sup_{x \in I} | T''(x) | / (T'(x))^2 \le B < \infty$; \smallskip 
\item[(iii)] $\inf_{x \in I} |T'(x)| \ge a > 0$.
\end{itemize} 
	Then, $(T_k)$, $T_k \in \cE_{a,B}$, satisfies (UE:1-3) provided that there exist $p \ge 1$ and $\Lambda > 1$ such that 
	\begin{align}\label{eq:unif_exp}
	\inf_{x \in I} | ( \cT_{j, p + j -1}  )'(x) | \ge \Lambda \quad \forall j \ge 1.
	\end{align}
	Indeed, (UE:1) and (UE:2) are clear by \eqref{eq:unif_exp} and (iii), respectively, and 
	(UE:3) follows by a standard computation. Namely, denoting 
	$\tilde{T}_\ell = T_{j + \ell - 1}$ and 
	$\tilde{\cT}_{\ell, k } = \tilde{T}_{k} \circ \cdots \circ \tilde{T}_{\ell}$, we have	
	\begin{align*}
		\biggl| \frac{ \cT_{j, j + k -1}''(x) }{  ( \cT_{j, j + k -1}'(x) )^2 } \biggr|
		&= \biggl| \sum_{ \ell = 1  }^k \frac{  \tilde{T}_\ell''(  \tilde{\cT}_{1, \ell - 1}(x) )  }{ ( \tilde{T}_\ell' (  \tilde{\cT}_{1, \ell - 1}(x) ) )^2  } \cdot \frac{1}{  \tilde{\cT}_{\ell + 1, k }'(  \tilde{\cT}_{1, \ell } (x)  )  } \biggr|
		\le B \sum_{\ell=1}^k \frac{1}{ | \tilde{\cT}_{\ell + 1, k }'(  \tilde{\cT}_{1,\ell} (x)  )  | } \\
		&\le B \min\{1, a\}^{-p} \Lambda \sum_{\ell = 1}^k \Lambda^{ - ( k - \ell  ) / p } \le C_* < \infty,
	\end{align*}
	where (ii) was used in the second-to-last inequality, and \eqref{eq:unif_exp} and (iii) were used in the last inequality. Letting $x_a \in a$ and $y_a \in a$ denote respectively the unique preimages 
	of $x$ and $y$ under $\cT_{j, j + k - 1}$, $a \in \cA(\cT_{j, j + k - 1})$, it follows that
	\begin{align*}
		&| \log \zeta_a^{(j, j + k - 1)}(x)  - \log \zeta_a^{(j, j + k - 1)}(y)  | \\
		&= | \log ( \cT_{j, j + k - 1} )'(x_a) - \log ( \cT_{j, j + k - 1} )'(y_a) | \\ 
		&\le  \int_{ [x_a,y_a] }  \biggl| 
		\frac{ \cT_{j, j + k -1}''(t) }{  \cT_{j, j + k -1}'(t) }
		\biggr| \, dt \le \int_{ [x,y] }  \biggl \Vert
		\frac{ \cT_{j, j + k -1}'' }{  ( \cT_{j, j + k -1}'   )^2 }
		\biggr \Vert_\infty \, dt  \le C_* |x - y|.
	\end{align*}
\end{exa}

\subsection{Main result} For  $\alpha \in (0,1]$ and $A > 0$, denote by $\cD_{\alpha, A}$ the class of all densities $\rho$ such that $| \rho |_{\alpha, \ell} \le A$. Let $\mu$ be a probability measure whose 
density $\rho$ lies in $\cD_{\alpha, A}$, and let $(\varphi_n)_{n \ge 1}$ be a sequence of functions 
$\varphi_n  : M \to \bR^d$, $d \ge 1$, such that 
\begin{align}\label{eq:varphi_cond}
\mu (\varphi_n \circ \cT_n ) = 0 \quad \text{and} \quad \Vert \varphi_n \Vert_{ \alpha } \le L \quad \forall n \ge 1,
\end{align}
where we assume that $L \ge 1$.
Note that the first of these two properties can be always recovered by centering. Namely, if 
$\psi_n : M \to \bR^d$ satisfies $\Vert \psi_n \Vert_{\alpha} \le L$, then for $\bar{\psi}_n : = \psi_n - \mu( \psi_n \circ \cT_n)$ 
we have that $\mu( \bar{\psi}_n \circ \cT_n ) = 0$ and $\Vert \bar{\psi }_n \Vert_{\alpha} \le 2L$.

For $N \ge 1$ and $0 \le \delta_1 \le \delta_2 \le 1$, define 
\begin{align}\label{eq:S_deltas}
&S_N( \delta_1, \delta_2) = \sum_{ \delta_1 N \le n < \delta_2 N } \varphi_n \circ \cT_n, \quad 
S_N = S_N(0,1).
\end{align}
We consider these quantities as random vectors on the probability space $(M, \cB, \mu)$. Further, we set 
\begin{align}\label{eq:Sigma_deltas}
&\Sigma_N( \delta_1, \delta_2 ) = \mu(  S_N( \delta_1, \delta_2 ) \otimes  S_N( \delta_1, \delta_2 )  ), \quad 
\Sigma_N = \Sigma_N( 0, 1),	
\end{align}
and
\begin{align}\label{eq:W_N}
W_N(\delta_1, \delta_2) = \Sigma_N^{-1/2}(\delta_1, \delta_2)  S_N( \delta_1, \delta_2 )  , \quad 
W_N = W_N(0, 1),
\end{align}
provided that $\Sigma_N(\delta_1, \delta_2)$ is invertible. 

The following theorem, which is our main result, gives an estimate 
on the distance between the law of $W$ and $\cN_d$ in the sense of 
the non-smooth metric
$d_c$ defined in \eqref{eq:metric}. The estimate holds under a condition 
which roughly stipulates that the eigenvalues of 
$\Sigma_N(\delta_1, \delta_2)$ have the same order of growth as  $N \to \infty$.

\begin{thm}\label{thm:main} 	
Let $N \ge 1$, and let	$(T_n)$ be a sequence of transformations satisfying (UE:1-3).
Suppose that the density of $\mu$ belongs to $\cD_{\alpha, A}$,
and that \eqref{eq:varphi_cond} holds. Moreover, suppose that $\Sigma_N$ is invertible, 
and that for some constants $C_0, C_0' \ge 1$ and $K_0 \ge 0$ 
the following conditions hold for all $0 \le \delta_1 \le \delta \le  \delta_2 \le 1$:
\begin{itemize}
	\item[(C1)]  if $| \delta_2 - \delta | \ge | \delta - \delta_1 |$,
	$$\lambda_{ \max }(   \Sigma_N (  \delta_1, \delta_2  )  ) 
	\le  
	\max \{  C_0' | \delta_2 - \delta_1 |^{-K_0}  ,  
	C_0 \lambda_{ \min }(   \Sigma_N (  \delta  , \delta_2  )  )      \};$$
	\item[(C2)] if $| \delta_2 - \delta | < | \delta - \delta_1 |$,
	$$\lambda_{ \max }(   \Sigma_N (  \delta_1, \delta_2  )  ) 
	\le  
	\max \{  C_0' | \delta_2 - \delta_1 |^{-K_0}   ,  
	  C_0 \lambda_{ \min }(   \Sigma_N (  \delta_1 , \delta   )  )      \}.$$ 
\end{itemize}
Then, there exists a constant $\bfC$ whose value is determined by $A, \alpha, K, K', \Lambda$, such that
$$
d_c( \cL( W_N ), \cN_d ) \le    ( d^{13/4} \bfC 2^{3K_0 / 2}
C_0^{3/2}  L^5 + 2 (C_0')^{3/2} )  \max \{ N \lambda_{ \min }^{-3/2} (  \Sigma_N  ), \lambda_{ \min }^{-1/2} (  \Sigma_N  ) \}.
$$
In particular, if $\lambda_{ \min }^{-1} (  \Sigma_N  ) = o(N^{-2/3})$, then $\cL(W_N) \stackrel{D}{\to} \cN_d$ as $N \to \infty$, where $\stackrel{D}{\to}$ 
denotes convergence in distribution.
\end{thm}

\begin{remark} The proof shows a slightly stronger conclusion. Namely that, under the assumptions of Theorem \ref{thm:main},
\begin{align*}
d_c( \cL( W(\delta_1, \delta_2 ) ), \cN_d )
&\le ( \delta_2 - \delta_1 )^{3K_0/2} C_* \max \{ N \lambda_{ \min }^{-3/2} (  \Sigma_N(\delta_1, \delta_2)  ), 
\lambda_{ \min }^{-1} (  \Sigma_N (\delta_1, \delta_2)  )  \}  , \\
C_* &=   d^{13/4}  \bfC  2^{3K_0 / 2}
C_0^{3/2}  L^5 + 2 (C_0')^{3/2}
\end{align*}
holds for all $0 \le \delta_1 \le \delta_2 \le 1$, whenever $\Sigma_N(\delta_1, \delta_2)$ is invertible.
Conditions (C1) and (C2) are related to the inductive method 
used to derive the upper bound on 
$d_c( \cL( W ), \cN_d )$, which involves controlling the ratio 
$\lambda_{ \max }(   \Sigma_N (  \delta_1, \delta_2  )  ) / \lambda_{ \min }(   \Sigma_N (  \delta  , \delta_2  )  )$
for varying $\delta \in [\delta_1, \delta_2]$; see the proof of Lemma \ref{lem:global_estimates} for details. 
The condition is not optimal, but rather a choice of convenience formulated with 
slowly or randomly varying transformations in mind. An application of the latter 
type is given in Theorem \ref{thm:random} below.
\end{remark}

\subsection{Random dynamical systems} We combine \cite[Theorem 4.1]{hella2020quenched}
with 
Theorem \ref{thm:main} to derive an error bound in the quenched multivariate 
central limit theorem  
for 
random expanding dynamical systems with a strongly mixing base transformation. To define the model, 
let $(\Omega_0, \cF_0)$ be a measurable space, and let $\bfP$ be a probability measure 
on the product space $(\Omega, \cF) = ( \Omega_0^{ \bN }, \cF_0^{ \bN } )$, 
where $\bN = \{1,2,\ldots\}$. Expectation with respect to $\bfP$ is denoted by $\bfE$.
We 
assume that the shift transformation $\tau : \Omega \to \Omega$, $(  \tau \omega )_k = \omega_{k+1}$,
preserves $\bfP$, and that, associated to each $\omega \in \Omega$ is a sequence 
of maps $( T_{ \omega_n})$ from the family $\cM$. Given $\omega \in \Omega$, for any $n \ge 1$  we write 
$
\cT_{\omega, n} = T_{\omega_n} \circ \cdots \circ T_{\omega_1}.
$
We then consider a random dynamical system specified by the following assumptions.

\noindent\textbf{Assumptions (RDS).}
\begin{itemize}[leftmargin=\dimexpr\lmar+\parindent+\labelsep] 
\item[(RDS:1)] The map $(\omega, x) \mapsto T_{\omega_n} \circ \cdots \circ T_{\omega_1}(x)$ is measurable between 
$\cF \otimes \cB$ and $\cB$ for any $n \ge 0$.  \smallskip 
\item[(RDS:2)] The random selection process is strongly mixing with rate $O(n^{ - \gamma })$, where $\gamma > 0$. That is, 
for some constant $C > 0$,
$$
\sup_{i \ge 1} \alpha ( \cF_1^i, \cF_{i+n}^\infty  ) \le C n^{- \gamma } \quad \forall n \ge 1,
$$
where $\cF_1^i$ is the sigma-algebra on $\Omega$ generated by the projections $\pi_1, \ldots, \pi_i$, $\pi_k(\omega) = \omega_k$;
$\cF_{i+n}^\infty$ is the sigma-algebra generated by $\pi_{i+n}, \pi_{i + n + 1}, \ldots$; and 
$$
\alpha ( \cF_1^i, \cF_{j}^\infty  ) = \sup_{ A \in \cF_1^i, \, B \in \cF_j^\infty  } |  \bfP(A \cap B)  - \bfP(A) \bfP(B) |.
$$
\item[(RDS:3)] There exist $p \ge 1$ and $\Lambda > 1$, and $K' \ge 1$, such that, for $\bfP$-a.e. $\omega \in \Omega$,
$$
d( \cT_{\omega, p} (x), \cT_{\omega, p} (y)  ) \ge \Lambda d(x,y) \quad \forall x,y \in a, \, \forall a \in \cA( \cT_{p} ),
$$
and if $1 \le \ell < p$, then for $\bfP$-a.e. $\omega \in \Omega$,
\begin{align*}
d(  x,   y ) \le K' d( \cT_{\omega, \ell } ( x), \cT_{\omega, \ell} (x)  ) \quad \forall x,y \in a, \, \forall a \in \cA( \cT_{\ell} ).
\end{align*}
\item[(RDS:4)] There exists $K > 0$ such that for $\bfP$-a.e. $\omega \in \Omega$,
\begin{align*}
\zeta_{\omega, a}^{(k)} = \frac{ d ( \cT_{\omega, k}  )_* ( \lambda |_a ) }{ d \lambda } \quad \text{satisfies} \quad 
| \zeta_{\omega, a}^{(k)}  |_{\alpha, \ell  } \le K,
\end{align*}
whenever $a \in \cA( \cT_k)$.
\end{itemize}

Since $\tau$ preserves $\bfP$, assumptions (RDS:3-4) are equivalent to saying that (UE:1-3) hold for $\bfP$-a.e. $\omega \in \Omega$.

Given a probability measure $\mu$ on $\cB$ together with a function $\varphi : M \to \bR^d$, $d \ge 1$, we set 
$
\varphi_{\omega, n} = \varphi - \mu( \varphi \circ \cT_{\omega, n})
$
for each $n \ge 1$, and define 
$$
S_N(\omega) = \sum_{n=0}^{N-1}  \varphi_{\omega, n} \circ \cT_{\omega, n}, \quad \Sigma_N(\omega) = \mu(  S_N(\omega) \otimes S_N(\omega) ), \quad W_N(\omega) =   \Sigma_N^{-1/2}(\omega) S_N(\omega),
$$
provided that $\Sigma_N(\omega)$ is invertible. 

\begin{thm}\label{thm:random} Consider a random dynamical system satisfying (RDS:1-4).
	Suppose that 
	the density of $\mu$ belongs to $\cD_{\alpha, A}$,  and 
	that $\Vert \varphi \Vert_\alpha < \infty$
	together with the following condition holds.
	\begin{itemize}
		\item[(V)] $\sup_{N \ge 1} \bfE[   v^T \Sigma_N  v  ] = \infty$ for each unit vector $v \in \bR$.
	\end{itemize}
	Then, for $\bfP$-a.e. $\omega \in \Omega$,
	$$
	d_c( \cL( W_N(\omega) ), \cN_d ) = O ( d^{13/4} N^{-1/2}  )  \quad \text{as $N \to \infty$.}
	$$
\end{thm}

\begin{proof} Throughout this proof, $\bfC$ denotes a constant determined by 
$\Lambda$, $p$, $K$, $K'$, $\alpha$, $A$. The value of $\bfC$ is allowed to change from one display to the next. For $k \ge 1$, we write
$$
\cP_{\omega, k} = P_{ \omega_{_k} } \cdots  P_{ \omega_{_1} },
$$
where $P_{\omega_{_i}} : L^1(\lambda) \to L^1(\lambda)$ is the transfer operator associated with $(\lambda, T_{ \omega_{_i}  })$, defined as in \eqref{eq:TO}. For $k \le 0$, we define $\cP_{\omega, k}$ as the identity operator.

First, we verify that $N^{-1} \Sigma_N(\omega)$ converges to a positive definite (nonrandom) limit $\Sigma_\infty$ almost surely with 
a polynomial rate of convergence as $N \to \infty$. To this end, we fix an arbitrary unit vector  $v \in \bR^d$ and define 
the real-valued quantities 
$$
\tilde{\varphi}_{\omega, n} = v^T \varphi_{\omega, n}, \quad \tilde{S}_N(\omega) = \sum_{n=0}^{N-1} \tilde{\varphi}_{\omega, n} \circ \cT_{\omega, n}, \quad 
\tilde{\sigma}_N^2(\omega) = \mu(  \tilde{S}_N^2(\omega) ), \quad \tilde{W}_N(\omega) =   
\tilde{\sigma}_N^{-1}(\omega) \tilde{S}_N(\omega).
$$
Note that $\tilde{\sigma}_N^2(\omega) = v^T \Sigma_N(\omega) v$. We will verify (SA1), (SA3) and (SA5) in 
\cite{hella2020quenched}, (SA2) and (SA4) in the same paper being automatically true by stationarity of $\bfP$ 
and the strong mixing assumption (RDS:2).
	
\noindent\textbf{(SA1):} Denoting $\xi_n = 
\tilde{\varphi}_{\omega, n} \circ \cT_{\omega, n}$, by Corollary \ref{cor:corr_decay} we have the 
upper bound 
\begin{align}\label{eq:corr_xi}
	| \mu( \xi_i \xi_j ) | \le \bfC L^2 q^{|i - j|}
\end{align}
for $\bfP$-a.e. $\omega \in \Omega$, where $q \in (0,1)$ is determined by 
$\Lambda$, $p$, $K$, $K'$, $\alpha$, $A$. Hence, (SA1) in \cite{hella2020quenched} holds with 
$\eta(j) = \bfC L^2 q^j$.

\noindent\textbf{(SA3):}  Let $\rho \in \cD_{\alpha, A}$ denote the density of $\mu$. By Lemma \ref{lem:return_to_D}, 
there exists $\tilde{A} \ge A$ determined by $K$, $K'$, $\alpha$, $A$ such that 
$\cP_{\omega, r}(\rho) \in \cD_{\alpha,  \tilde{A}}$ holds for all $r \ge 0$ and $\bfP$-a.e. $\omega \in \Omega$.
Hence, by Theorem \ref{thm:exp_loss}, for $\bfP$-a.e. $\omega \in \Omega$,
\begin{align*}
	\Vert  \cP_{\omega, k}(  \rho  )  - \cP_{ \tau^{r} \omega , k - r}(\rho )  \Vert_{L^1(\lambda)} 
	= \Vert  \cP_{ \tau^{r} \omega ,k - r } (  \cP_{\omega, r }(\rho) - \rho  ) \Vert_{L^1(\lambda)}  
	\le \bfC q^{k-r}
\end{align*}
holds whenever $k \ge r$. It follows that (SA3) in \cite{hella2020quenched} holds with 
$\eta(j) = \bfC q^j$.

\noindent\textbf{(SA5):} Since $\rho \in \cD_{\alpha, A}$ satisfies $\inf_{x \in M} \rho(x) > 0$ and $\sup_{n \ge 0} \Vert \cP_{\omega, n}(\rho) \Vert_\infty \le \bfC$ holds for $\bfP$-a.e. $\omega \in \Omega$
by \eqref{eq:psi_lb_ub} and Lemma \ref{lem:return_to_D}, we have
\begin{align*}
\biggl\Vert  \frac{d (\cT_{\omega, n})_* \mu  }{d \mu  } \biggr\Vert_{L^2(\mu)}^2 \le ( \inf_{x \in M} \rho(x) )^{-1}
\int_M ( \cP_{\omega, n} ( \rho ) )^2 \, d \lambda \le \bfC
\end{align*}
for $\bfP$-a.e. $\omega \in \Omega$. Moreover, (47) in \cite{hella2020quenched} follows from the memory loss property in Theorem \ref{thm:exp_loss}.
Hence,  (SA5') in \cite{hella2020quenched} is satisfied. 

Having verified Assumptions (SA1-5) in \cite{hella2020quenched}, it now follows by  \cite[Theorem 4.1]{hella2020quenched} 
and \cite[Lemma 4.4]{hella2020quenched} that there exist  nonrandom $\Sigma_\infty \in \bR^{d \times d}$ and 
$\psi > 0$, such that 
for $\bfP$-a.e. $\omega \in \Omega$,
\begin{align}\label{eq:conv_sigma}
\max_{r,s} | N^{-1} [  \Sigma_N(\omega) ]_{ r,s } -    [  \Sigma_\infty ]_{r,s} | = O(N^{-\psi}) \quad \text{as $N \to \infty$}.
\end{align}
Moreover, under condition (V) it follows by  \cite[Lemma B.1]{hella2020quenched} that $\lambda_{\min}(\Sigma_\infty) > 0$. 
In particular, $\lambda_{\min}(\Sigma_N^{-1}(\omega) ) = O(N^{-1} \lambda_{ \min }( \Sigma_\infty)^{-1} )$ as $N \to \infty$, for $\bfP$-a.e. $\omega \in \Omega$.

It remains to verify (C1) and (C2) in Theorem \ref{thm:main}. Fix $0 \le \delta_1 \le \delta \le  \delta_2 \le 1$ 
with $\delta_1 < \delta_2$. In the remainder of the proof, we 
write $\Sigma_N$ for $\Sigma_N(\omega)$
and $S_N$ for $S_N(\omega)$, etc., omitting the dependencies on $\omega$, and define
$S_N(\delta_1, \delta_2)$ and 
$\Sigma_N(\delta_1, \delta_2)$ as 
in \eqref{eq:S_deltas} and
\eqref{eq:Sigma_deltas}, respectively.
Suppose that $| \delta_2 - \delta | < | \delta - \delta_1 |$. Then, for an arbitrary unit 
vector $v \in \bR^d$,
\begin{align}\label{eq:var-1}
	v^T \Sigma_N(\delta_1, \delta) v = v^T \Sigma_N( 0, \delta ) v - v^T \Sigma_N( 0, \delta_1 ) v 
	- 2 v^T\mu(  S_N( \delta_1 , \delta ) \otimes  S_N( 0, \delta_1)   )v.
\end{align}
By \eqref{eq:corr_xi}, for $\bfP$-a.e. $\omega \in \Omega$,
\begin{align}\label{eq:var-2}
	\biggl| v^T\mu(  S_N( \delta_1 , \delta ) \otimes  S_N( 0, \delta_1)   )v \biggr| 
	\le \sum_{ \delta_1 N \le i < \delta N} \sum_{0 \le j < \delta_1 N } |  \mu ( \xi_i\xi_j )  | \le L^2 \bfC.
\end{align}
From \eqref{eq:conv_sigma}, \eqref{eq:var-1}, \eqref{eq:var-2}, 
and $\delta - \delta_1 \ge ( \delta_2 - \delta_1)/2$, 
it follows that for some constant $C > 0$,
\begin{align*}
v^T \Sigma_N(\delta_1, \delta) v  \ge \frac12 N( \delta_2 - \delta_1) v^T \Sigma_\infty v - d^2 C  \max \{ 1, N^{ 1 - \psi  } \}  
- L^2 \bfC,
\end{align*}
so that 
\begin{align*}
	 \lambda_{ \min } ( \Sigma_N (\delta_1, \delta) )  \ge \frac12 N( \delta_2 - \delta_1) \lambda_{\min} ( \Sigma_\infty ) - d^2 C  \max \{ 1, N^{ 1 - \psi  } \}  
	-  L^2 \bfC.
\end{align*}
Similarly, we obtain
\begin{align*}
\lambda_{\max}(  \Sigma_N( \delta_1, \delta_2)  ) \le N( \delta_2 - \delta_1) \lambda_{\max } ( \Sigma_\infty) 
+ d^2 C \max\{ 1 , N^{1 - \psi } \} +  L^2 \bfC.
\end{align*}
Consequently, for some constant $C_1 > 0$, whenever 
$$N \ge (\delta_2 - \delta_1)^{ - 1 /  \max \{\psi,1 \} } (d^2 C_1 / \lambda_{ \min}( \Sigma_\infty )  )^{ 1 /  \max \{\psi,1 \} },$$
we have 
$$
\lambda_{\max} ( \Sigma_N( \delta_1, \delta_2 )  ) \le 4  \frac{\lambda_{\max}( \Sigma_\infty )  }{   \lambda_{ \min}( \Sigma_\infty )     } 
\lambda_{\min}(  \Sigma_N( \delta_1, \delta)  ) \quad \text{for $\bfP$-a.e. $\omega \in \Omega$.}
$$
For $N < (\delta_2 - \delta_1)^{ - 1 /  \max \{\psi,1 \} } (d^2 C_1 /  \lambda_{ \min}( \Sigma_\infty )    )^{  1 /  \max \{\psi,1 \} }$, we have the 
trivial estimate 
$$
\lambda_{\max}(  \Sigma_N( \delta_1, \delta_2)  ) \le \bfC L^2 N \le \bfC L^2  (d^2 C_1 / \lambda_{ \min}(\Sigma_\infty)  )^{  1 /  \max \{\psi,1 \} } 
(\delta_2 - \delta_1)^{ - 1 /  \max \{\psi,1 \} }.
$$
Hence, (C1) follows with $K_0 = 1/ \max\{\psi, 1\}$, $C_0' = \bfC L^2 (d^2 C_1 / \lambda_{ \min}( \Sigma_\infty )   )^{  1 /  \max \{\psi,1 \} }$, and $C_0 = 4  \lambda_{\max}( \Sigma_\infty )  /  \lambda_{\min}( \Sigma_\infty )$. 
The verification of (C2) is almost verbatim the same. The desired estimate now follows 
by Theorem \ref{thm:main}.
\end{proof}

\subsection{Overview of the proof of Theorem \ref{thm:main}}
The proof of Theorem~\ref{thm:main} is guided by the approach of \cite{fang2016multivariate} in the case of locally dependent vectors, but requires suitable modifications since the dynamical process $(\varphi_n \circ \cT_n)$ exhibits a different weak dependence structure, described by a set of correlation decay bounds.
The proof consists of three main steps, outlined below. We emphasize that in Steps~1-2, specific properties of the dynamical system are 
not used,  and hence in these steps $\varphi_n \circ \cT_n$
could be replaced by general random vectors $X^n$.

\noindent\emph{Step 1: Stein's method and smoothing.} The starting point for applications of Stein's method to normal approximation in the multivariate setting is the following characterization of the multivariate standard normal distribution
$\cN_d$ (see \cite[Lemma 2]{CM08} for a precise statement):
a random vector $Y$ on $(M, \cF, \mu)$ is distributed as $\cN_d$ if and only if
$$
\mu[ \Delta f(Y) - Y^T \nabla f(Y) ] = 0  
$$
for all sufficiently smooth $f : \bR^d \to \bR$. This characterization is quantified 
by the following
second order ODE, known as a Stein equation:
\begin{align}\label{eq:se}
h(w) -  \cN_d[h] = \Delta f(w) - w^T \nabla f(w).
\end{align}
Here, $\cN_d[h] := \int_{\bR^d}  h(x)  \phi(x) \, dx$,
where $h : \bR^d \to \bR$ is a given test function, and $\phi(x)$ denotes the density of $\cN_d$.
For any differentiable $h$ with bounded gradient, there exists a solution $f$ to \eqref{eq:se} which is three times differentiable \cite{gaunt2016}. Substituting $w = W_N$, where $W_N$ is defined as in \eqref{eq:W_N}, and taking expectations, we obtain
\begin{align}\label{eq:stein_expectation}
\mu[ h(W_N)  ] - \cN_d[h]   =   \mu[ \Delta f(W_N) - W_N^T \nabla f (W_N)  ].
\end{align}
Thus, for an upper bound on $|\mu[ h(W_N)  ] - \cN_d[h]|$ it suffices to control
\begin{align}\label{eq:control-stein}
|\mu[  \Delta f(W_N) - W_N^T \nabla f (W_N)  ]|.
\end{align}
In the case of the metric $d_c$,
the test functions $h$ are discontinuous, namely indicators $\mathbf{1}_C$ of sets $C \in \cC$, and Taylor expansion cannot be applied to a sufficient degree to control \eqref{eq:control-stein}.
To circumvent this issue, the smoothing technique introduced by Bentkus \cite{bentkus2003dependence} is applied.
In \cite{bentkus2003dependence}, a parametrized family $\{ h_{C, \ve} \}_{\ve > 0}$ of 
smooth approximations of $\mathbf{1}_C$
is constructed, satisfying the properties
$ \Vert \nabla h_{C, \ve}(x) \Vert = O(\ve^{-1})$
and 
\begin{align}\label{eq:smooth_error}
	d_c( \cL(W_N), \cN_d ) \le 4 d^{1/4} \ve + \sup_{C \in \cC} | \mu[ h_{C, \ve}(W_N)  ] - \cN_d[h_{C, \ve}]  |.
\end{align}
Combining \eqref{eq:stein_expectation} and \eqref{eq:smooth_error} yields
$$
d_c( \cL(W_N), \cN_d ) \le 4 d^{1/4} \ve + \sup_{f \in \fF_\ve } |  \mu[  \Delta f(W_N) - W_N^T \nabla f (W_N)  ] |,
$$
where $\fF_\ve$ is the class of all solutions to \eqref{eq:se} for functions $h_{C, \ve}$ with 
$C \in \cC$.

\noindent\emph{Step 2: Decomposition of $\mu[  \Delta f(W_N) - W_N^T \nabla f (W_N)  ]$.} 
For $f \in \fF_\ve$ with $\ve > 0$, we first apply a decomposition from \cite{leppanen2020sunklodas, sunklodas2007}, which is a counterpart of the leave-one-out decomposition (often used in applications of Stein's method in independent settings) adapted to weakly dependent processes. 
A basic observation for obtaining this decomposition is that, since $\mu(\varphi_n \circ \cT^n) = 0$, 
the punctured sums
$$
W^{n,m}_N = \sum_{ \substack{ 0 \le i < N, \\ |  i - n  | > m }  } Y^i \quad \text{with $Y^i = \Sigma_N^{-1/2}  \varphi_i \circ \cT^i$} 
$$
can be used to represent
$\mu[  W_N^T \nabla  f (W_N)  ]$ as the following telescopic sum:
\begin{align*}
	\mu[  W_N^T \nabla  f (W_N)  ] &= \sum_{n = 0}^{N - 1} \mu[ ( Y^n )^T \nabla f (W_N)  ] 
	= \sum_{n = 0}^{N - 1} \mu[ ( Y^n )^T  (  \nabla f (W_N^{ n, - 1 } )   - \nabla f (W^{ n, N -1  }_N  )    )  ] \\
	&= \sum_{n = 0}^{N-1} \sum_{m=0}^{N-1}  \mu[ ( Y^n )^T  (  \nabla f (W_N^{ n,  m - 1 } )   - \nabla f (W^{ n, m }_N  )    )  ].
\end{align*}
Combined with first order Taylor expansion of 
$\nabla f (W^{ n, m }_N  )$, such representations lead to the aforementioned decomposition $\mu[  \Delta f(W_N) - W_N^T \nabla f (W_N)  ]$, given in Lemma \ref{lem:sunklodas},  
consisting of several terms with similar structure. We only discuss one of these terms
($E_5$) here, which is given by 
\begin{align}\label{eq:to_control_intro}
-\sum_{n=0}^{N-1} \sum_{m=1}^{N-1} \sum_{r,s=1}^d  \mu [ \, \overline{
\partial_{rs} 
f(W^{n,m-1}_N) - \partial_{rs} f(W^{n,m}_N)
} \, Y_r^n Y_s^n  ],
\end{align}
where $Y_r$ denotes the $r$th component of a random vector $Y$ and 
$\overline{Y} = Y - \mu(Y)$. Using the correlation decay properties of the system,
discussed in Section \ref{sec:decor}, one can establish
the existence of $q \in (0,1)$ such that 
\begin{align}\label{eq:mixmix}
\mu [ \, \overline{
	\partial_{rs} 
	f(W_N^{n,m-1}) - \partial_{rs} f(W_N^{n,m})
} \, Y_r^n Y_s^n  ] = O( q^m ) \quad \text{as $m \to \infty$}.
\end{align}
However, this does not provide sufficient control over \eqref{eq:to_control_intro},  
since the constant in \eqref{eq:mixmix} diverges as $\ve \to 0$. 
Instead, we use the 
explicit formula for the solution to \eqref{eq:se}, given in \eqref{eq:solution}, which can be 
expressed as 
$f(w) = \int_0^1 g(w, \tau) \, d\tau$
for a function $g(\cdot, \tau) \in C^3$ whose definition involves a Gaussian integral of the test function $h = h_{C, \ve}$ corresponding to the solution $f$.
As in \cite{fang2016multivariate}, we split 
$\partial_{rs} f(w) = \int_{\ve^2}^1 \partial_{rs} g(w; \tau) \, d \tau
+  \int_{0}^{\ve^2} \partial_{rs} g(w; \tau) \, d \tau$.
This leads to a corresponding decomposition of \eqref{eq:to_control_intro} into two parts, which we control separately.
We only discuss the first part, involving $ \int_{\ve^2}^1 \partial_{rs} g(w, \tau) \, d\tau$, as the second part is easier to handle. 
For this part, we take one step further and derive, in Section \ref{sec:Ri_decomp},
a decomposition of 
\begin{align*}
	-\sum_{n=0}^{N-1} \sum_{m=1}^{N-1} \sum_{r,s=1}^d  \mu [ \, \overline{
		\partial_{rs} 
		g(W^{n,m-1}_N; \tau) - \partial_{rs} g(W^{n,m}_N; \tau )
	} \, Y_r^n Y_s^n  ]
\end{align*}
involving third derivatives of $g$. The terms in this decomposition have structure similar to one of the following three forms:
\begin{align*}
I &= \frac{\sqrt{1-\tau} }{2\tau^{3/2}}   \sum_{n=0}^{N-1} 
\sum_{m=1}^{N-1} 
\sum_{ \ell = 2 }^{N - 1}  
\sum_{r,s, t =1}^d  
\int_{\bR^d}
\mu \biggl\{  
\overline{ \eta^{ n, m \ell, m( \ell + 1) }(\tau, z) }  Y_r^n Y_s^n Y_t^{n,m}
\biggr\}   
\partial_{rst} 
\phi(z)  \, dz, \\
II &= \frac{\sqrt{1-\tau} }{2\tau^{3/2}}   \sum_{n=0}^{N-1} 
\sum_{m=1}^{N-1}  
\sum_{r,s, t =1}^d  
\int_{\bR^d}
\mu \biggl\{    \tilde{h}_{\tau, z} (\sqrt{1 - 
	\tau} W_N  - 
\sqrt{\tau} 
z)  \biggr\} \mu( Y_r^n Y_s^n Y_t^{n,m})
\partial_{rst} 
\phi(z)   \, dz, \\
III &= \sum_{n=0}^{N-1}
\sum_{m=1}^{N-1}   
\sum_{r,s, t =1}^d  
\cN_d[ \partial_{rst} g( \cdot, \tau )  ] \mu( Y_r^n Y_s^n Y_t^{n,m} ),
\end{align*}
where 
\begin{align*}
	&Y^{n,m} = \sum_{ \substack{|i - n| = m \\ 
			0 \le i < N } } Y^i, \quad 
	\tilde{h}_{\tau, z}(w) = h(w) - \cN_d [ h(  \sqrt{1 - \tau } \cdot - \sqrt{\tau} z )  ], \\
	&\eta^{n,m,k}(\tau, z) = h(\sqrt{1 - 
			\tau}W^{n,m}_N - 
		\sqrt{\tau}
		z)   - h(\sqrt{1 - 
			\tau}W^{n,k}_N  - 
		\sqrt{\tau}
		z).
\end{align*}

\noindent\emph{Step 3: Induction and decorrelation.} 
In Step 3, we establish estimates necessary to control the terms 
from the decomposition in Step 2.
This involves an induction similar to
\cite{raic2019multivariate, fang2015rates, gotze1991rate},
used to counter the factor $\ve^{-1}$ that appears from 
integrating the terms $I$ and $II$ over the domain $[\ve^2, 1]$.
By an observation from \cite{fang2016multivariate}, 
the quantity $|\cN_d[ \partial_{rst} g( \cdot, \tau ) ]|$ 
is bounded uniformly in $\ve$, which allows estimating
$III$ via a multiple correlation bound.
To estimate $II$, we combine properties of $h_{C, \ve}$ with
Gaussian measure estimates from \cite{bentkus2003dependence}, which yield
\begin{align}\label{eq:indind}
\mu \biggl\{    \tilde{h}_{\tau, z} (\sqrt{1 - 
	\tau} W_N  - 
\sqrt{\tau} 
z)  \biggr\} = O( \ve +  d_c(  \cL( W_N ),  \cL(Z) )  ),
\end{align}
where the constant is independent of $\ve$.
For $I$, we establish in Lemma \ref{lem:decor} decorrelation bounds
that allow us to control
$$
\mu \biggl\{  
\overline{ \eta^{ n, m \ell, m( \ell + 1) }(\tau, z) }  Y_r^n Y_s^n Y_t^{n,m}
\biggr\}
$$
as $m \ell \to \infty$. These bounds depend on both $\ve$ and 
$d_c(  \cL( W_N ),  \cL(Z) )$.
The proof uses a combination
of correlation decay properties of the system, 
conditions (C1)-(C2), and arguments similar to those leading to \eqref{eq:indind},
after partitioning the domain of integration 
into cylinder sets induced by a suitable iterate 
$\cT_j$. We ultimately optimize for $\ve$ to
obtain the desired estimate on $d_c(  \cL( W_N ),  \cL(Z) )$.

\begin{remark} In the proof of Theorem \ref{thm:main}, we use the
exponential memory loss property 
	\begin{align}\label{eq:ml_rate}
		\Vert P_n \cdots P_1 (  \varphi - \psi  ) \Vert_\infty = O( q^n ),
	\end{align}
	where $q \in (0,1)$, $\varphi, \psi \in \cD_{\alpha, A}$, and $P_i$ denotes the transfer operator associated with $\lambda$ and $T_i$. It can be 
	seen from the proof that the exponential rate in \eqref{eq:ml_rate} is not needed, but rather the following polynomial rate would suffice:
	\begin{align*}
		\sum_{n = 1}^\infty n^2 \Vert P_n \cdots P_1 (  \varphi - \psi  ) \Vert_\infty < \infty.
	\end{align*}
	However, as part of the proof, specifically in \eqref{eq:dec_replace_cond}, it is essential that we have control on the decay rate of $P_n \cdots P_1 (  \varphi - \psi  )$ with respect to the strong norm $\Vert \cdot \Vert_\infty$ as opposed to, say, $\Vert \cdot \Vert_{L^1}$. This obstacle has prevented us from extending our results to non-uniformly expanding systems such as interval maps with neutral fixed points, for which polynomial rates of memory loss 
	in $L^1$ have been obtained in 
	\cite{aimino2015,korepanov2021loss}. It would be interesting to explore whether the 
	techniques of \cite{maume2001projective} could be used to address 
	this limitation.
\end{remark}

\section{Preliminaries}\label{sec:prelim} 

\subsection{Decorrelation properties of time-dependent expanding maps}\label{sec:decor} 

In this section, we consider a fixed sequence $(T_n)$ of maps $T_n : M \to M$ satisfying (UE:1-3) in Section \ref{sec:model}. Let $P_n : L^1(\lambda) \to L^1(\lambda)$ 
be the transfer operator associated to $T_n$ and $\lambda$, characterized by the property
\begin{align}\label{eq:TO}
\int_M P_n(  f  ) \cdot g \, d \lambda = \int_M f \cdot g \circ T_n \, d \lambda \quad \forall f \in L^1(\lambda)  \, \forall g \in L^\infty(\lambda).
\end{align}
Time-dependent compositions along the sequence $(P_n)$ will be denoted by
\begin{align}\label{eq:composition_to}
	\cP_{\ell, k} = P_k \cdots P_\ell,  \, \quad \cP_{k} = \cP_{1,k}.
\end{align}

\begin{thm}[Exponential loss of memory]\label{thm:exp_loss} There exist $C_\#$ and $q \in (0,1)$ which 
	depend continuously on $\Lambda$, $p$, $K$, $K'$ and $\alpha$, such that for any $u \in C^\alpha$ with $\lambda ( u) = 0$, 
	and any $i \ge 1$,
	\begin{align*}
		\Vert \cP_{i, n + i  -1} u \Vert_\alpha  \le C_\# q^n | u |_\alpha  \quad \forall n \ge 0.
	\end{align*}
\end{thm}

\begin{proof} 
The result follows from the explicit coupling argument of Korepanov, Kosloff, and Melbourne	
\cite[Section 3]{korepanov2019explicit}; further details are
provided in Appendix \ref{sec:ml_proof} for completeness.
Alternatively, a similar bound can be deduced from the general result in~
\cite[Theorem 2.4]{dolgopyat2024rates}.
\end{proof}

In the proof of Theorem \ref{thm:main}, the memory loss property of Theorem \ref{thm:exp_loss} will be applied 
after conditioning a measure on elements of $\cA( \cT_n )$. To prepare for this step, we make the 
following simple observation:

\begin{cor}\label{cor:ml_cond} Let $\mu$ be a probability measure with 
	density $\psi \in \cD_{\alpha, A}$. 
	For $j, m \ge 1$
	and $a \in \cA( \cT_{j, j + m -1} )$, define $\psi_a = \mu( a )^{-1}  \psi \mathbf{1}_a$, provided that $\mu(a) \neq 0$. Then, for 
	any $n \ge 0$, 
	\begin{align*}
		\Vert \cP_{j, j + m + n -1} ( \psi - \psi_a ) \Vert_\alpha  \le  2 ( K +   A (K')^\alpha ) e^{  ( K +   A (K')^\alpha )  } C_\# q^n,
	\end{align*}
	where $C_\#$ and $q$ are as in Theorem \ref{thm:exp_loss}.
\end{cor}

\begin{proof}
	By Lemma \ref{lem:to_ub_1}, 
	for $\varphi \in \{   \psi, \psi_a \}$,
	\begin{align*}
		| \cP_{j, j + m - 1} ( \varphi ) |_{\alpha , \ell} 
		\le K +   | \psi  |_{\alpha, \ell} (K')^\alpha.
	\end{align*}
	Hence, by \eqref{eq:from_lip_to_ll}
	\begin{align*}
		|  \varphi    |_\alpha  \le  ( K +   | \psi  |_{\alpha, \ell} (K')^\alpha ) e^{  ( K +   | \psi  |_{\alpha, \ell} (K')^\alpha )  }
	\end{align*}
	The desired estimate now follows from Theorem \ref{thm:exp_loss}.
\end{proof}

Another easy consequence of Theorem \ref{thm:exp_loss} is the exponential 
decay of second and third order correlations:

\begin{cor}\label{cor:corr_decay} Let $\psi_i \in C^{\alpha}$, $1 \le i \le 3$, be functions
	$\psi_i : M \to \bR$ and, for each $n \ge 1$, define
	$$
	\bar{\psi}_i^n = \psi_i^n - \mu( \psi_i^n ), \quad 
	\psi_i^n = \psi_i \circ \cT_n.
	$$
	Then, there exists a constant $\bfC > 0$ depending only on $\Lambda$, $p$, $K$, $K'$, $\alpha$, $A$, such that the following hold for any $n,m,k \ge 0$:
	\begin{align}\label{eq:correlations}
		\begin{split}
			| \mu(   \bar{\psi}_1^n  \bar{\psi}_2^{n+m}    ) | &\le \bfC \Vert \psi_1 \Vert_\alpha 
			\Vert \psi_2 \Vert_\alpha q^m, \\
			| \mu(   \bar{\psi}_1^n  \bar{\psi}_2^{n+m}  \bar{\psi}_3^{n+m + k}   ) | &\le \bfC \Vert \psi_1 \Vert_\alpha 
			\Vert \psi_2 \Vert_\alpha  \Vert \psi_3 \Vert_\alpha   q^{  \max\{  m, k \}  },
		\end{split}
	\end{align}
	where $q \in (0,1)$ is as in Theorem \ref{thm:exp_loss}.
	
\end{cor}

\begin{proof} Both upper bounds in \eqref{eq:correlations} follow by standard 
	arguments using Theorem \ref{thm:exp_loss} together with properties of $P_n$ and functions in $\cD_{\alpha, A}$. We provide more details in Appendix \ref{sec:proof_decor}.
\end{proof}

\subsection{Stein's method and smoothing} 
In this section, we review some preliminary definitions and results that are essential 
for deriving Berry--Esseen type bounds through Stein's method in the spirit of \cite{gotze1991rate, fang2016multivariate, raic2019multivariate}. Our presentation follows \cite{fang2016multivariate},
and we refer the reader to \cite{fang2016multivariate, raic2019multivariate} and references therein for further details.

For a differentiable test function $h : \mathbb{R} \to \mathbb{R}$ 
with bounded gradient, we 
consider the following Stein equation for the 
$d$-dimensional standard normal distribution:
\begin{align}\label{eq:stein_eq}
	\Delta f(w) - w^T \nabla f(w)
	= h(w) - \bfE[ h(Z)],
\end{align}
where $\Delta$ denotes the Laplacian, 
and $Z \sim \cN_d$.

By a direct computation (see e.g. \cite{gotze1991rate})
it can be verified that, defining 
\begin{align}\label{eq:g}
	g(w, \tau) = 
	-\frac{1}{2 ( 1 - \tau)}
	\bfE [ h( \sqrt{1 - \tau}w - \sqrt{\tau}
	Z )  - h(Z) ],
\end{align}
the function
\begin{align}\label{eq:solution}
	f_h(w) = \int_0^1 g(w, \tau) \, d\tau
\end{align}
is a solution to \eqref{eq:stein_eq}. Given a function 
$f : \mathbb{R}^d \to \mathbb{R}$, for brevity, 
we write $f_r(x)$ for the first order partial derivative 
$\partial f(x) / \partial x_r$, 
$f_{rs}(x)$ for the second order partial derivative 
$\partial^2 f(x) / \partial x_r \partial x_s$, and so on.
We denote by  $\phi$ the density of $Z$. 
Then, 
the following relations can be  verified 
using integration by parts:
\begin{align}\label{eq:diffs}
	\begin{split}
	g_{rs}(w, \tau) &=  - \frac{1}{2\tau} 
	\int_{\bR^d} h( \sqrt{1 - \tau}w - \sqrt{\tau}
	z ) \phi_{rs}(z) \, dz \\
	&=  \frac{1}{2\sqrt{\tau}} 
	\int_{\bR^d} h_s( \sqrt{1 - \tau}w - \sqrt{\tau}
	z ) \phi_{r}(z) \, dz, \\ 
	g_{rst}(w, \tau) &=   
	\frac{\sqrt{1-\tau}}{2\tau^{3/2}} 
	\int_{\bR^d} h( \sqrt{1 - \tau}w - \sqrt{\tau}
	z ) \phi_{rst}(z) \, dz.
	\end{split}
\end{align}
As observed in \cite{fang2016multivariate}, starting from \eqref{eq:g}
and using a change of variables together with basic properties of 
normal distributions, 
one obtains the following for any $w \in \bR^d$:
\begin{align}\label{eq:gaussian_sol}
\begin{split}
	\cN_d[ g( \cdot + w , \tau) ] &= 	-\frac{1}{2 ( 1 - \tau)} \int_{\bR^d}
	\bfE [ h( \sqrt{1 - \tau}(z + w) - \sqrt{\tau}
	Z )  - h(Z) ] \phi(z) \, dz \\
	&= 	-\frac{1}{2 ( 1 - \tau)} \int_{\bR^d}
	\bfE [ h( \sqrt{1 - \tau}(Z + w) - \sqrt{\tau}
	z )  - h(Z) ] \phi(z) \, dz \\
	&= -\frac{1}{2 ( 1 - \tau)} \int_{\bR^d}
	 h(  \sqrt{1 - \tau } w  + z ) \phi(z) \, dz + \frac{1}{2 ( 1 - \tau )} \bfE[ h(Z) ] \\
	&= 	-\frac{1}{2 ( 1 - \tau)} \int_{\bR^d}
	h( x )  \phi( x - \sqrt{1 - \tau } w) \, dx + \frac{1}{2 ( 1 - \tau )} \bfE[ h(Z) ],
\end{split}
\end{align}
Differentiating \eqref{eq:gaussian_sol} with respect to $w_r, w_s, w_t$ and evaluating at $w = 0$, it follows that
\begin{align}\label{eq:diff_gN}
\cN_d[  g_{rst}( \cdot, \tau ) ] = \frac{ \sqrt{1 - \tau } }{ 2 } \int_{\bR^d} h(x) \phi_{rst}(x) \, dx.
\end{align}
In particular, since $|h(x)| \le 1$, we see that 
$|\cN_d[  g_{rst}( \cdot, \tau ) ]|$ is bounded by an absolute constant.

If the test function 
$h$ in \eqref{eq:stein_eq} is not 
smooth, as in the case of the metric $d_c$,
then the regularity of $f_h$ will not be 
sufficient in order to control the left hand 
side of \eqref{eq:stein_eq} via Taylor expansion. 
For this 
reason, smoothing will be applied to the 
indicator $h = \mathbf{1}_A$ following Bentkus \cite{bentkus2003dependence}.

For each $\ve > 0$ and $C \in \cC$, define 
\begin{align*}
	h_{C, \ve}(x) = \psi\biggl( 
	\frac{ \text{dist}(x, C) }{\ve}
	\biggr),
\end{align*}
where 
\begin{align*}
	\psi(x) = \begin{cases}
		1, & x < 0 \\
		1 - 2x^2, & 0 \le x < \frac12, \\ 
		2(1 - x)^2 & \frac12 \le x < 1, \\ 
		0, & x \ge 1.
	\end{cases}
\end{align*}
For any $\ve > 0$ and $C \in \cC = \{ C \subset \bR^d \: : \: 
\text{$C$ convex} \}$, let  
\begin{align*}
	C^\ve = \{ x \in \mathbb{R}^d : 
	\text{dist}(x, C) \le \ve \}  \quad
	\text{and} \quad
	C^{-\ve} = \{ x \in \mathbb{R}^d : 
	\text{dist}(x, \mathbb{R}^d \setminus C) > \ve \}.
\end{align*}

\begin{lem}[Lemma 2.3 in \cite{bentkus2003dependence}]\label{lem:approx} For each $\ve > 0$ and 
	$C \in \cC$, $h = h_{C, \ve}$ satisfies the following properties:
	\begin{itemize}
		\item[(i)] {$h(x) = 1$ for all 
			$x \in C$,} \smallskip 
		\item[(ii)] {$h(x) = 0$ for all 
			$x \in \mathbb{R}^d \setminus C^\ve$, } \smallskip 
	\item[(iii)] {
		$0 \le h(x) \le 1$ for all $x \in \bR^d$,
 } \smallskip 
\item[(iv)]{
	$ \Vert \nabla h(x) \Vert \le 2 \ve^{-1}$ for all $x \in \bR^d$,
} \smallskip
\item[(v)]{
	$ \Vert \nabla h(x) - \nabla h(y) \Vert  
	\le 8 \ve^{-2} \Vert x - y \Vert $ for all $x,y \in \bR^d$.
}
\end{itemize} 
\end{lem}

Building on Fang's approach \cite{fang2016multivariate}, we will employ the following two results 
to control error terms that arise when passing from indicators $\mathbf{1}_C$ of sets $C \in \cC$ to 
their smooth approximations $h_{C, \ve}$:

\begin{lem}[See \cite{ball1993reverse, bentkus2003dependence}]\label{lem:normal_approx} 
For any $\ve > 0$, 
\begin{align*}
\sup_{C \in \cC} \max 
\{ \cN_d ( C^{\ve} \setminus C ),
\cN_d (  C \setminus C^{-\ve} )  \} 
\le 4 d^{\frac14} \ve.
\end{align*}
\end{lem}

\begin{lem}[See \cite{fang2015rates}]\label{lem:reduce} Let $Y$ be an arbitrary 
$\bR^d$-valued random vector. For 
any $\ve > 0$, 
\begin{align*}
d_c(  \cL(Y), \cN_d )
\le 
4 d^{\frac14} \ve
+ \sup_{C \in \cC}
| \cN_d (h_{C,\ve})
-
\bfE ( h_{C,\ve}(Y )) |,
\end{align*}
where $h_{C,\ve}$ is the function from 
Lemma \ref{lem:approx}.
\end{lem}

\subsection{A decomposition of $\mu [\Delta f(W) - W^T \nabla f(W)]$}\label{sec:decomp_sunklodas}

For $N \ge 1$ and a sequence of bounded $\bR^d$-valued random vectors $X^n$ on $(M, \cB, \mu)$ with 
$\mu(X^n) = 0$, define 
\begin{align*}
	S  = \sum_{n=0}^{N-1} X^n \quad \text{and} \quad  
	W = \sum_{n=0}^{N-1} Y^n, \quad 
\end{align*}
where $Y^n = \Sigma^{-1/2} X^n$ and the covariance matrix $\Sigma := \mu( S \otimes S)$ is assumed to 
be invertible. By Lemma \ref{lem:reduce}, 
for any $\ve > 0$,
\begin{align}
&d_c(  \cL(W), \cN_d ) 
\le 
4 d^{\frac14} \ve
+ \sup_{C \in \cC}
| 
\mu ( h_{C,\ve}(W ))  - \cN_d (h_{C,\ve}) |. \label{eq:reduction}
\end{align}
For $h = h_{C, \ve}$ with $C \in \cC$ and $\ve > 0$, we have 
\begin{align}\label{eq:solve_stein}
\cN_d (h)  - 
\mu ( h(W ) ) = 
\mu[ W^T \nabla f_h(W) - \Delta f_h(W) ],
\end{align}
where $f_h$ 
is given by
\eqref{eq:solution}. 
Hence, 
\begin{align}\label{eq:stein_applied}
d_c(  \cL(W), \cN_d ) 
\le 4 d^{\frac14} \ve
+ \sup_{ f \in \mathfrak{F}_\ve }
| \bfE[  \Delta f(W) - W^T \nabla f(W) ] |,
\end{align}
where 
$\mathfrak{F}_\ve = \{ f \: : \: 
f = f_h, \, 
h = h_{A, \ve}, \, C \in \cC \}$. That is, for a bound on $d_c(  \cL(W), \cN_d )$ it suffices to control the right-hand side of 
\eqref{eq:stein_applied}. This task is facilitated by a decomposition from \cite{leppanen2020sunklodas, sunklodas2007}, 
which will be recorded in the lemma below for the reader's convenience.

For $n,m \in \bZ$, define the auxiliary random vectors 
\begin{align*}
&W^{n,m} = W - \sum_{i\in[n]_m} Y^i, \quad [n]_m = \{0\le i < N : |i-n|\le m\}, \\ 
&Y^{n,m} = W^{n,m-1} - W^{n,m} 
= \sum_{ \substack{|i - n| = m \\ 
	0 \le i < N } } Y^i,
\end{align*}
and set $\overline{X} = X - \mu(X)$ for a random vector $X$ on $(M, \mu, \cB)$.

\begin{lem}(Proposition 5.3 in \cite{leppanen2020sunklodas})\label{lem:sunklodas}
Suppose $f 
\in C^2(\bR^d, \bR)$. 
Denote
\beqn
\delta^{n,m}(u) = D^2
f(W^{n,m} + u\, Y^{n,m}) - 
D^2f(W^{n,m})
\eeqn
and
\beqn
\delta^{n,m} = 
\delta^{n,m}(1) = D^2 
f(W^{n,m-1}) - D^2f(W^{n,m}).
\eeqn
Then,
$
\mu[ \Delta f(W) - W^T 
\nabla f(W)] = \sum_{i=1}^7 E_i$, 
where $E_i = E_i(f)$ are defined as follows:
\begin{align*}
E_1 &= - \int_0^1 \sum_{n=0}^{N-1}\sum_{m=1}^{N-1}  \mu[  ( Y^{n} )^T \delta^{n,m}(u) Y^{n,m}   ] \, 
d u , \quad 
E_2  = -  \int_0^1 \sum_{n=0}^{N-1}   \mu[ (Y^{n})^T \delta^{n,0}(u) Y^n ] \, du , 
\\
E_3 &= - \sum_{n=0}^{N-1}\sum_{m=1}^{N-1} \sum_{k=m+1}^{2m}    \mu [  (Y^{n})^T \, \overline{\delta^{n,k}} \, Y^{n,m}  ], \quad 
E_4  = - \sum_{n=0}^{N-1}\sum_{m=1}^{N-1} \sum_{k=2m+1}^{N-1}     \mu [  (Y^{n})^T \, \overline{\delta^{n,k}} \, Y^{n,m}  ],
\\
E_5 & = - \sum_{n=0}^{N-1} \sum_{k=1}^{N-1}  \mu [  (Y^{n})^T \, \overline{\delta^{n,k}} \, Y^n  ], \quad 
E_6 = \sum_{n=0}^{N-1}\sum_{m=1}^{N-1}  \mu \left[  (Y^{n})^T   \sum_{k=0}^m \mu( \delta^{n,k} )  Y^{n,m}  \right],
\\
E_7 & =  \sum_{n=0}^{N-1}  \mu [ ( Y^{n} )^T   \mu( \delta^{n,0} )  Y^n  ].
\end{align*}
\end{lem}

The following preliminary estimate 
is an immediate consequence 
of \eqref{eq:reduction}, \eqref{eq:solve_stein}, and Lemma \ref{lem:sunklodas}:

\begin{lem}\label{lem:prelim_bound} For any $\ve > 0$, 
\begin{align*}
d_c( \cL(W), \cN_d )
\le 4 d^{\frac14} \ve
+ \sup_{ f \in \mathfrak{F}_\ve }
\sum_{i=1}^7 |E_i(f)|,
\end{align*}
where $E_i(f)$ are as in Lemma 
\ref{lem:sunklodas}.
\end{lem}

\section{Proof of Theorem \ref{thm:main}}\label{sec:proof}

Let $N \ge 1$, $(T_n)$, $\mu$, and $(\varphi_n)$ be as in Theorem \ref{thm:main}, and assume that
(C1) and (C2) in the same theorem hold. By Lemma \ref{lem:return_to_D}, there exists $\tilde{A}$ depending only on $A, K, K', \alpha$ such that 
\begin{align}\label{eq:return_to_cone_2}
	 \cP_{j, j + m - 1 }( \cD_{\alpha, A } ) \subset \cD_{\alpha,  \tilde{A} } \quad \forall j \ge 1, \: \forall 
	m \ge 0.
\end{align}

Throughout the proof, we denote by $\bfC \ge 1$ a system constant whose value is determined by  
$\Lambda$, $p$, $K$, $K'$, $\alpha$, $A$. The value of $\bfC$ is allowed to change from one display to the next.
Given two functions $f : S \to \bR$ and  $g : S \to \bR_{+}$ defined on a set $S$, we write
$f(x) \lesssim g(x)$
if there exists an absolute constant $C > 0$ such that $f(x) \le C g(x)$ for all $x \in S$.
Moreover, we write $f(x) = O(g(x))$ if $|f(x)| \lesssim  g(x) $. Finally, we denote by $w_r$
the $r$th
component of a vector $w \in \bR^d$ ($r = 1,\ldots, d$).

\subsection{Induction} Recall that, for $0 \le \delta_1 \le \delta_2 \le 1$,
\begin{align}
\begin{split}\label{eq:notations-1}
W_N(\delta_1, \delta_2) &= \Sigma_N^{-1/2}(\delta_1, \delta_2)  S_N( \delta_1, \delta_2 ), \quad 
\Sigma_N( \delta_1, \delta_2 ) = \mu(  S_N( \delta_1, \delta_2 ) \otimes  S_N( \delta_1, \delta_2 )  ), \\ 
S_N( \delta_1, \delta_2) &= \sum_{ \delta_1 N \le n < \delta_2 N } \varphi_n \circ \cT_n.
\end{split}
\end{align}
To prepare for the inductive argument in \cite{fang2016multivariate, gotze1991rate, raic2019multivariate}, define 
\begin{align}\label{eq:iso_D}
	\fD = \fD( N ) = \sup \,  ( \delta_2 - \delta_1 )^{3 K_0 / 2}  \:   \frac{ d_c( \cL(  W_N(\delta_1, \delta_2) , 
		\cN_d   )  }{   \max \{  N  \Vert \Sigma^{-1/2}_N(\delta_1, \delta_2)  \Vert_s^3, 
		 \Vert \Sigma^{-1/2}_N(\delta_1, \delta_2)  \Vert_s  \}   },
\end{align}
where the supremum is taken over all $0 \le \delta_1 \le  \delta_2 \le 1$ such that 
$ \lambda_{ \min } ( \Sigma_N(  \delta_1, \delta_2 ) ) > 0$. Note that $\fD(N) < \infty$ because 
there are only finitely many terms 
$$
\frac{ d_c( \cL(  W_N(\delta_1, \delta_2) , 
	\cN_d   )  }{   \max \{  N  \Vert \Sigma^{-1/2}_N(\delta_1, \delta_2)  \Vert_s^3, 
	\Vert \Sigma^{-1/2}_N(\delta_1, \delta_2)  \Vert_s  \}   }
$$
included in \eqref{eq:iso_D}, one of them being 
$ d_c( \cL(  W_N ) , \cN_d   )  /   \max \{  N  \Vert \Sigma^{-1/2}_N \Vert_s^3, 
\Vert \Sigma^{-1/2}_N  \Vert_s  \}$ by our assumption. The aim is to derive an upper bound on 
$\fD(N)$ independent of $N$. To this end, 
we fix $0 \le \delta_1  <  \delta_2 \le 1$ such that 
$ \lambda_{ \min } ( \Sigma_N(  \delta_1, \delta_2 ) ) > 0$. In the sequel, we shall use the following notation 
for convenience:
\begin{align}\label{eq:notations}
	\begin{split}
	&X^n = \varphi_n \circ \cT_n, \quad Y^n = \mathbf{1}_{ \delta_1 N \le n < \delta_2 N  } b^{-1}_N(\delta_1, \delta_2) X^n , \quad 
	b_N(\delta_1, \delta_2) = \Sigma^{1/2}_N(\delta_1, \delta_2), \\ 
	&\bar{b} = \max \{  N  \Vert \Sigma^{-1/2}_N(\delta_1, \delta_2)  \Vert_s^3, 
	\Vert \Sigma^{-1/2}_N(\delta_1, \delta_2)  \Vert_s  \}.
	\end{split}
\end{align}
Note that $W_N(\delta_1, \delta_2) = \sum_{n=0}^{N-1} Y^n$. For brevity, we will omit the 
dependencies on $N$ and $\delta_1, \delta_2$ from our notation, writing 
$W$ in place of $W_N(\delta_1, \delta_2)$, $\Sigma$ in place of $\Sigma_N(\delta_1, \delta_2)$, etc.

First, suppose that $\lambda_{ \max }(   \Sigma  ) \le C_0' | \delta_2 - \delta_1 |^{-K_0}$. Then, 
we have the trivial estimate 
\begin{align}\label{eq:trivial}
	&d_c( \cL(  W), 
	\cN_d   ) \le (C_0')^{3/2} | \delta_2 - \delta_1 |^{ - 3 K_0 / 2 } N \Vert b^{-1} \Vert_s^3.
\end{align}
From now on, we assume that 
$$
\lambda_{ \max }(   \Sigma_N (  \delta_1, \delta_2  )  ) > C_0' ( \delta_2 - \delta_1 )^{-K_0}.
$$
By (C1) and (C2), we then have that 
\begin{align}\label{eq:cc}
	\begin{split}
	&\lambda_{ \max }(   \Sigma_N (  \delta_1, \delta_2  )  ) 
	\le C_0 
	\lambda_{ \min }(   \Sigma_N (  \delta  , \delta_2  )  )   \: \text{ if $| \delta_2 - \delta | \ge | \delta - \delta_1 |$},  \\
	&\lambda_{ \max }(   \Sigma_N (  \delta_1, \delta_2  )  ) 
	\le C_0 
	\lambda_{ \min }(   \Sigma_N (  \delta_1  , \delta  )  )   \: \text{ if $| \delta_2 - \delta | < | \delta - \delta_1 |$}.
	\end{split}
\end{align}
We will derive an upper bound on $d_c( \cL(  W), 
\cN_d   ) $ by controlling each term $E_i$ in Lemma \ref{lem:sunklodas}. This step is technical and occupies 
the remainder of the proof.

\subsection{Decomposition of $E_i$}\label{sec:E1}  Let 
$\ve \in (0, 1)$ and $f = f_h \in \mathfrak{F}_\ve$, 
where $h = h_{C,\ve}$ for some 
$C \in \cC$. Drawing inspiration from \cite{fang2016multivariate}, we start  
by decomposing $E_1 + E_2$ in Lemma \ref{lem:sunklodas} as follows:
\begin{align*}
	E_1 + E_2 &= -  \int_0^1 \sum_{n=0}^{N-1}
	\sum_{m=1}^{N-1}   
	\, \mu[  (Y^{n})^T \delta^{n,m}(u) Y^{n,m}  
	]   \, du -  \int_0^1 \sum_{n=0}^{N-1}   \mu[ (Y^{n})^T \delta^{n,0}(u) Y^n ] \, du \\ 
	&= -  \int_0^1 \sum_{n=0}^{N-1} \sum_{m=1}^{N-1}
	\sum_{r,s=1}^d \mu  \biggl\{   ( 
	f_{rs}(W^{n,m} + u 
	Y^{n,m}) - f_{rs}(W^{n,m})
	) 
	Y^n_r  Y^{n,m}_s \biggr\}  \, du\\ 
	&-  \int_0^1 \sum_{n=0}^{N-1}
	\sum_{r,s=1}^d \mu  \biggl\{   ( 
	f_{rs}(W^{n,0} + u  
	Y^{n,0}) - f_{rs}(W^{n,0})
	) 
	Y^n_r  Y^{n}_s \biggr\}  \, du.
	\end{align*}
	Recalling \eqref{eq:solution}, we thus have 
	\begin{align*}
	E_1 + E_2 &= -  \int_0^1 \sum_{n=0}^{N-1} \sum_{m=1}^{N-1}
	\sum_{r,s=1}^d \int_0^1
	\mu \biggl\{   ( 
	g_{rs}(W^{n,m} + u 
	Y^{n,m}, \tau) - 
	g_{rs}(W^{n,m}, \tau)
	) 
	Y^n_r  Y^{n,m}_s \biggr\} \, d\tau  \, du\\ 
	&-  \int_0^1 \sum_{n=0}^{N-1}
	\sum_{r,s=1}^d \int_0^1
	\mu \biggl\{   ( 
	g_{rs}(W^{n,0} + u 
	Y^n, \tau) - 
	g_{rs}(W^{n,0}, \tau)
	) 
	Y^n_r  Y^{n}_s \biggr\} \, d\tau \, du  \\ 
	&=  \int_0^1 \biggl[   \int_{0}^{ \ve^2 } R_{1}(\tau,u)  \, d \tau  +  \int_{\ve^2}^{ 1 } R_{1}(\tau,u)  \, d \tau 
	+  \int_{0}^{ \ve^2 } R_{2}(\tau,u)  \, d \tau  +  \int_{\ve^2}^{ 1 } R_{2}(\tau,u)  \, d \tau \biggr] \, d u,
\end{align*}
where 
\begin{align*}
R_1 = R_1(\tau, u) &= - \sum_{n=0}^{N-1} 
\sum_{m=1}^{N-1}  \sum_{r,s=1}^d \mu \biggl\{    \gamma^{n,m}_{r,s}(u, \tau)
Y^n_r  Y^{n,m}_s \biggr\},  \\
R_2 = R_2(\tau, u) &= - \sum_{n=0}^{N-1}   \sum_{r,s=1}^d  \mu \biggl\{    \gamma^{n,0}_{r,s}(u, \tau)
Y^n_r  Y^{n}_s \biggr\}, 
\end{align*}
and
\begin{align*}
	\gamma^{n,k}_{r,s}(u, \tau) = g_{rs}(W^{n,k} + u 
	Y^{n,k}, \tau) - 
	g_{rs}(W^{n,k}, \tau).
\end{align*}
Similarly, for $3 \le i \le 7$, we decompose $E_i = \int_0^{\ve^2} R_i(\tau) \, d \tau + \int_{\ve^2}^1 R_i(\tau) \, d \tau$, 
where 
\begin{align*}
	R_3 &=  - \sum_{n=0}^{N-1}\sum_{m=1}^{N-1} \sum_{k=m+1}^{2m} \sum_{r,s=1}^d  \mu \biggl\{  
	\overline{ \gamma^{n,k}_{r,s}(1, \tau) } 
	Y^n_r  Y^{n,m}_s \biggr\}, \\ 
		R_4 &=  - \sum_{n=0}^{N-1}\sum_{m=1}^{N-1} \sum_{k=2m+1}^{N-1} \sum_{r,s=1}^d  	\mu \biggl\{  
	\overline{ \gamma^{n,k}_{r,s}(1, \tau) } 
	Y^n_r  Y^{n,m}_s
	\biggr\}, \quad 
	R_5 =  - \sum_{n=0}^{N-1} \sum_{m=1}^{N-1} \sum_{r,s=1}^d  \mu \biggl\{  
	\overline{ \gamma^{n,m}_{r,s}(1, \tau) } 
	Y^n_r  Y^{n}_s
	\biggr\}, \\
	 R_{6} &=  \sum_{n=0}^{N-1} \sum_{m=1}^{N-1}  \sum_{k=0}^m  \sum_{r,s=1}^d  
	 \mu ( \gamma^{n,k}_{r,s}(1, \tau)  )  
	 \mu( Y^n_r  Y^{n,m}_s ), \quad 
	 R_{7} =  \sum_{n=0}^{N-1}   \sum_{r,s=1}^d  
	 \mu ( \gamma^{n,0}_{r,s}(1, \tau)  ) 
	 \mu( Y^n_r  Y^{n}_s ).
\end{align*}

\subsection{Decomposition of $R_i$}\label{sec:Ri_decomp} We will derive a decomposition of
each $R_i$, $1 \le i \le 7$, to facilitate controlling $\int_{\ve^2}^1 R_i(\tau) \, d \tau$ using 
Lemma \ref{lem:normal_approx} together with 
the decorrelation properties stated in Section \ref{sec:decor}.
For convenience and brevity, we introduce the following notation:
\begin{align*}
&F_{r,s,t}^{n,m,k} = Y^n_r
Y^{n,m}_s  Y^{n,k}_t, \quad G_{r,s}^{n,m} = Y_r^n Y_s^{n,m}, \\ 
&\eta^{n,m,k}(v, \tau, z) = h(\sqrt{1 - 
	\tau}(W^{n,m} + vY^{n,m}) - 
\sqrt{\tau}
z)   - h(\sqrt{1 - 
	\tau}W^{n,k}  - 
\sqrt{\tau}
z).
\end{align*}

\subsubsection{Decomposition of $R_4$} We first derive a decomposition of $R_4$, which is the most complicated of all the error terms $R_i$ to treat. 
By employing the first equality in
\eqref{eq:diffs} and subsequently applying the formula
$f(x + a) - f(x) = \int_0^1 a^T 
\nabla f(x + va) \, dv$, we can express
\begin{align}\label{eq:R1_overline_decomp}
	\begin{split}
		R_4 &= - \sum_{n=0}^{N-1}\sum_{m=1}^{N-1} \sum_{k=2m+1}^{N-1} \sum_{r,s=1}^d  	\mu \biggl\{  
		\overline{ \gamma^{n,k}_{r,s}(1, \tau) } 
		Y^n_r  Y^{n,m}_s
		\biggr\}  \\
		&= \frac{1}{2 \tau } \sum_{n=0}^{N-1} 
		\sum_{m=1}^{N-1}  \sum_{k=2m+1}^{N-1}  \sum_{r,s=1}^d \mu \biggl\{ \int_{\bR^d} 
		\biggl( 
		\overline{h( \sqrt{1 - 
			\tau}(W^{n,k} + Y^{n,k}) - 
		\sqrt{\tau}
		z )} \\
		&- 
		\overline{h( \sqrt{1 - \tau}W^{n,k} - \sqrt{\tau}
		z )}
		\biggr) 
		\phi_{rs}(z) \, dz  \cdot Y^n_r  Y^{n,m}_s  \biggr\}  \\ 
		&=   \int_0^1  \frac{\sqrt{1-\tau} }{2\tau}   \sum_{n=0}^{N-1} 
		\sum_{m=1}^{N-1} 
		\sum_{k=2m+1}^{N-1}  
		 \sum_{r,s, t =1}^d  
		\mu \biggl\{ \int_{\bR^d}  
		\overline{h_t(\sqrt{1 - 
			\tau}(W^{n,k} + vY^{n,k}) - 
		\sqrt{\tau}
		z)  Y_{t}^{n,k}}   \\
		&\times  \phi_{rs}(z)  \, dz \cdot Y_r^n Y_s^{n,m}  \biggr\} \, dv.
	\end{split}
\end{align}
Using integration by parts, we find that 
\begin{align}\label{eq:int_by_parts}
	\begin{split}
		&\int_{\bR^d}  
		h_t(\sqrt{1 - 
			\tau}(W^{n,m} + uvY^{n,m}) - 
		\sqrt{\tau}
		z) \phi_{rs}(z)  \, dz \\ 
		&=  \frac{1}{\sqrt{\tau}} \int_{\bR^d}  h(\sqrt{1 - 
			\tau}(W^{n,m} + uvY^{n,m}) - 
		\sqrt{\tau}
		z)    \phi_{rst}(z) \, dz.
	\end{split}
\end{align}
Consequently,
\begin{align}\label{eq:R4_first}
	\begin{split}
	R_4 &=   \int_0^1  \frac{\sqrt{1-\tau} }{2\tau^{3/2}}   \sum_{n=0}^{N-1} 
	\sum_{m=1}^{N-1} 
	\sum_{k=2m+1}^{N-1}  
	\sum_{r,s, t =1}^d  
	 \int_{\bR^d}
	 \mu \biggl\{   
	\overline{h(\sqrt{1 - 
			\tau}(W^{n,k} + vY^{n,k}) - 
		\sqrt{\tau}
		z)  Y_{t}^{n,k}} \\
		&\times Y_r^n Y_s^{n,m}  \biggr\}    
	 \phi_{rst}(z)  \, dz  \, dv = R_4' + Q_4 - S_4,
	 \end{split}
\end{align}
where
\begin{align}
	\begin{split}\label{eq:R4'}
	R'_4 &=  \int_0^1  \frac{\sqrt{1-\tau} }{2\tau^{3/2}}   \sum_{n=0}^{N-1} 
	\sum_{m=1}^{N-1} 
	\sum_{k=2m+1}^{N-1}  
	\sum_{r,s, t =1}^d  \\
	&\int_{\bR^d}  
	\mu \biggl\{
	\overline{ [  h(\sqrt{1 - 
			\tau}(W^{n,k} + vY^{n,k}) - 
		\sqrt{\tau}
		z)  -  h(\sqrt{1 - 
			\tau}  W^{n, 2k} - 
		\sqrt{\tau}
		z)  ]Y_{t}^{n,k}}    
	Y_r^n Y_s^{n,m}  \biggr\} \\
	&\times  \phi_{rst}(z)  \, dz  \, dv \\
	&= \int_0^1  \frac{\sqrt{1-\tau} }{2\tau^{3/2}}   \sum_{n=0}^{N-1} 
	\sum_{m=1}^{N-1} 
	\sum_{k=2m+1}^{N-1}  
	\sum_{r,s, t =1}^d  \int_{\bR^d}   \mu \biggl\{ 
	\overline{ \eta^{n,k, 2k}(v, \tau, z)
		Y_{t}^{n,k}}   G_{r,s}^{n,m}  \biggr\}   \phi_{rst}(z)  \, dz  \, dv,
	\end{split}
\end{align}
\begin{align*}
Q_4 &=  \frac{\sqrt{1-\tau} }{2\tau^{3/2}}   \sum_{n=0}^{N-1} 
\sum_{m=1}^{N-1} 
\sum_{k=2m+1}^{N-1}  
\sum_{r,s, t =1}^d  
\int_{\bR^d}
\mu \biggl\{   h(\sqrt{1 - 
		\tau} W^{n,2k}  - 
	\sqrt{\tau}
	z)  F_{r,s,t}^{n,m,k}  \biggr\}    
\phi_{rst}(z)  \, dz,
\end{align*}
and
\begin{align*}
S_4 = \frac{\sqrt{1-\tau} }{2\tau^{3/2}}   \sum_{n=0}^{N-1} 
\sum_{m=1}^{N-1} 
\sum_{k=2m+1}^{N-1}  
\sum_{r,s, t =1}^d  
\int_{\bR^d}
\mu \biggl\{   h(\sqrt{1 - 
	\tau} W^{n,2k}  - 
\sqrt{\tau}
z)  Y_{t}^{n,k} \biggr\} \mu ( G_{r,s}^{n,m} )
\phi_{rst}(z)  \, dz.
\end{align*}
Since $\mu(Y_t^{n,k}) = 0$, we can express $S_4$ as the following telescopic sum:
\begin{align}\label{eq:S4}
	\begin{split}
	S_4 &= \frac{\sqrt{1-\tau} }{2\tau^{3/2}}   \sum_{n=0}^{N-1} 
	\sum_{m=1}^{N-1} 
	\sum_{k=2m+1}^{N-1}  
	\sum_{r,s, t =1}^d  
	\int_{\bR^d}
	\mu \biggl\{   \biggl( h(\sqrt{1 - 
		\tau} W^{n,2k}  - 
	\sqrt{\tau}
	z)  -   h( - \sqrt{\tau} z )  \biggr) Y_{t}^{n,k} \biggr\} \\
	&\times \mu ( G_{r,s}^{n,m} )
	\phi_{rst}(z)  \, dz \\
	&= \frac{\sqrt{1-\tau} }{2\tau^{3/2}}   \sum_{n=0}^{N-1} 
	\sum_{m=1}^{N-1} 
	\sum_{k=2m+1}^{N-1} \sum_{ \ell = 2 }^{N-1}
	\sum_{r,s, t =1}^d  
	\int_{\bR^d}
	\mu \biggl\{   \biggl( h(\sqrt{1 - 
		\tau} W^{n, k \ell }  - 
	\sqrt{\tau}
	z)  \\
	&-   h(   \sqrt{1 - 
		\tau} W^{n, k ( \ell + 1) }  - \sqrt{\tau} z )  \biggr)
	  Y_{t}^{n,k} \biggr\} \mu ( G_{r,s}^{n,m} )
	\phi_{rst}(z)  \, dz \\
	&= \frac{\sqrt{1-\tau} }{2\tau^{3/2}}   \sum_{n=0}^{N-1} 
	\sum_{m=1}^{N-1} 
	\sum_{k=2m+1}^{N-1} \sum_{ \ell = 2 }^{N-1}
	\sum_{r,s, t =1}^d  
	\int_{\bR^d}
	\mu \biggl\{   \eta^{n,k\ell, k( \ell + 1)}(0, \tau, z)
	Y_{t}^{n,k} \biggr\} \mu ( G_{r,s}^{n,m} )
	\phi_{rst}(z)  \, dz.
\end{split}
\end{align}

Next, we write $Q_4 = \bar{Q}_4 + \tilde{Q}_4$, where 
\begin{align*}
	\bar{Q}_{4} &=  \frac{\sqrt{1-\tau} }{2\tau^{3/2}}   \sum_{n=0}^{N-1} 
	\sum_{m=1}^{N-1} 
	\sum_{k=2m+1}^{N-1}  
	\sum_{r,s, t =1}^d  
	\int_{\bR^d}
	\mu \biggl\{   \overline{h(\sqrt{1 - 
		\tau} W^{n,2k}  - 
	\sqrt{\tau}
	z)}   F_{r,s,t}^{n,m,k}  \biggr\}    
	\phi_{rst}(z)  \, dz
\end{align*}
and 
\begin{align*}
	\tilde{Q}_{4} &=  \frac{\sqrt{1-\tau} }{2\tau^{3/2}}   \sum_{n=0}^{N-1} 
	\sum_{m=1}^{N-1} 
	\sum_{k=2m+1}^{N-1}  
	\sum_{r,s, t =1}^d  
	\int_{\bR^d}
	\mu \biggl\{   h(\sqrt{1 - 
		\tau} W^{n,2k}  - 
	\sqrt{\tau}
	z)  \biggr\} \mu( F_{r,s,t}^{n,m,k} )
	\phi_{rst}(z)  \, dz.
\end{align*}
Since $h(- \sqrt{\tau } z )$ is nonrandom, we have $\overline{h(- \sqrt{\tau } z )} = 
h(- \sqrt{\tau } z ) - \mu( h(- \sqrt{\tau } z ) ) = 0$. Exploiting this identity together with 
$W^{ n, N k } = 0$ for $k \ge 1$,  we can also express $\bar{Q}_{4}$ as 
a telescopic sum:
\begin{align}\label{eq:Q4_bar}
\begin{split}
\bar{Q}_4 &= 
\frac{\sqrt{1-\tau} }{2\tau^{3/2}}   \sum_{n=0}^{N-1} 
\sum_{m=1}^{N-1} 
\sum_{k=2m+1}^{N-1}  
\sum_{r,s, t =1}^d  
\int_{\bR^d}
\mu \biggl\{  \biggl(   \overline{h(\sqrt{1 - 
		\tau} W^{n,2k} 
	- 
	\sqrt{\tau}
	z)}  \\
	&- \overline{h( - \sqrt{\tau} z  ) } \biggr)  F_{r,s,t}^{n,m,k}  \biggr\}    
\phi_{rst}(z)  \, dz \\
&= \frac{\sqrt{1-\tau} }{2\tau^{3/2}}   \sum_{n=0}^{N-1} 
\sum_{m=1}^{N-1} 
\sum_{k=2m+1}^{N-1} \sum_{ \ell = 2 }^{N - 1}  
\sum_{r,s, t =1}^d  
\int_{\bR^d}
\mu \biggl\{  \biggl(   \overline{h(\sqrt{1 - 
		\tau} W^{n, k \ell } 
	- 
	\sqrt{\tau}
	z)}  \\
&- \overline{h(   \sqrt{1 - 
		\tau} W^{n, k ( \ell + 1) }  - \sqrt{\tau} z  ) } \biggr)  F_{r,s,t}^{n,m,k}  \biggr\}    
\phi_{rst}(z)  \, dz \\
&= \frac{\sqrt{1-\tau} }{2\tau^{3/2}}   \sum_{n=0}^{N-1} 
\sum_{m=1}^{N-1} 
\sum_{k=2m+1}^{N-1} \sum_{ \ell = 2 }^{N - 1}  
\sum_{r,s, t =1}^d  
\int_{\bR^d}
\mu \biggl\{  
\overline{ \eta^{ n, k \ell, k( \ell + 1) }(0, \tau, z) }  F_{r,s,t}^{n,m,k}
\biggr\}    
\phi_{rst}(z)  \, dz.
\end{split}
\end{align}

To handle $\tilde{Q}_4$, 
let $\tilde{Z} \sim \cN_d$ be a random variable independent of all other involved variables.  We assume without loss of generality that 
$\tilde{Z}$ is defined on $(M, \cF, \mu)$, and decompose $\tilde{Q}_4$ into three terms as follows:
\begin{align*}
	\tilde{Q}_4 &= 
	\frac{\sqrt{1-\tau} }{2\tau^{3/2}}   \sum_{n=0}^{N-1} 
	\sum_{m=1}^{N-1} 
	\sum_{k=2m+1}^{N-1}  
	\sum_{r,s, t =1}^d  
	\int_{\bR^d}
	\mu \biggl\{   h(\sqrt{1 - 
		\tau} W^{n,2k}  
		- 
	\sqrt{\tau}
	z) \\
	&-  h( \sqrt{1 - \tau } W - \sqrt{\tau} z )  \biggr\} \mu( F_{r,s,t}^{n,m,k} )
	\phi_{rst}(z)  \, dz \\
	&+\frac{\sqrt{1-\tau} }{2\tau^{3/2}}   \sum_{n=0}^{N-1} 
	\sum_{m=1}^{N-1} 
	\sum_{k=2m+1}^{N-1}  
	\sum_{r,s, t =1}^d  
	\int_{\bR^d}
	\mu \biggl\{   h(\sqrt{1 - 
		\tau} W  - 
	\sqrt{\tau} 
	z) \\
	&-  h( \sqrt{1 - \tau } \tilde{Z} - \sqrt{\tau} z )  \biggr\} \mu( F_{r,s,t}^{n,m,k} )
	\phi_{rst}(z)  \, dz \\
	&+\frac{\sqrt{1-\tau} }{2\tau^{3/2}}   \sum_{n=0}^{N-1} 
	\sum_{m=1}^{N-1} 
	\sum_{k=2m+1}^{N-1}  
	\sum_{r,s, t =1}^d  
	\int_{\bR^d}
	\mu \biggl\{  h( \sqrt{1 - \tau } \tilde{Z} - \sqrt{\tau} z )  \biggr\} \mu( F_{r,s,t}^{n,m,k} )
	\phi_{rst}(z)  \, dz \\
	&= I + II + III.
\end{align*}
Observe that 
\begin{align*}
	I = \frac{\sqrt{1-\tau} }{2\tau^{3/2}}   \sum_{n=0}^{N-1} 
	\sum_{m=1}^{N-1} 
	\sum_{k=2m+1}^{N-1}  
	\sum_{r,s, t =1}^d  
	\int_{\bR^d}
	\mu \biggl\{   
	\eta^{n,2k, -1}(0, \tau, z)
	\biggr\} \mu( F_{r,s,t}^{n,m,k} )
	\phi_{rst}(z)  \, dz
\end{align*}
and, by the third equality in \eqref{eq:diffs},
\begin{align*}
	III
	&= \sum_{n=0}^{N-1}
	\sum_{m=1}^{N-1} 
	\sum_{k=2m+1}^{N-1}  
	\sum_{r,s, t =1}^d  
	\cN_d[ g_{rst}( \cdot, \tau )  ] \mu( F_{r,s,t}^{n,m,k} ). 
\end{align*}
Ultimately, we will exploit \eqref{eq:diff_gN} to control $III$. We conclude that
\begin{align}\label{eq:Q4_tilde}
	\begin{split}
	\tilde{Q}_4 &=  \frac{\sqrt{1-\tau} }{2\tau^{3/2}}   \sum_{n=0}^{N-1} 
	\sum_{m=1}^{N-1} 
	\sum_{k=2m+1}^{N-1}  
	\sum_{r,s, t =1}^d  
	\int_{\bR^d}
	\mu \biggl\{   
	\eta^{n,2k, -1}(0, \tau, z)
	\biggr\} \mu( F_{r,s,t}^{n,m,k} )
	\phi_{rst}(z)  \, dz \\
	&+\frac{\sqrt{1-\tau} }{2\tau^{3/2}}   \sum_{n=0}^{N-1} 
	\sum_{m=1}^{N-1} 
	\sum_{k=2m+1}^{N-1}  
	\sum_{r,s, t =1}^d  
	\int_{\bR^d}
	\mu \biggl\{   h(\sqrt{1 - 
		\tau} W  - 
	\sqrt{\tau} 
	z) \\
	&-  h( \sqrt{1 - \tau } \tilde{Z} - \sqrt{\tau} z )  \biggr\} \mu( F_{r,s,t}^{n,m,k} )
	\phi_{rst}(z)  \, dz \\
	&+ \sum_{n=0}^{N-1}
	\sum_{m=1}^{N-1} 
	\sum_{k=2m+1}^{N-1}  
	\sum_{r,s, t =1}^d  
	\cN_d[ g_{rst}( \cdot, \tau )  ] \mu( F_{r,s,t}^{n,m,k} ). 
	\end{split}
\end{align}

With this, the 
decomposition of $R_{4}$ is complete: 
\begin{align}\label{eq:decomp_R1}
R_4 = R_4' + \bar{Q}_4 + \tilde{Q}_4 - S_4,
\end{align}
where $R_4', \bar{Q}_4, \tilde{Q}_4, S_4$ are given by \eqref{eq:R4'}, \eqref{eq:Q4_bar}, \eqref{eq:Q4_tilde}, and \eqref{eq:S4}, respectively.

\subsubsection{Decomposition of $R_i$ for $i \neq 4$} 
The remaining terms $R_i$, $i \neq 4$, are treated in a similar manner to $R_4$, with minor necessary modifications. 
As for $R_1$, following the steps leading to \eqref{eq:R4_first}, we see that 
\begin{align*}
	R_1 
	&=  \int_0^1 u \frac{\sqrt{1-\tau} }{2\tau^{3/2}}  \sum_{n=0}^{N-1} 
	\sum_{m=1}^{N-1}  \sum_{r,s, t =1}^d  
	\int_{\bR^d}   \mu \biggl\{ 
	h(\sqrt{1 - 
		\tau}(W^{n,m} + uvY^{n,m}) - 
	\sqrt{\tau}
	z)  F_{r,s,t}^{n,m,m}   \biggr\} \\
	&\times \phi_{rst}(z)  \, dz \, dv \\
	&= R_1' + \bar{Q}_1 + \tilde{Q}_1,
\end{align*}
where 
\begin{align}\label{eq:R1'}
	\begin{split}
	R'_{1} &= \int_0^1 u \frac{\sqrt{1-\tau} }{2\tau^{3/2}}  \sum_{n=0}^{N-1} 
	\sum_{m=1}^{N-1}  \sum_{r,s, t =1}^d  
	\int_{\bR^d}   \mu \biggl\{ 
	\biggl[  h(\sqrt{1 - 
		\tau}(W^{n,m} + uvY^{n,m}) - 
	\sqrt{\tau}
	z)  \\
	&- h(\sqrt{1 - 
		\tau}W^{n,m}  - 
	\sqrt{\tau}
	z)  \biggr] F_{r,s,t}^{n,m,m}   \biggr\} 
	 \phi_{rst}(z)  \, dz \, dv  \\
	&= \int_0^1 u \frac{\sqrt{1-\tau} }{2\tau^{3/2}}  \sum_{n=0}^{N-1} 
	\sum_{m=1}^{N-1}  \sum_{r,s, t =1}^d  
	\int_{\bR^d}   \mu \biggl\{ 
	\eta^{n,m,m}(uv, \tau, z)
	F_{r,s,t}^{n,m,m}   \biggr\} 
	\phi_{rst}(z)  \, dz \, dv,
	\end{split}
\end{align}
\begin{align*}
\bar{Q}_{1} &=  u \frac{\sqrt{1-\tau} }{2\tau^{3/2}}  \sum_{n=0}^{N-1} 
\sum_{m=1}^{N-1}  \sum_{r,s, t =1}^d  
\int_{\bR^d}   \mu \biggl\{ 
\overline{h(\sqrt{1 - 
	\tau}W^{n,m}  - 
\sqrt{\tau}
z)}  F_{r,s,t}^{n,m,m}   \biggr\} \phi_{rst}(z)  \, dz,
\end{align*}
and 
\begin{align*}
	\tilde{Q}_{1} &= u \frac{\sqrt{1-\tau} }{2\tau^{3/2}}  \sum_{n=0}^{N-1} 
	\sum_{m=1}^{N-1}  \sum_{r,s, t =1}^d  
	\int_{\bR^d}   \mu \biggl\{ 
	h(\sqrt{1 - 
			\tau} W^{n,m}  - 
		\sqrt{\tau}
		z) \biggr\} \mu ( F_{r,s,t}^{n,m,m} ) \phi_{rst}(z)  \, dz.
\end{align*}
Further, we obtain the following representations of $\bar{Q}_1$ and $\tilde{Q}_1$ using manipulations similar to those in derivations of the formulas for $\bar{Q}_4$ and $\tilde{Q}_4$, respectively:
\begin{align}\label{eq:Q1_bar}
	\begin{split}
	\bar{Q}_1 
	&= u \frac{\sqrt{1-\tau} }{2\tau^{3/2}}     \sum_{n=0}^{N-1} 
	\sum_{m=1}^{N-1}  \sum_{r,s,t=1}^d  \int_{\bR^d}  
	\mu \biggl\{  \overline{  \eta^{n,m,2m}(0, \tau, z) }  F_{r,s,t}^{n,m,m} \biggr\}  \cdot \phi_{rst}(z)   \, dz \\
	&+ u \frac{\sqrt{1-\tau} }{2\tau^{3/2}}    \sum_{n=0}^{N-1} 
	\sum_{m=1}^{N-1}  \sum_{\ell = 2}^{N-1}  \sum_{r,s,t=1}^d  \int_{\bR^d}  
	\mu \biggl\{   \overline{ \eta^{n, m \ell, m ( \ell + 1)}(0, \tau, z)  }   F_{r,s,t}^{n,m,m} \biggr\}  \phi_{rst}(z)   \, dz,
	\end{split}
\end{align}
and
\begin{align}\label{eq:Q1_tilde}
\begin{split}
\tilde{Q}_1 &=  u \frac{\sqrt{1-\tau} }{2\tau^{3/2}}  \sum_{n=0}^{N-1} 
\sum_{m=1}^{N-1}  \sum_{r,s, t =1}^d  
\int_{\bR^d}   \mu \biggl\{ 
\eta^{ n,m, -1  }(0, \tau, z)
\biggr\} \mu ( F_{r,s,t}^{n,m,m} )  \cdot \phi_{rst}(z)  \, dz \\
&+  u \frac{\sqrt{1-\tau} }{2\tau^{3/2}}  \sum_{n=0}^{N-1} 
\sum_{m=1}^{N-1}  \sum_{r,s, t =1}^d  
\int_{\bR^d}   \mu \biggl\{ 
h(\sqrt{1 - 
	\tau}W - 
\sqrt{\tau}
z) \\
&-  h(\sqrt{1 - 
	\tau} \tilde{Z} 
- 
\sqrt{\tau}
z)  \biggr\} \mu ( F_{r,s,t}^{n,m,m} )  \cdot \phi_{rst}(z)  \, dz \\
&+ u \sum_{n=0}^{N-1} 
\sum_{m=1}^{N-1}  \sum_{r,s, t =1}^d  \cN_d [  g_{rst}( \cdot, \tau  )   ]  \mu ( F_{r,s,t}^{n,m,m} ).
\end{split}
\end{align}
With this, the 
decomposition of $R_{1}$ is complete: 
\begin{align}\label{eq:decomp_R1}
R_1 = R_1' + \bar{Q}_1 +  \tilde{Q}_1,
\end{align}
where $R_1'$, $\bar{Q}_1$, and $\tilde{Q}_1$ are given by \eqref{eq:R1'}, \eqref{eq:Q1_bar}, and
\eqref{eq:Q1_tilde}, respectively.

In the case of $R_2$, we have a decomposition similar to \eqref{eq:decomp_R1}:
\begin{align*}
	R_2 &= R_2' + \bar{Q}_2 +  \tilde{Q}_2,  \\
	R_2' &= \int_0^1 u \frac{\sqrt{1-\tau} }{2\tau^{3/2}}  \sum_{n=0}^{N-1} 
	\sum_{r,s, t =1}^d  
	\int_{\bR^d}   \mu \biggl\{ 
	\eta^{n,0,0}(uv, \tau, z)
	F_{r,s,t}^{n,0,0}   \biggr\} 
	\cdot \phi_{rst}(z)  \, dz \, dv, \\
	\bar{Q}_{2} &= u \frac{\sqrt{1-\tau} }{2\tau^{3/2}}    \sum_{n=0}^{N-1} 
	\sum_{\ell = 0}^{N-1}  \sum_{r,s,t=1}^d  \int_{\bR^d}  
	\mu \biggl\{   \overline{ \eta^{n,  \ell,  \ell + 1}(0, \tau, z)  }   F_{r,s,t}^{n,0,0} \biggr\}  \phi_{rst}(z)   \, dz,
\end{align*}
\begin{align*}
	\tilde{Q}_{2} &= u \frac{\sqrt{1-\tau} }{2\tau^{3/2}}  \sum_{n=0}^{N-1} 
\sum_{r,s, t =1}^d  
\int_{\bR^d}   \mu \biggl\{ 
\eta^{n,0,-1}(0,\tau, z)
 \biggr\} \mu ( F_{r,s,t}^{n,0,0} )  \cdot \phi_{rst}(z)  \, dz \\
&+ u \frac{\sqrt{1-\tau} }{2\tau^{3/2}}  \sum_{n=0}^{N-1} 
\sum_{r,s, t =1}^d  
\int_{\bR^d}   \mu \biggl\{ 
h(\sqrt{1 - 
	\tau}W - 
\sqrt{\tau}
z) \\
&-  h(\sqrt{1 - 
	\tau} \tilde{Z} 
- 
\sqrt{\tau}
z)  \biggr\} \mu ( F_{r,s,t}^{n,0,0} )  \cdot \phi_{rst}(z)  \, dz \\
&+ u \sum_{n=0}^{N-1}  \sum_{r,s, t =1}^d  \cN_d [  g_{rst}( \cdot, \tau  )   ]  \mu ( F_{r,s,t}^{n,0,0} ).
\end{align*}

Each of $R_6$ and $R_7$ admits a 
decomposition into two parts corresponding to
$R_4'$ and $S_4$ in the decomposition of $R_4$:
\begin{align*}
	R_6 &= R_6' + S_6, \\
	R_6' &= - \int_0^1  \frac{\sqrt{1-\tau} }{2\tau^{3/2}}   \sum_{n=0}^{N-1} 
	\sum_{m=1}^{N-1} 
	\sum_{k=0}^{m}  
	\sum_{r,s, t =1}^d  
	\int_{\bR^d}
	\mu \biggl\{   
	\eta^{n,k, 2m}(v, \tau, z) Y_{t}^{n,k}
	 \biggr\} 
	\mu ( G_{r,s}^{n,m} )     \phi_{rst}(z)  \, dz  \, dv,    \\
	S_6 &= - \frac{\sqrt{1-\tau} }{2\tau^{3/2}}   \sum_{n=0}^{N-1} 
	\sum_{m=1}^{N-1} 
	\sum_{k=0}^{m} \sum_{\ell = 2}^{N- 1}
	\sum_{r,s, t =1}^d  
	\int_{\bR^d}
	\mu \biggl\{   
	\eta^{n,m \ell, m ( \ell + 1)}(0, \tau, z)
	 Y_{t}^{n,k}  \biggr\} 
	\mu ( G_{r,s}^{n,m} )     \phi_{rst}(z)  \, dz
\end{align*}
and
\begin{align*}
	R_7 &= R_7' + S_7, \\
	R_7' &= - \int_0^1  \frac{\sqrt{1-\tau} }{2\tau^{3/2}}   \sum_{n=0}^{N-1} 
	\sum_{r,s, t =1}^d  
	\int_{\bR^d}
	\mu \biggl\{   
	\eta^{n,0, 1}(v, \tau, z) Y_{t}^{n}
	\biggr\} 
	\mu ( G_{r,s}^{n,0} )     \phi_{rst}(z)  \, dz  \, dv,    \\
	S_7 &=  \frac{\sqrt{1-\tau} }{2\tau^{3/2}}   \sum_{n=0}^{N-1} 
	\sum_{\ell = 1}^{N- 1}
	\sum_{r,s, t =1}^d  
	\int_{\bR^d}
	\mu \biggl\{   
	\eta^{n, \ell,  \ell + 1}(0, \tau, z)
	Y_{t}^{n}  \biggr\} 
	\mu ( G_{r,s}^{n,0} )     \phi_{rst}(z)  \, dz.
\end{align*}

In the case of $R_3$, we first decompose
\begin{align*}
	R_3 &= - \sum_{n=0}^{N-1}\sum_{m=1}^{N-1} \sum_{k=m+1}^{2m} \sum_{r,s=1}^d  \mu \biggl\{  
	\gamma^{n,k}_{r,s}(1, \tau) 
	Y^n_r  Y^{n,m}_s \biggr\} + \sum_{n=0}^{N-1}\sum_{m=1}^{N-1} \sum_{k=m+1}^{2m} \sum_{r,s=1}^d  \mu (
	\gamma^{n,k}_{r,s}(1, \tau)  )
	 \mu( Y^n_r  Y^{n,m}_s ) \\
	 &= Q_3 - Q_3'.
\end{align*}
We decompose $Q_3$ as in the case of $R_1$:
\begin{align*}
	Q_3 &= R_3' + \bar{Q}_3 + \tilde{Q}_3, \\
	R_3' &= \int_0^1  \frac{\sqrt{1-\tau} }{2\tau^{3/2}}  \sum_{n=0}^{N-1} 
	\sum_{m=1}^{N-1}  \sum_{k=m+1}^{2m}  \sum_{r,s, t =1}^d  
	\int_{\bR^d}   \mu \biggl\{ 
	\eta^{n,k, 3m }(v, \tau, z)
	F_{r,s,t}^{n,m,k}   \biggr\} 
	\cdot \phi_{rst}(z)  \, dz \, dv, \\
	\bar{Q}_3 &= 
	\frac{\sqrt{1-\tau} }{2\tau^{3/2}}    \sum_{n=0}^{N-1} 
	\sum_{m=1}^{N-1} \sum_{k=m+1}^{2m}   \sum_{\ell = 3}^{N-1}  \sum_{r,s,t=1}^d  \int_{\bR^d}  
	\mu \biggl\{   \overline{ \eta^{n, m \ell, m ( \ell + 1)}(0, \tau, z)  }   F_{r,s,t}^{n,m,k} \biggr\}  \phi_{rst}(z)   \, dz, \\
	\tilde{Q}_3 &=  \frac{\sqrt{1-\tau} }{2\tau^{3/2}}  \sum_{n=0}^{N-1} 
	\sum_{m=1}^{N-1} \sum_{k=m+1}^{2m}   \sum_{r,s, t =1}^d  
	\int_{\bR^d}   \mu \biggl\{ 
	\eta^{ n, 3m , -1  }(0, \tau, z)
	\biggr\} \mu ( F_{r,s,t}^{n,m,k} )  \cdot \phi_{rst}(z)  \, dz \\
	&+  \frac{\sqrt{1-\tau} }{2\tau^{3/2}}  \sum_{n=0}^{N-1} 
	\sum_{m=1}^{N-1} 
	\sum_{k=m+1}^{2m} 
	\sum_{r,s, t =1}^d  
	\int_{\bR^d}   \mu \biggl\{ 
	h(\sqrt{1 - 
		\tau}W - 
	\sqrt{\tau}
	z) \\
	&-  h(\sqrt{1 - 
		\tau} \tilde{Z} 
	- 
	\sqrt{\tau}
	z)  \biggr\} \mu ( F_{r,s,t}^{n,m,k} )  \cdot \phi_{rst}(z)  \, dz \\
	&+ \sum_{n=0}^{N-1} 
	\sum_{m=1}^{N-1} \sum_{k=m+1}^{2m}   \sum_{r,s, t =1}^d  \cN_d [  g_{rst}( \cdot, \tau  )   ]  \mu ( F_{r,s,t}^{n,m,k} ).
\end{align*}
For $Q_3'$, we have a decomposition similar to that of $R_6$:
\begin{align*}
	Q_3' &= R_3'' + S_3, \\
	R_3'' &= - \int_0^1  \frac{\sqrt{1-\tau} }{2\tau^{3/2}}   \sum_{n=0}^{N-1} 
	\sum_{m=1}^{N-1} 
	\sum_{k=m + 1}^{2m}  
	\sum_{r,s, t =1}^d  
	\int_{\bR^d}
	\mu \biggl\{   
	\eta^{n,k, 3m}(v, \tau, z) Y_{t}^{n,k}
	\biggr\} 
	\mu ( G_{r,s}^{n,m} )     \phi_{rst}(z)  \, dz  \, dv,    \\
	S_3 &= - \frac{\sqrt{1-\tau} }{2\tau^{3/2}}   \sum_{n=0}^{N-1} 
	\sum_{m=1}^{N-1} 
	\sum_{k=m+1}^{2m} \sum_{\ell = 3m}^{N- 1}
	\sum_{r,s, t =1}^d  
	\int_{\bR^d}
	\mu \biggl\{   
	\eta^{n, \ell,  \ell + 1 }(0, \tau, z)
	Y_{t}^{n,k}  \biggr\} 
	\mu ( G_{r,s}^{n,m} )     \phi_{rst}(z)  \, dz.
\end{align*}
Hence,
$$
R_3 = R_3' + \bar{Q}_3 + \tilde{Q}_3 - R_3'' - S_3.
$$

Finally, we have the following decomposition of $R_5$, derived 
in the same way as the decomposition of $R_4$:
\begin{align*}
	    R_5 &= R_5' + \bar{Q}_5 + \tilde{Q}_5 - S_5,\\
		R'_5
		&= \int_0^1  \frac{\sqrt{1-\tau} }{2\tau^{3/2}}   \sum_{n=0}^{N-1} 
		\sum_{m=1}^{N-1} 
		\sum_{r,s, t =1}^d  \mu \biggl\{ \int_{\bR^d}  
		\overline{ \eta^{n,m, 2m}(v, \tau, z)
			Y_{t}^{n,m}}   G_{r,s}^{n,0}  \biggr\}   \phi_{rst}(z)  \, dz  \, dv, \\
		\bar{Q}_5 &= \frac{\sqrt{1-\tau} }{2\tau^{3/2}}   \sum_{n=0}^{N-1} 
		\sum_{m=1}^{N-1} 
		\sum_{ \ell = 2 }^{N - 1}  
		\sum_{r,s, t =1}^d  
		\int_{\bR^d}
		\mu \biggl\{  
		\overline{ \eta^{ n, m \ell, m( \ell + 1) }(0, \tau, z) }  F_{r,s,t}^{n,0,m}
		\biggr\}    
		\phi_{rst}(z)  \, dz,
\end{align*}
\begin{align*}
		\tilde{Q}_5 &=  \frac{\sqrt{1-\tau} }{2\tau^{3/2}}   \sum_{n=0}^{N-1} 
		\sum_{m=1}^{N-1} 
		\sum_{r,s, t =1}^d  
		\int_{\bR^d}
		\mu \biggl\{   
		\eta^{n,2m, -1}(0, \tau, z)
		\biggr\} \mu( F_{r,s,t}^{n,0,m} )
		\phi_{rst}(z)  \, dz \\
		&+\frac{\sqrt{1-\tau} }{2\tau^{3/2}}   \sum_{n=0}^{N-1} 
		\sum_{m=1}^{N-1}  
		\sum_{r,s, t =1}^d  
		\int_{\bR^d}
		\mu \biggl\{   h(\sqrt{1 - 
			\tau} W  - 
		\sqrt{\tau} 
		z) \\
		&-  h( \sqrt{1 - \tau } \tilde{Z} - \sqrt{\tau} z )  \biggr\} \mu( F_{r,s,t}^{n,0,m} )
		\phi_{rst}(z)  \, dz \\
		&+ \sum_{n=0}^{N-1}
		\sum_{m=1}^{N-1}   
		\sum_{r,s, t =1}^d  
		\cN_d[ g_{rst}( \cdot, \tau )  ] \mu( F_{r,s,t}^{n,0,m} ), \\
		S_5 &= \frac{\sqrt{1-\tau} }{2\tau^{3/2}}   \sum_{n=0}^{N-1} 
		\sum_{m=1}^{N-1} 
		\sum_{ \ell = 2 }^{N-1}
		\sum_{r,s, t =1}^d  
		\int_{\bR^d}
		\mu \biggl\{   \eta^{n,m\ell, m( \ell + 1)}
		Y_{t}^{n,m} \biggr\} \mu ( G_{r,s}^{n,m} )
		\phi_{rst}(z)  \, dz.
\end{align*}

\subsection{Estimates on $\int_{\ve^2}^1 |\tilde{Q}_i(\tau)| \, d \tau$} 
Given $n, m \in \bZ$, we decompose $W^{n,m} = W_{-}^{n,m}  + W_{+}^{n,m}$ and 
$Y^{n,m} = Y_{-}^{n,m} + Y_{+}^{n,m}$, where 
\begin{align*}
	W^{n,m}_{-} = \sum_{ \substack{  0 \le i < N \\ i < n - m } } Y_i, \quad 
	W^{n,m}_{+} = \sum_{ \substack{  0 \le i < N \\ i > n + m } } Y_i, \quad 
	Y_{-}^{n,m} = \mathbf{1}_{ n-m \ge 0 }  Y^{n-m}, \quad  
	Y_{+}^{n,m} = \mathbf{1}_{ n+m < N }  Y^{n+m}.
\end{align*} 

For fixed $v \in [0,1]$, $\tau \in [0,1]$, $z \in \bR^d$, 
we introduce the following notation for convenience:
\begin{align}\label{eq:eta_tilde}
	\begin{split}
		\tilde{\eta}^{n,m,k}(x,y) &=  h(\sqrt{1 - 
			\tau} ( \tilde{W}^{n,m}(x,y)   + v  \tilde{Y}^{n,m}(x,y)   ) - 
		\sqrt{\tau}
		z)  \\ 
		&-   h(\sqrt{1 - 
			\tau} \tilde{W}^{n,k}(x,y)  - 
		\sqrt{\tau}
		z),
	\end{split}
\end{align}
where
\begin{align}\label{eq:W_tilde}
\tilde{W}^{n,m}(x,y) = W_{-}^{n,m}(x)  + W_{+}^{n,m}(y) \quad \text{and} \quad 
\tilde{Y}^{n,m}(x,y) = Y_{-}^{n,m}(x) + Y_{+}^{n,m}(y).
\end{align}
Recall the definition of $\bar{b}$ from \eqref{eq:notations}. Set
	\begin{align}\label{eq:E}
	\cE &=  d^{1/4}  (  \ve +  \Vert b^{-1} \Vert_s   ) +   \bar{b} 
	\fD,
\end{align}	
where $\fD$ is as in \eqref{eq:iso_D}. Recall also that, by \eqref{eq:varphi_cond}, $\Vert X^n \Vert_\infty \le L$ holds 
for all $n \ge 0$ with $L \ge 1$.

\begin{lem}\label{lem:global_estimates} Assume \eqref{eq:varphi_cond} and
\eqref{eq:cc}. Let $0 \le n < N$, $k \ge 0$, $m \ge - 1$,
	$\ve > 0$, $\tau \in [0,1]$, $z \in \bR^d$, 
	$v \in [0,1]$. 
	Then, 
	\begin{align}\label{eq:global_e1}
		&\int_M |  \tilde{\eta}^{n, m , m + k  } (x,x)   | \, d \mu(x) 
		\lesssim  (\delta_2 -  \delta_1)^{- 3 K_0 / 2} L^2  \ve^{-1}( m + k + 2)^2   \Vert b^{-1} \Vert_s  \cE,
	\end{align}
	and 
	\begin{align}\label{eq:global_e2}
		\begin{split}
		&\iint_{M^2} | \tilde{\eta}^{n,m , m + k } (x,y) | \, d \mu(x) \, d \mu(y) \\
		&\lesssim  (\delta_2 -  \delta_1)^{- 3 K_0 / 2} 2^{3K_0/2}C_0^{3/2} L^2 \ve^{-1} (k + m + 2)^2
		\Vert b^{-1} \Vert_s
		\cE.
		\end{split}
	\end{align}
\end{lem}

\begin{remark}
	Conditions (C1) and (C2) have been devised for the purpose of obtaining \eqref{eq:global_e2}.
\end{remark}


\begin{proof}[Proof of Lemma \ref{lem:global_estimates}] Recall that $h = h_{C, \ve}$ is the function from Lemma \ref{lem:approx}. 
	
\noindent\textbf{Proof of \eqref{eq:global_e1}.} We have $\tilde{\eta}^{n,m,m+k}(x,y) = h(w_1) - h(w_2)$, where 
	\begin{align*}
		\Vert w_1 - w_2 \Vert &= \Vert  \sqrt{1 - 
			\tau}( \tilde{W}^{n,m}(x,y) + v\tilde{Y}^{n,m}(x,y)) - 
		\sqrt{\tau}
		z - ( \sqrt{1 - 
			\tau}\tilde{W}^{n,m + k}(x,y)  - 
		\sqrt{\tau}z ) \Vert \\ 
		&\lesssim (k  + 1)\sqrt{1 - \tau}  L \Vert b^{-1} \Vert_s.
	\end{align*}
	Then, using Lemma \ref{lem:approx}, we obtain the following implications:
	\begin{align*}
		&\tilde{\eta}^{n,m , m + k  }(x,y)  \neq 0 \\
		&\implies 
		\sqrt{1 - 
			\tau}(\tilde{W}^{n,m}(x,y) + v\tilde{Y}^{n,m }(x,y)) - 
		\sqrt{\tau}z  \in  C^{  \ve + c ( k + 1)\sqrt{1 - \tau} L \Vert b^{-1} \Vert_s  } \setminus C^{- c (k + 1)\sqrt{1 - \tau} L \Vert b^{-1} \Vert_s  } \\ 
		&\implies  \tilde{W}^{n,m }(x, y) + v\tilde{Y}^{n,m }(x, y) \in D^{   \frac{ \ve} { \sqrt{1 - \tau }} + c (k + 1) L \Vert b^{-1} \Vert_s   } \setminus 
		D^{    - c ( k  + 1) L \Vert b^{-1} \Vert_s   },
	\end{align*}
	where
	$
	D = D(\tau, z) \in \cC
	$ and $c > 0$ is an absolute constant. Moreover, since $$
	\Vert W -  (W^{n,m } + v Y^{n,m} ) \Vert \lesssim ( m + 2) L \Vert b^{-1} \Vert_s,
	$$
	it follows that 
	\begin{align}\label{eq:A1}
		A_1 &:= \{ x \in M \: : \: \eta^{n,m , m + k}(x,x)  \neq 0  \} \subset \{ W \in D^{ a_1 } \setminus D^{- a_2}  \},
	\end{align}
	where
	\begin{align*}
		a_1 &:= \frac{ \ve} { \sqrt{1 - \tau }} + c' (k + m + 2) L \Vert b^{-1} \Vert_s, \: 
		a_2 :=  c' (k + m + 2) L \Vert b^{-1} \Vert_s
	\end{align*}
	for some absolute constant $c' > 0$. Recalling the definition of $\fD$ 
	from \eqref{eq:iso_D} and using 
	Lemma \ref{lem:normal_approx}, we obtain
	\begin{align}\label{eq:A1_to_gaussian}
		\begin{split}
		\mu(A_1) &\le \mu( W \in D^{a_1} \setminus D^{-a_2} )  \\
		&= ( \mu( W \in D^{a_1} ) - \cN_d ( D^{a_1}  ) ) + ( \cN_d( D^{-a_2} ) - \mu( W \in D^{- a_2} )  ) \\
		&+ \cN_d( D^{a_1} \setminus D ) + \cN_d( D  \setminus D^{-a_2} ) \\
		&\le 2 (\delta_2 -  \delta_1)^{- 3 K_0 / 2}  \bar{b}   \fD +  4 d^{1/4} ( a_1 + a_2 ) \\
		&\lesssim (\delta_2 -  \delta_1)^{- 3 K_0 / 2} L (k + m + 2)  \frac{1}{ \sqrt{1 - \tau } }
		\biggl[  \bar{b}   \fD
		+ d^{1/4} (  \ve +  \Vert b^{-1} \Vert_s ) \biggr] \\
		&\lesssim (\delta_2 -  \delta_1)^{- 3 K_0 / 2} L (k + m + 2)  \frac{1}{ \sqrt{1 - \tau } }
		\cE.
		\end{split}
	\end{align}
	Moreover,  by Lemma \ref{lem:approx}-(iv), for all $x,y \in M$,
	\begin{align}\label{eq:eta_tilde_sup}
		&| \tilde{\eta}^{n,m  ,m + k }(x,y)| \lesssim   \ve^{-1}  ( k + 1)  L \sqrt{1 - \tau} \Vert b^{-1} \Vert_s.
	\end{align}
	Estimate \eqref{eq:global_e1} follows by combining \eqref{eq:A1}, \eqref{eq:A1_to_gaussian}, and \eqref{eq:eta_tilde_sup}:
	\begin{align*}
		&\int_M |   \tilde{\eta}^{n,m , m+k }(x,x)   | \, d \mu(x)  \\ 
		&\lesssim  \ve^{-1} ( k + 1) L \sqrt{1 - \tau} \Vert b^{-1} \Vert_s \mu(A_1) \\
		&\lesssim (\delta_2 -  \delta_1)^{- 3 K_0 / 2} L^2\ve^{-1}  (k + m + 2)^2    \Vert b^{-1} \Vert_s  \biggl[   
		  \bar{b}    \fD
		+ d^{1/4}(  \ve +   \Vert b^{-1} \Vert_s  ) \biggr].
	\end{align*}
	
	\noindent\textbf{Proof of \eqref{eq:global_e2}.} We define 
	\begin{align*}
		\tilde{W}^n(x,y) &= b^{-1}(\delta_1, \delta_2)  S( \delta_1 ,  (n / N)   )(x) +  b^{-1}(\delta_1, \delta_2)   S(  (n/N)  ,\delta_2 )(y)  \\
		&= \sum_{\delta_1 N \le i < n } Y^i(x) + \sum_{n \le i < \delta_2 N } Y^i(y),
	\end{align*}
	where we recall that $S(\delta_1, \delta_2) = S_N(\delta_1, \delta_2)$ and $b^{-1}(\delta_1, \delta_2) = b_N^{-1}(\delta_1, \delta_2) $ are defined as in \eqref{eq:notations-1} and \eqref{eq:notations}, respectively.
	Then, as in the proof of \eqref{eq:global_e1}, we see that  
	\begin{align*}
		A_2 &:= \{ (x,y) \in M \times M \: : \:  \tilde{\eta}^{n,m, m + k}(x,y) \neq 0 \} \\ 
		&\subset \{  (x,y) \in M \times M \: : \: \tilde{W}^n(x,y) \in D^{a_1} \setminus D^{-a_2} \},
	\end{align*}
	where $D = D(\tau, z) \in \cC$. 
	Set $\delta = n / N$. 
	We suppose that $\delta \in [\delta_1, \delta_2]$. Otherwise, we either have 
	$\tilde{W}^n(x,y) = b^{-1}(\delta_1, \delta_2)   S( \delta_1 ,  \delta_2  )(x)$ or
	$\tilde{W}^n(x,y) = b^{-1}(\delta_1, \delta_2)   S(  \delta_1 ,\delta_2 )(y)$, and \eqref{eq:global_e2} reduces 
	to \eqref{eq:global_e1}.
	
	\noindent\emph{Case $1^\circ$:} Suppose $\delta_2 - \delta < \delta - \delta_1$. Then,
	$\delta - \delta_1 > (\delta_2 - \delta_1) / 2$ and, 
	by the second inequality in 
	\eqref{eq:cc}, the matrix
	$\Sigma (  \delta_1, \delta  )$ is invertible since $ \lambda_{\min} ( \Sigma (  \delta_1, \delta  ) ) \ge C_0^{-1} \lambda_{ \max }(   \Sigma (  \delta_1, \delta_2  )  ) > 0$.
	In this case we \say{discard} the part 
	$b^{-1}(\delta_1, \delta_2) S(  \delta  ,\delta_2 )(y)$ from $\tilde{W}^n(x,y)$ and repeat the argument used to prove \eqref{eq:global_e1}, after replacing the scaling $b^{-1}(\delta_1, \delta_2)$ in $b^{-1}(\delta_1, \delta_2)  S( \delta_1 , \delta )$ with 
	$b^{-1}(\delta_1, \delta)$. To control the error that results 
	from this replacement, we use the estimate
	\begin{align}\label{eq:replacement}
	\Vert b^{-1}( \delta_1, \delta ) b(\delta_1, \delta_2) \Vert_s \le \frac{ \lambda_{ \max }( \Sigma ( \delta_1, \delta_2 ) ) }{  \lambda_{\min} (\Sigma ( \delta_1, \delta ))   } \le C_0,
	\end{align}
	which is a direct consequence of \eqref{eq:cc}.  We now proceed to detail the argument.
	
	First, observe that 
	$$
	\tilde{W}^n(x,y) \in  D^{a_1} \setminus D^{ - a_2  } \iff    b^{-1}(\delta_1, \delta_2)  S( \delta_1 , \delta )(x)  \in  D_1^{a_1}(y) \setminus D_1^{ - a_2  }(y),
	$$
	where 
	$$
	D_1(y) = 
	\{ v - 
	b^{-1}(\delta_1, \delta_2)   S(  (n/N)  ,\delta_2 )(y) 
	\: : \: v \in D  \} \in \cC.
	$$
	Since
	$$W( \delta_1, \delta ) 
	= b^{-1}( \delta_1, \delta ) S(\delta_1, \delta),
	$$
	multiplying $b^{-1}(\delta_1, \delta_2)  S( \delta_1 , \delta )$ by 
	$b^{-1}( \delta_1, \delta ) b(\delta_1, \delta_2)$ from the left, we obtain
	\begin{align}\label{eq:qqs}
		(x,y) \in A_2 \implies W( \delta_1, \delta )(x) \in D_2^{ a_1'}(y)
		\setminus D_2^{- a_2'}(y),
	\end{align}
	where $D_2(y) \in \cC$ and, by \eqref{eq:replacement},
	\begin{align*}
		a_1' &= \frac{ \lambda_{ \max }( \Sigma( \delta_1, \delta_2 ) ) }{  \lambda_{\min} (\Sigma( \delta_1, \delta ))   } a_1 
		\le C_0 a_1, \quad 
		a_2' = \frac{ \lambda_{ \max }( \Sigma ( \delta_1, \delta_2 ) ) }{  \lambda_{\min} (\Sigma ( \delta_1, \delta ))   } a_2 
		\le C_0 a_2.
	\end{align*}
	An application of Lemma \ref{lem:normal_approx} as in \eqref{eq:A1_to_gaussian} yields the following estimate 
	for any $y \in M$:
	\begin{align*}
		&\mu \biggl(  W_N( \delta_1, \delta ) \in 
		D_2^{ a_1' }(y) \setminus D_2^{ a_2' }(y)  \biggr) \\
		&\lesssim d^{1/4} C_0 (a_1 + a_2) + d_c(   \cL( W( \delta_1  , \delta ) ), \cN_d   ) \\
		&\lesssim  d^{1/4} C_0 (a_1 + a_2) + (  \delta - \delta_1 )^{ - 3K_0 / 2}  
		\max\{  N \Vert b^{-1}( \delta_1, \delta )   \Vert_s^{3},  \Vert b^{-1}( \delta_1, \delta )   \Vert_s \}  
		\fD.
	\end{align*}
	Since $\delta - \delta_1 \ge ( \delta_2 - \delta_1) / 2$ and
	\begin{align*}
		\Vert b^{-1} ( \delta_1, \delta) \Vert_s^{-2} &= 
		\lambda_{ \min }( \Sigma(\delta_1, \delta) ) \ge C_0^{-1}  \lambda_{\max} ( \Sigma(\delta_1, \delta_2)) 
		\ge C_0^{-1}  \lambda_{\min} ( \Sigma(\delta_1, \delta_2)) \\
		&= C_0^{-1} \Vert b^{-1}(\delta_1, \delta_2) \Vert_s^{-2},
	\end{align*}
	it follows that 
	\begin{align}
		\begin{split}\label{eq:qq2}
		&\mu \biggl(  W( \delta_1, \delta ) \in 
		D_2^{ a_1' }(y) \setminus D_2^{ a_2' }(y)  \biggr) \\
		&\lesssim d^{1/4} C_0 (a_1 + a_2) +  C_0^{3/2} 2^{3K_0 / 2} ( \delta_2 - \delta_1 )^{ - 3K_0 / 2}  
		\bar{b}
		\fD \\
		&\lesssim (\delta_2 -  \delta_1)^{- 3 K_0 / 2} 2^{3K_0/2}C_0^{3/2} L (k + m + 2)  \frac{1}{ \sqrt{1 - \tau } }
		\cE.
		\end{split}
	\end{align}

	Combining \eqref{eq:eta_tilde_sup}, \eqref{eq:qqs}, and \eqref{eq:qq2}, we 
	arrive at the desired estimate \eqref{eq:global_e2}:
	\begin{align*}
		&\iint_{M^2} | \tilde{\eta}^{n,m  , m + k  } (x,y) | \, d \mu(x) \, d \mu(y) \\ 
		&\lesssim \ve^{-1} ( k + 1) L \sqrt{1 - \tau} \Vert b^{-1} \Vert_s ( \mu \otimes \mu )(A_2) \\
		&\lesssim \ve^{-1}( k + 1)  L \sqrt{1 - \tau} \Vert b^{-1} \Vert_s  \int_M \mu \biggl(   W_N( \delta_1, \delta ) \in 
		D_2^{ a_1' }(y) \setminus D_2^{ a_2' }(y)
		\biggr) \, d \mu(y) \\ 
		&\lesssim  (\delta_2 -  \delta_1)^{- 3 K_0 / 2} 2^{3K_0/2}C_0^{3/2} L^2 \ve^{-1} (k + m + 2)^2
		\Vert b^{-1} \Vert_s
		\cE.
	\end{align*}
	
	\noindent\emph{Case $2^{\circ}$:}  $| \delta_2 - \delta |  \ge | \delta - \delta_1 |$. We obtain 
	\eqref{eq:global_e2} as in Case $1^{\circ}$, but applying
	the first inequality in \eqref{eq:cc} instead of the second one.
\end{proof}

Observe that, whenever $n \ge 0$ and $0 \le m \le k$, \eqref{eq:correlations} implies
\begin{align}\label{eq:multiple}
| \mu ( G_{r,s}^{n,m} ) | \le \bfC L^2 \Vert b^{-1} \Vert_s q^m \quad \text{and} \quad 
| \mu( F_{r,s,t}^{n,m,k} ) | \le \bfC \Vert  b^{-1} \Vert_s^3 L^3 q^{ \max\{  m, k - m \} }.
\end{align}

Next, using Lemma \ref{lem:global_estimates} together with \eqref{eq:diff_gN} and properties of 
$h_{C, \ve}$ in Lemma \ref{lem:approx}, we estimate $\int_{\ve^2}^1 |\tilde{Q}_i(\tau)| \, d \tau$ for $\tilde{Q}_i$ in the decompositions of Section \ref{sec:Ri_decomp}.

\begin{prop}\label{prop:qi_tilde} For $1 \le i \le 5$,
\begin{align}\label{eq:Qi_tilde_estim}
	\int_{\ve^2}^1 | \tilde{Q}_i(\tau) | \, d \tau \le  \bfC 
	(\delta_2 -  \delta_1)^{1 - 3 K_0 / 2} 
	d^3 
	L^5 N
	\Vert b^{-1} \Vert_s^3
	\biggl[  \Vert b^{-1} \Vert_s 
	\ve^{-2}  \cE + 
	\ve^{-1} \cE  + 1 \biggr].
\end{align}
\end{prop}

\begin{proof} Consider $\tilde{Q}_4 = I + II + III$, where $I,II,III$ are defined as in \eqref{eq:Q4_tilde}. Since $\eta^{n,2k, -1}(0, \tau, z) = -\eta^{n,-1, 2k}(0, \tau, z)$, 
by \eqref{eq:global_e1},
\begin{align*}
		\biggl|  \mu \biggl\{   
	\eta^{n,2k, -1}(0, \tau, z)
	\biggr\} \biggr|  \lesssim  (\delta_2 -  \delta_1)^{- 3 K_0 / 2} L^2  \ve^{-1} k^2   \Vert b^{-1} \Vert_s  \cE
\end{align*} 
holds for $k \ge 1$. Hence, applying the second estimate in
\eqref{eq:multiple}, we obtain
\begin{align}\label{eq:I_estim}
	\begin{split}
	|I| &\le \frac{\sqrt{1-\tau} }{2\tau^{3/2}}   \sum_{\delta_1 N \le n < \delta_2 N} 
	\sum_{m=1}^{N-1} 
	\sum_{k=2m+1}^{N-1}  
	\sum_{r,s, t =1}^d  
	\int_{\bR^d}
	\biggl|  \mu \biggl\{   
	\eta^{n,2k, -1}(0, \tau, z)
	\biggr\} \biggr| | \mu( F_{r,s,t}^{n,m,k} ) |
	| \phi_{rst}(z) |   \, dz \\
	&\le  \bfC (\delta_2 -  \delta_1)^{1 - 3 K_0 / 2}  L^5 \Vert b^{-1} \Vert_s^4  \cE    \frac{\sqrt{1-\tau} }{2\tau^{3/2}} \ve^{-1} 
	N
	\sum_{m=1}^{N-1} 
	\sum_{k=2m+1}^{N-1}  k^2  
	q^{k - m}
	\sum_{r,s, t =1}^d  
	\int_{\bR^d}
	| \phi_{rst}(z) |  \, dz \\
	&\le  \bfC d^3 (\delta_2 -  \delta_1)^{1 - 3 K_0 / 2}  L^5 \Vert b^{-1} \Vert_s^4  \ve^{-1} 
	N \cE 
	 \frac{\sqrt{1-\tau} }{2\tau^{3/2}}.
	\end{split}
\end{align}
Note that, by properties of $h = h_{C, \ve}$ in Lemma \ref{lem:approx}, 
\begin{align*}
&| \mu \{   h(\sqrt{1 - 
		\tau} W  - 
	\sqrt{\tau} 
	z)  -  h( \sqrt{1 - \tau } \tilde{Z} - \sqrt{\tau} z )  \} |\\
&\le | 
\mu(  \sqrt{1 - \tau } W - \sqrt{\tau} z \in C^{\ve}   ) -  \mu ( \sqrt{1 - \tau } \tilde{Z} - \sqrt{\tau} z  \in C  ) | \\
&+ |
 \mu ( \sqrt{1 - \tau } \tilde{Z} - \sqrt{\tau} z  \in C^\ve  ) - \mu(  \sqrt{1 - \tau } W - \sqrt{\tau} z \in C  )   |.
\end{align*}
Thus, applying Lemma \ref{lem:normal_approx}, we obtain
\begin{align*}
	&| \mu \{   h(\sqrt{1 - 
		\tau} W  - 
	\sqrt{\tau} 
	z)  -  h( \sqrt{1 - \tau } \tilde{Z} - \sqrt{\tau} z )  \} | \\
	&\lesssim d^{1/4} \frac{\ve}{ \sqrt{1 - \tau } } + \fD \bar{b} ( \delta_2 - \delta_1 )^{-3K_0 / 2} 
	\lesssim ( \delta_2 - \delta_1 )^{-3K_0 / 2} \frac{1}{ \sqrt{1 - \tau } } \cE.
\end{align*}
It follows that 
\begin{align}\label{eq:II_estim}
	\begin{split}
	|II| &\le \frac{\sqrt{1-\tau} }{2\tau^{3/2}} 
	 \sum_{\delta_1 N \le n < \delta_2 N} 
	\sum_{m=1}^{N-1} 
	\sum_{k=2m+1}^{N-1}  
	\sum_{r,s, t =1}^d  
	\int_{\bR^d}
	\biggl| \mu \biggl\{   h(\sqrt{1 - 
		\tau} W  - 
	\sqrt{\tau} 
	z) \\
	&-  h( \sqrt{1 - \tau } \tilde{Z} - \sqrt{\tau} z )  \biggr\} \biggr| | \mu( F_{r,s,t}^{n,m,k} ) |
	| \phi_{rst}(z) |   \, dz \\
	&\le \bfC d^3 ( \delta_2 - \delta_1 )^{1 -3K_0 / 2}   
	L^3  \Vert  b^{-1} \Vert_s^3 N   \cE \frac{ 1 }{2\tau^{3/2}}.
	\end{split}
\end{align}
Finally, by \eqref{eq:diff_gN},
\begin{align*}
	 | \cN_d[ g_{rst}( \cdot, \tau )  ] | \le \frac{ \sqrt{1 - \tau } }{ 2 } \int_{\bR^d} | h(x) || \phi_{rst}(x) | \, dx \le  \frac{ \sqrt{1 - \tau } }{ 2 } \int_{\bR^d} | \phi_{rst}(x) | \, dx 
	 \lesssim 1.
\end{align*}
Hence,
\begin{align}\label{eq:III_estim}
	\begin{split}
	|III| &\le  \sum_{\delta_1 N \le n < \delta_2 N} 
	\sum_{m=1}^{N-1} 
	\sum_{k=2m+1}^{N-1}  
	\sum_{r,s, t =1}^d  
	 | \cN_d[ g_{rst}( \cdot, \tau )  ] | \mu( F_{r,s,t}^{n,m,k} ) | \\
	 &\le \bfC d^3 ( \delta_1 - \delta_1 ) L^3   N \Vert b^{-1} \Vert_s^3.
	 \end{split}
\end{align}

Gathering \eqref{eq:I_estim}, \eqref{eq:II_estim}, and \eqref{eq:III_estim}, and integrating over $\tau \in [\ve^2, 1]$, 
we obtain 
\eqref{eq:Qi_tilde_estim} for $i = 4$. For the other terms 
$\tilde{Q}_i$ with $i \neq 4$, 
 \eqref{eq:Qi_tilde_estim} can be derived 
with a computation almost identical to that for $\tilde{Q}_4$, 
and we therefore omit the details.
\end{proof}

\subsection{Decorrelation bounds} In this section, we establish decorrelation bounds involving $\eta^{n,m,k}$, which will be used to control $\bar{Q}_i$, $R_i'$, $R_i''$, $S_i$.

\begin{lem}\label{lem:decor} Assume \eqref{eq:cc}. Suppose that $0 \le \tau \le 1$, $\ve > 0$, $z \in \bR^d$, $v \in [0,1]$, and $r,s,t \in \{1, \ldots, d\}$.  Set 
	\begin{align*}
	A_*(m, k, \ell) &= ( \delta_2 - \delta_1 )^{ - 3K_0 / 2} 2^{3K_0 / 2}
	C_0^{3/2} 
	\ve^{-1} \cE ( m + k + \ell + 1)^2, \quad 
	q_* = \max \{ q, \Lambda^{- \alpha /p} \},
	\end{align*}
where $\cE$ is defined as in \eqref{eq:E},  and $q \in (0,1)$ is 
as in Theorem \ref{thm:exp_loss}.
Then, the following upper bounds 
hold for all $0 \le n < N$, $k,\ell \ge 0$ and $0 \le \hat{m} \le m$:
\begin{align}
	&\biggl| \mu \biggl\{ 
	\overline{\eta^{n, m + k, m + k + \ell}(v, \tau, z)}  F_{r,s,t}^{n,\hat{m},m} \biggr\}  \biggr|  
	\le  \bfC   L^5  \Vert b^{-1} \Vert_s^4  A_*(m, k, \ell) q_*^{ k / 2 }  , \label{eq:decorr_1} \\ 
	&\biggl| \mu \biggl\{ 
	\eta^{n,m + k, m + k + \ell}(v, \tau, z)  F_{r,s,t}^{n, \hat{m} ,m} \biggr\}  \biggr|  
	\le  \bfC   L^5  \Vert b^{-1} \Vert_s^4 A_*(m, k, \ell) q_*^{ \hat{m} / 2 }  , \label{eq:decorr_2} \\ 
	&\biggl| \mu \biggl\{ 
	\eta^{n,m + k, m + k + \ell}(v, \tau, z)  Y_t^{n,m} \biggr\}  \biggr|  \le  \bfC  
	L^3  \Vert b^{-1} \Vert_s^2
	A_*(m, k, \ell) q_*^{ k / 2 }  ,  \label{eq:decorr_3} \\
	&\biggl| \mu \biggl\{ 
	\overline{   \eta^{n, m + k,m + k + \ell }(v, \tau, z)  Y_t^{n,m + k } }   G_{r,s}^{n,m} \biggr\} \biggr| \le  \bfC L^5  \Vert b^{-1} \Vert_s^4 A_*(m, k, \ell )  q_*^{ k / 2 }.  \label{eq:decorr_4}
\end{align}
	
\end{lem}


\subsubsection{Proof of \eqref{eq:decorr_1}}\label{sec:estimate_r1} 
Let $0 \le n < N$,  $k,\ell \ge 0$, $0 \le \hat{m} \le m$, 
$0 \le \tau \le 1$, $\ve > 0$, $z \in \bR^d$, 
$v \in [0,1]$, and $r,s,t \in \{1, \ldots, d\}$. Recall the definitions of 
$\tilde{\eta}^{n,m,k}(x,y)$, $\tilde{W}^{n,m}(x,y)$,  and $\tilde{Y}^{n,m}(x,y)$ from \eqref{eq:eta_tilde} and 
\eqref{eq:W_tilde}.
We aim to control 
\begin{align*}
	\cI &:=\mu \biggl\{ 
	\overline{\eta^{n,m + k, m + k + \ell}(v, \tau, z)}  F_{r,s,t}^{n, \hat{m},m} \biggr\} \\ 
	&= \int_M   \tilde{\eta}^{n,m + k , m + k + \ell  }(x,x)    F_{r,s,t}^{n,\hat{m},m}(x) \, d \mu(x) 
	- 
	 \int_M  \tilde{\eta}^{n,m + k , m + k + \ell  } (y,y)  \, d \mu(y)   \int_M  F_{r,s,t}^{n,\hat{m},m} (x) \, d \mu(x).
\end{align*}
To exploit the gap between the indices appearing in $\eta^{n,m + k , m + k + \ell  }$ and $F_{r,s,t}^{n, \hat m,m}$, 
we decompose 
\begin{align}\label{eq:decomp_I}
\cI = \cI_1 + \cI_2 + \cI_3,
\end{align}
where 
\begin{align*}
	\cI_1 &=  \int_M  \tilde{\eta}^{n,m + k , m + k + \ell  }(x,x)  F_{r,s,t}^{n,\hat{m},m}(x)\, d \mu(x) \\
	&- 
	\int_{M} \int_M \tilde{\eta}^{n,m + k , m + k + \ell  }(x,y)   \, d \mu(y) \, 
	F_{r,s,t}^{n,\hat{m},m}(x) \, d \mu(x), \\
	\cI_2 &= \int_{M} \int_M   \tilde{\eta}^{n,m + k , m + k + \ell  }(x,y)  \, d \mu(y)  \, F_{r,s,t}^{n, \hat{m},m}(x) \, d \mu(x) \\
	 &-  
		\iint_{M^2}  \tilde{\eta}^{n,m + k , m + k + \ell  }(x,y)  \, d \mu(x) \, d \mu(y) 
		\int_M  F_{r,s,t}^{n,\hat{m},m}(x') \, d \mu(x'), \\ 
	\cI_3 &=  \iint_{M^2}   \tilde{\eta}^{n,m + k , m + k + \ell  } (x,y)  \, d \mu(x) \, d \mu(y) 
	\int_M  F_{r,s,t}^{n,\hat{m},m}(x')  \, d \mu(x')  \\
	&-
	 \int_M  \tilde{\eta}^{n,m + k , m + k + \ell  } (y,y)  \, d \mu(y)   \int_M  F_{r,s,t}^{n,\hat{m},m}(x')  \, d \mu(x').
\end{align*}
Each of these three terms can be controlled through a similar procedure consisting of two steps, which we carry out in detail for $\cI_1$ in 
what follows.

\noindent\textbf{$\cI_1$ -- Step 0.} We decompose the integral $\int_M$ using the partition $\cA( \cT_{i}  )$ induced by a suitable iterate 
$\cT_i$. After that we replace  $F_{r,s,t}^{n,\hat{m},m}$ in  $\cI_1$ by a constant on each $a \in \cA( \cT_i )$ and estimate the error using (UE:1). Without loss of generality, we will assume that $\mu(a) > 0$ for all $a \in \cA( \cT_i )$. Otherwise we can replace $\cA( \cT_i )$ with $\cA_*( \cT_i ) = \{ a \in \cA( \cT_i ) \: : \: \mu(a) > 0\}$.

In the case of $\cI_1$, we set
$
i = i(n,m, k) =    \lceil n + m + k/2 \rceil ,
$
and decompose 
\begin{align*}
&\cI_1 
=   \sum_{a \in  \cA (  \cT_i   ) } \mu(a)  \int_a \biggl[  \tilde{\eta}^{n,m + k , m + k + \ell  } (x,x)   - \int_M  \tilde{\eta}^{n,m + k , m + k + \ell  }(x,y)  \, d \mu(y) \biggr]
F_{r,s,t}^{n, \hat{m} ,m}(x)  \, d \mu_a(x),
\end{align*}
where $\mu_a$ denotes the probability measure with density $\rho_a = \mathbf{1}_a \rho / \mu(a)$.  

By (UE:1), for any $x, y \in a \in \cA( \cT_i )$, any $0 \le j \le n + m$, and any $r \in \{1, \ldots, d\}$, we have 
\begin{align}\label{eq:error_phi}
|X_r^j(x) - X_r^j(y) | \le L \Lambda \Lambda_1^{- k/2},
\end{align}
where $\Lambda_1 = \Lambda^{ \alpha /p} > 1$.
Fix $c_a \in a$ for each $a \in \cA( \cT_i)$.
Combining \eqref{eq:global_e1}, \eqref{eq:global_e2} and \eqref{eq:error_phi}, we obtain 
\begin{align*}
	\cI_1 &=  \sum_{a \in \cA( \cT_i) } \mu(a)  \int_a \biggl[  \tilde{\eta}^{n, m + k, m + k + \ell }(x,x)   - \int_M \tilde{\eta}^{n, m + k, m + k + \ell }(x,y)  \, d \mu(y) \biggr]
	F_{r,s,t}^{n, \hat{m} ,m}(x)  \, d \mu_a(x)  \notag \\ 
	&=  \sum_{a \in \cA( \cT_i) }  F_{r,s,t}^{n,\hat{m},m}(c_a)  \mu(a)  \int_a \biggl[ \tilde{\eta}^{n, m + k, m + k + \ell }(x,x)   - \int_M \tilde{\eta}^{n, m + k, m + k + \ell }(x,y)  \, d \mu(y) \biggr]
	 \, d \mu_a(x)  \notag \\ 
	&+ O \biggl( L^5  \Vert b^{-1} \Vert_s^4  ( \delta_2 - \delta_1 )^{ - 3K_0 / 2}
	  2^{3K_0/2}C_0^{3/2}  \ve^{-1}
	  \cE   (m + k +  \ell + 1)^2   \Lambda   \Lambda_1^{- k/2} 
	 \biggr),
\end{align*}
where $\cE = d^{1/4}  (  \ve +  \Vert b^{-1} \Vert_s   ) + \bar{b}
\fD$ and the constant in the error term is absolute.

\noindent\textbf{$\cI_1$ -- Step 1.} By essentially repeating the argument from Lemma \ref{lem:global_estimates}, we approximate 
$$
 \tilde{\eta}^{n, m + k, m + k + \ell } (x,x) \approx \int_{a}  \tilde{\eta}^{n, m + k, m + k + \ell }  (x',x) \, d\mu_a(x')
$$
for $x \in a \in \cA(\cT_i)$. More precisely, we have the following estimates.

\begin{claim}\label{claim:approx_integral_1} Set 
$$
B_1(x', x) = h(\sqrt{1 - 
	\tau}(  \tilde{W}^{n,m + k}(x',x)   + v( \tilde{Y}^{n,m + k }(x', x)  ) - 
\sqrt{\tau}
z)
$$
and 
$$
B_2(x', x) = h(  \sqrt{1 - \tau} \tilde{W}^{n,m + k + \ell }(x', x)  - \sqrt{\tau} z  ),
$$
so that 
$
\tilde{\eta}^{n, m + k, m + k + \ell}(x',x) = B_1(x', x) - B_2(x', x).
$
Then, 
\begin{align}\label{eq:approx_integral_1-1}
	\begin{split}
	&\sum_{a \in \cA( \cT_i )} 
	\mu(a) \int_{a}  \biggl| 
	B_1(x, x)   -  \int_{a} B_1(x',x) \, d \mu_a(x')  
	\biggr|    \, d \mu_a(x) \\ 
	&\le   \bfC  L^2  \Vert b^{-1} \Vert_s (\delta_2 -  \delta_1)^{- 3 K_0 / 2}  \ve^{-1} \cE
	(m + k + 1) 
	\Lambda_1^{-k - m},
	\end{split}
 \end{align}
and
\begin{align}\label{eq:approx_integral_1-2}
	\begin{split}
	&\sum_{a \in \cA( \cT_i ) } 
	\mu(a) \int_{a}  \biggl| 
	B_2(x, x)   -  \int_{a} B_2(x',x) \, d \mu_a(x')  
	\biggr|    \, d \mu_a(x)  \\ 
	&\le \bfC L^2 \Vert b^{-1} \Vert_s   (\delta_2 -  \delta_1)^{- 3 K_0 / 2}    \ve^{-1}  \cE     (m + k +  \ell + 1) \Lambda_1^{-k - m}.
	\end{split}
\end{align}
\end{claim}

\begin{proof}[Proof of Claim \ref{claim:approx_integral_1}] Note that, if $x,x' \in a \in \cA( \cT_i )$, $y \in M$, it follows by 
(UE:1) that 
$$
B_1(x, x) - B_1(x', x) = h( w_1 ) - h(w_2)
$$
where
\begin{align}\label{eq:replace_const}
\begin{split}
&\Vert w_1 - w_2 \Vert \\
&= 
\Vert
\sqrt{1 - 
	\tau}(  \tilde{W}^{n,m + k}(x,x)   + v( \tilde{Y}^{n,m + k }(x, x)  ) - 
\sqrt{\tau}
z
\\  
&- (
\sqrt{1 - 
	\tau}(  \tilde{W}^{n,m + k}(x',x)   + v( \tilde{Y}^{n,m + k }(x', x)  ) - 
\sqrt{\tau}
z ) \Vert 
  \\ 
&\le  \Vert  Y_{-}^{n, m+k}(x)  -  Y_{-}^{n, m+k}(x') \Vert + 
 \Vert W^{n,m + k}_{-}(x) - W^{n,m + k}_{-}(x')   \Vert 
\\ 
&\le   \Vert b^{-1} \Vert_s  L  \Lambda \Lambda_1^{ - k - m}   
+  \Vert b^{-1} \Vert_s  L \mathbf{1}_{n - m - k> 0} \sum_{j=0}^{ \round{ ( n - m - k)/p}  } 
\sum_{  jp \le  q  < (j + 1)p }  d( \cT_q x, \cT_q x')^{\alpha }   \\ 
&\le   \Vert b^{-1} \Vert_s L  \Lambda \Lambda_1^{ - k - m}    +   \mathbf{1}_{n - m - k > 0}  \Vert b^{-1} \Vert_s L   \sum_{j=0}^{ \round{ ( n - m - k)/p}  }  p \Lambda^{  
  -  \alpha \round{(n + m + k/2)/ p } + \alpha  j  }  
 \\
 &\lesssim  
  \Vert b^{-1} \Vert_s  \frac{ L  \Lambda p }{ 1 - \Lambda^{-1}_1 } \Lambda_1^{ - k - m}.
 \end{split}
\end{align}
Using \eqref{eq:replace_const} and arguing as in the proof of Lemma \ref{lem:global_estimates}, we see that 
\begin{align}\label{eq:rel_a}
	&A := \biggl\{  x \in a \: : \:   B_1(x,x)  
	- \int_a  B_1 (x', x) \, d \mu_a(x') \neq 0  \biggr\} \subset 
	 a \cap \{  W \in  D^{ c_1 } \setminus D^{ - c_2  } \},
\end{align}
where $D = D(\tau, z) \in \cC$, and, for some absolute constant $c > 0$,
\begin{align*}
	c_1 &= \frac{\ve}{ \sqrt{1 - \tau} } + c \Vert b^{-1} \Vert_s   \frac{L \Lambda  p}{  1 - \Lambda^{-1}_1 }  + c ( k + m + 1	) L \Vert b^{-1} \Vert_s, \\
	c_2 &= c \Vert b^{-1} \Vert_s 
	\frac{  L \Lambda  p }{    1 - \Lambda^{-1}_1 }  + c ( k + m + 1)  L \Vert b^{-1} \Vert_s.
\end{align*}
In combination with Lemma \ref{lem:normal_approx}, \eqref{eq:replace_const}, and Lemma \ref{lem:approx}-(iv), \eqref{eq:rel_a} yields 
\begin{align*}
	&\sum_{a} \mu(a) \int_a \biggl| 
	 B_1(x,x) - \int_a  B_1(x', x) d \mu_a(x')
	\biggr| \, d \mu_a(x) \\ 
	&\le \bfC L  \ve^{-1}  \Vert b^{-1} \Vert_s  \Lambda_1^{-k - m}  \sqrt{1- \tau }  \mu( W \in  D^{ c_1 } \setminus D^{ - c_2  }  )  \\ 
	&\le \bfC L \ve^{-1}  \Vert b^{-1} \Vert_s  \Lambda_1^{-k - m} \biggl\{ 
	d_c (  \cL(W), \cN_d  ) +  \sqrt{1- \tau }  d^{1/4}(c_1 + c_2)
	\biggr\} \\ 
	&\le \bfC (\delta_2 -  \delta_1)^{- 3 K_0 / 2} L^2\ve^{-1}  \Vert b^{-1} \Vert_s   (k + m + 1) \Lambda_1^{-k - m} \biggl\{ 
       \bar{b} \fD +   d^{1/4} ( \ve +   \Vert b^{-1}  \Vert_s  ) 	\biggr\}.
\end{align*}
This establishes \eqref{eq:approx_integral_1-1}. We can obtain  \eqref{eq:approx_integral_1-2} in a similar way, replacing
$\tilde{W}^{n,m + k}   + v \tilde{Y}^{n,m + k }$ with $\tilde{W}^{n,m + k + \ell}$ in the preceding proof.
\end{proof}

By Claim \ref{claim:approx_integral_1} and the estimate established in Step 0,
\begin{align*}
	\cI_1 = \cI_1' + O \biggl( \bfC L^5  \Vert b^{-1} \Vert_s^4  ( \delta_2 - \delta_1 )^{ - 3K_0 / 2}
	2^{3K_0/2}C_0^{3/2}  \ve^{-1}
	\cE   ( m + k +  \ell + 1 )^2    \Lambda_1^{- k/2} 
	\biggr),
\end{align*}
where
\begin{align*}
	\cI_1' &= \sum_{a  }  F_{r,s,t}^{n, \hat{m},m}(c_a)  \mu(a)  \int_a \biggl[   \int_a  \tilde{\eta}^{n, m + k, m + k + \ell } (x',x) \, d \mu_a(x')  \\
	&- \int_M \tilde{\eta}^{n, m + k, m + k + \ell } (x,y)  \, d \mu(y) \biggr]
	\, d \mu_a(x).
\end{align*}

\noindent\textbf{$\cI_1$ -- Step 2.} We replace each conditional 
measure $\mu_a$, $a \in \cA( \cT_i )$, with the measure $\mu$ in $\cI_1'$ 
and estimate the resulting error using Corollary \ref{cor:ml_cond}.

Let $a \in \cA(\cT_i)$, and denote by $\hat{\eta}^{n,m + k, m + k + \ell}(x,y)$ the function that 
satisfies $$
\hat{\eta}^{n,m + k, m + k + \ell}(x, \cT_{n + m  + k  } y) = \tilde{\eta}^{n,m + k , m + k + \ell}(x,y).
$$
Recall that $\rho_a = \mathbf{1}_a \rho / \mu(a)$, where $\rho$ 
is the density of $\mu$. If $n + m + k \ge N$, then $\cI_1' = 0$. Otherwise, by Corollary \ref{cor:ml_cond},
\begin{align}\label{eq:dec_replace_cond}
	\begin{split} 
	&\biggl| \int_M \biggl[  \int_a  \tilde{\eta}^{n, m + k, m + k + \ell }(x',x)  \, d \mu_a(x')  \\
	&- \int_M \int_a \tilde{\eta}^{n, m + k, m + k + \ell }(x',y) \, d \mu_a(x')  \, d \mu(y) \biggr]  \,  (\rho_a - \rho ) d \lambda(x) \biggr|  \\ 
	&= \biggl| \int_M \biggl[  \int_a \hat{\eta}^{n,m + k, m + k + \ell}(x',x) \, d \mu_a(x') \\
	 &- \int_M \int_a  
	\tilde{\eta}^{n, m + k, m + k + \ell }(x',y)
	\, d \mu_a(x')  \, d \mu(y) \biggr] \circ \cT_{n + m + k}  \,  (\rho_a - \rho ) d \lambda(x) \biggr|  \\ 
	&\le \biggl[   \int_M \int_a  |   \hat{\eta}^{n,m + k, m + k + \ell}(x',x)  | \, d \mu_a(x')  \, d  \lambda (x)  \\
	&+ 
	 \int_M \int_a  |  \tilde{\eta}^{n,m + k, m + k + \ell}(x',y)  | \, d \mu_a(x')  \, d \mu(y) 
	  \biggr]  \Vert  \cP_{n +  m + k - i + i } (  \rho_a - \rho  )   \Vert_\infty   \\ 
	&\le \bfC  \frac{1}{  \inf_M   \cP_{ n + m + k }( \rho )  } \int_M \int_a | \eta^{n,m + k, m + k + \ell }(x,y)| \, d \mu_a(x) \, d \mu(y)  
	q^{ k/2 }   \\ 
	&\le  \bfC  \int_M \int_a | \eta^{n,m + k, m + k + \ell }(x,y)| \, d \mu_a(x) \, d \mu(y)  q^{ k/2 },
	\end{split}
\end{align}
where \eqref{eq:return_to_cone_2} was used in the last inequality.
Consequently, using \eqref{eq:global_e2}, we obtain 
\begin{align*}
	| \cI_1' | &\le  \bfC  L^3 \Vert b^{-1} \Vert_s^3 \sum_{a} \mu(a)  
	  \int_M \int_a | \eta^{n,m + k , m  + k + \ell  }(x,y)|  \, d \mu_a(x) \, d \mu(y) q^{ k / 2 } \\ 
	 &\le  \bfC L^5  \Vert b^{-1} \Vert_s^4  ( \delta_2 - \delta_1 )^{ - 3K_0 / 2}
	 2^{3K_0/2}C_0^{3/2}  \ve^{-1}
	 \cE   ( m + k +  \ell + 1 )^2   q^{k/2}.
\end{align*}
We have established the estimate 
\begin{align}\label{eq:bound_on_I1}
	\begin{split}
	|\cI_j| &\le \bfC L^5  \Vert b^{-1} \Vert_s^4  ( \delta_2 - \delta_1 )^{ - 3K_0 / 2}
	2^{3K_0/2}C_0^{3/2}  \ve^{-1}
	\cE   ( m + k +  \ell + 1)^2   q_*^{k/2}  \\
	&= \bfC L^5 \Vert b^{-1} \Vert_s ^4 
	A_*(m, k, \ell ) q^{ k / 2 }_*
	\end{split}
\end{align}
for $j = 1$, where $q_* = \max \{ q, \Lambda^{- \alpha /p} \}$.

\noindent\textbf{Estimates on $\cI_2$ and $\cI_3$.} 
Since the remaining terms $\cI_2$ and $\cI_3$ can be treated in a manner similar to $\cI_1$, 
we provide only an outline of the approach to deriving \eqref{eq:bound_on_I1} for $j \in \{2,3\}$.

In the case of $\cI_2$, we have $\cI_2 = 0$ if $n - m - k< 0$. Otherwise, we set
$
i =    \lceil  n - m - k/2  \rceil,
$
and decompose 
\begin{align*}
	&\int_{M} \int_M  \tilde{\eta}^{n, m + k, m + k + \ell } (x,y)  \, d \mu(y)  \, F_{r,s,t}^{n,\hat{m},m}(x) \, d \mu(x) \\ 
	&= \int_M  \sum_{a \in \cA( \cT_i )} \mu(a) \int_a F_{r,s,t}^{n,\hat{m},m}(x)  \int_M \tilde{\eta}^{n, m + k, m + k + \ell } (x,y)  \, d \mu(y)  \, d \mu_a(x) .
\end{align*}
As in the case of $\cI_1$, we approximate $$\int_{M} \tilde{\eta}^{n, m + k, m + k + \ell }(x,y)  \, d \mu(y)  \approx \int_{M} \int_{a} \tilde{\eta}^{n, m + k, m + k + \ell }(x',y) \, d \mu_a(x')  \, d \mu(y)$$ 
on each $a \in \cA( \cT_i )$. For this, we use the following counterpart of Claim \ref{claim:approx_integral_1}.

\begin{claim}\label{claim:approx_integral_2} Define $B_1,B_2$ as in Claim \ref{claim:approx_integral_1} so that
$$
\tilde{\eta}^{n, m + k, m + k + \ell}(x',x) = B_1(x', x) - B_2(x', x).
$$
Then,  
	\begin{align*}
		&\sum_{a \in \cA( \cT_i ) } 
		\mu(a) \int_{a}  \biggl| 
		 \int_M B_1(x, y) \, d \mu(y)   -   \int_M \int_{a} B_1(x',y) \, d \mu_a(x')  \, d \mu(y)
		\biggr|    \, d \mu_a(x) \notag \\ 
		&\le  \bfC L^2  \Vert b^{-1} \Vert_s (\delta_2 -  \delta_1)^{- 3 K_0 / 2}   2^{3K_0 / 2}
		C_0^{3/2}   \ve^{-1}   \cE (k + m + 1) \Lambda_1^{-k / 2},
	\end{align*}
	and
	\begin{align*}
		&\sum_{a \in  \cA( \cT_i )  } 
		\mu(a) \int_{a}  \biggl| 
		 \int_M B_2(x, y) \, d \mu(y)   -  \int_{M} \int_a  B_2(x',y) \, d \mu_a(x')  \, d \mu(y)
		\biggr|    \, d \mu_a(x)  \notag \\ 
		&\le  \bfC L^2 \Vert b^{-1} \Vert_s 
		(\delta_2 -  \delta_1)^{- 3 K_0 / 2} 2^{3K_0 / 2}
		C_0^{3/2}  \ve^{-1}   \cE (k + m + \ell 
		+ 1) \Lambda_1^{-k / 2 - \ell}.
	\end{align*}
\end{claim}

\begin{proof}[Proof of Claim \ref{claim:approx_integral_2}] The result can be established by 
	arguing as in the proof of Claim \ref{claim:approx_integral_1}, and then conducting  
	a case-by-case analysis depending on the value of $\delta = n/N$, as in the proof of \eqref{eq:global_e2}.
	The multiplicative constant $2^{3K_0 / 2}
	C_0^{3/2} $ arises as a consequence of the latter step.
	Details are left to the reader.
\end{proof}

By Claim \ref{claim:approx_integral_2}, we have 
\begin{align*}
		&\iint_{M^2}  \tilde{\eta}^{n, m + k, m + k + \ell } (x,y)  \, d \mu(y)  \, F_{r,s,t}^{n,\hat{m},m}(x) \, d \mu(x) \notag  \\ 
	&= \cI_2' + O \biggl(   
	\bfC L^5\Vert b^{-1} \Vert_s^4 
	 (\delta_2 -  \delta_1)^{- 3 K_0 / 2}   2^{3K_0 / 2}
	 C_0^{3/2}   \ve^{-1}   \cE
	 (m + k +  \ell + 1) \Lambda_1^{-k / 2  } 
	\biggr),
\end{align*}
where
$$
\cI_2' = \sum_{a \in \cA( \cT_i )  } \mu(a) \int_a F_{r,s,t}^{n,\hat{m},m}(x')\, d \mu_a(x')  \int_M \int_a \tilde{\eta}^{n, m + k, m + k + \ell }(x,y) \, d \mu_a(x) \, d \mu(y).
$$
Moreover, by Corollary \ref{cor:ml_cond}, 
\begin{align*}
	&\biggl| \int_a F_{r,s,t}^{n,\hat{m},m} \, d \mu_a - \int_M F_{r,s,t}^{n,\hat{m},m} \, d \mu \biggr| \lesssim  L^3 \Vert b^{-1} \Vert_s^3 
	\Vert \cP_{n-m - i + i } (  \rho_a - \rho  ) \Vert_{L^1(\lambda)} 
	\le \bfC   L^3    \Vert b^{-1} \Vert_s^3  q^{ k/2 }.
\end{align*} 
Therefore, 
\begin{align*}
	\cI_2' &= \int_M F_{r,s,t}^{n,\hat{m},m}(x')  \, d \mu(x')  \iint_{M^2} \tilde{\eta}^{n, m + k, m + k + \ell } (x,y)  \, d \mu(x) \, d \mu(y)  \\ 
	&+ O \biggl( 
	  \bfC L^3   \Vert b^{-1} \Vert_s^3  q^{ k/2 }  \iint_{M^2} |  \tilde{\eta}^{n, m + k, m + k + \ell } (x,y)  | \, d \mu(x) \, d \mu(y)  
	\biggr) \\ 
	&=   \int_M F_{r,s,t}^{n,\hat{m},m}(x')  \, d \mu(x')  \iint_{M^2} \tilde{\eta}^{n, m + k, m + k + \ell } (x,y) \, d \mu(x) \, d \mu(y) \\ 
	&+ O \biggl(  \bfC L^5  \Vert b^{-1} \Vert_s^4 ( \delta_2 - \delta_1 )^{ - 3K_0 / 2}
	 2^{3K_0 / 2}
	 C_0^{3/2} \ve^{-1} \cE   (m + k +  \ell + 1)^2   q^{ k/2 }  
	\biggr),
\end{align*}
where \eqref{eq:global_e2} was used to obtain the last equality. Consequently, \eqref{eq:bound_on_I1} holds for $j = 2$.

Finally, for $\cI_3$, we set $i = n$, and once more decompose 
\begin{align*}
	\int_{M} \tilde{\eta}^{n, m + k, m + k + \ell } (x,x)  \, d \mu(x) = \sum_{ a \in \cA( \cT_i ) } \int_{a} \tilde{\eta}^{n, m + k, m + k + \ell } (x,x) \, d \mu_a(x).
\end{align*}
As in the proof of Claim \ref{claim:approx_integral_1}, we obtain 
\begin{align*}
	&\sum_{a \in \cA( \cT_i )  } 
	\mu(a) \int_{a}  \biggl| 
	\tilde{\eta}^{n, m + k, m + k + \ell } (x,x) -  \int_{a} \tilde{\eta}^{n, m + k, m + k + \ell } (x',x)  \, d \mu_a(x')  
	\biggr|    \, d \mu_a(x) \\ 
	&\le  \bfC L^2 \Vert b^{-1} \Vert_s   (\delta_2 -  \delta_1)^{- 3 K_0 / 2}   \ve^{-1}    \cE   (m + k +  \ell + 1) \Lambda_1^{-k - m} ,
\end{align*}
so that
\begin{align}\label{eq:remains_I3}
	\begin{split}
	&\int_{M}  \tilde{\eta}^{n, m + k, m + k + \ell } (x,x)  \, d \mu(x) 
	\\
	&= \sum_{ a \in \cA( \cT_i) }  \mu(a ) \int_{a} \int_{a}   \tilde{\eta}^{n, m + k, m + k + \ell } (x',x)  \, d \mu_a(x') \, d \mu_a(x) \\ 
	&+ O \biggl( 
	 \bfC L^2 \Vert b^{-1} \Vert_s   (\delta_2 -  \delta_1)^{- 3 K_0 / 2}   \ve^{-1}    \cE   (m + k +  \ell + 1) \Lambda_1^{-k - m} 
	\biggr).
	\end{split}
\end{align}
Building upon \eqref{eq:dec_replace_cond} and \eqref{eq:global_e2},  we find that 
\begin{align*}
&\sum_{ a \in \cA( \cT_i) }  \mu(a ) \int_{a} \int_{a}   \tilde{\eta}^{n, m + k, m + k + \ell } (x',x)  \, d \mu_a(x') \, d \mu_a(x) \\
&= \iint_{M^2} \eta^{n, m + k, m + k + \ell } (x,y)  \, d \mu(x) \, d  \mu(y) \\ 
&+ O \biggl( \bfC L^2 \Vert b^{-1} \Vert_s 
 ( \delta_2 - \delta_1 )^{ - 3K_0 / 2}   2^{3K_0 / 2}
 C_0^{3/2} 
 \ve^{-1} \cE ( k + m + \ell + 1)^2 q^{ k / 2 }
\biggr).
\end{align*}
Hence, \eqref{eq:bound_on_I1} holds for $j = 3$. This completes the proof of \eqref{eq:decorr_1}.

\subsubsection{Proofs of \eqref{eq:decorr_2}, \eqref{eq:decorr_3}, and \eqref{eq:decorr_4}}  
The proofs of \eqref{eq:decorr_2}, \eqref{eq:decorr_3}, and \eqref{eq:decorr_4} are similar to that of \eqref{eq:decorr_1}, with the only notable difference being the way in which the iterated integrals in the decomposition \eqref{eq:decomp_I} are factored. The integral in \eqref{eq:decorr_3} differs from that in \eqref{eq:decorr_1} only in that the factor \( F_{r,s,t}^{n, \hat{m},m} \) is replaced with \( Y_{t}^{n,m} \). Thus, the proof of \eqref{eq:decorr_3} is almost identical to that of \eqref{eq:decorr_1}, and we omit it to avoid repetition. Below, we provide an outline of the proofs of \eqref{eq:decorr_2} and \eqref{eq:decorr_4}.
 
 In the case of  \eqref{eq:decorr_2}, we exploit the gap 
 of size $\hat{m}$
 between the indices in
 $Y_r^n$ and $\eta^{n,m + k, m + k + \ell} H_{s,t}^{n, \hat{m},m}$, with $H_{s,t}^{n, \hat{m},m} = Y_s^{n,\hat{m}} Y_t^{n,m}$, by decomposing 
\begin{align*}
	&\int_M
	\tilde{\eta}^{n,m + k, m + k + \ell}(x,x)  F_{r,s,t}^{n, \hat{m},m}(x) \, d \mu(x) \\
	&=  \int_M Y_r^n(x)  
	\tilde{\eta}^{n,m + k, m + k + \ell}(x,x)   H_{s,t}^{n, \hat{m},m}(x)  \, d \mu(x) = \cK_1 + \cK_2,
\end{align*}
where 
\begin{align*}
	\cJ_1 &=  \int_M Y_r^n(x)  
	\tilde{\eta}^{n,m + k, m + k + \ell}(x,x)   \tilde{H}_{s,t}^{n, \hat{m}, m}(x,x) \, d \mu(x)  \\
	&- \int_{M} Y_r^n(x) \int_M \tilde{\eta}^{n,m + k, m + k + \ell}(x,y)  \tilde{H}_{s,t}^{n, \hat{m}, m}(x,y) \, d \mu(y)    \, d \mu(x) \\ 
	\cJ_2& = \int_{M} Y_r^n(x) \int_M \tilde{\eta}^{n,m + k, m + k + \ell}(x,y)  \tilde{H}_{s,t}^{n, \hat{m}, m}(x,y) \, d \mu(y)    \, d \mu(x)  \\ 
	&-  \int_{M} Y_r^n(x)  \, d \mu(x)   \iint_{M^2}  \tilde{\eta}^{n,m + k, m + k + \ell}(x',y)  \tilde{H}_{s,t}^{n, \hat{m}, m}(x',y)    \, d \mu(x') \, d \mu(y),
\end{align*}
and we have used the notation 
$$
\tilde{H}_{s,t}^{n,\hat{m}, m}(x,y) =   \tilde{Y}_s^{n, \hat{m} }(x,y) \tilde{Y}_t^{n,m}(x,y).
$$
Note that, since $\mu(Y^n) = 0$, the last term in the expression of $\cJ_2$ vanishes.
With only minor modifications, we can carry out the procedure used to estimate
$\cI_1, \cI_2, \cI_3$ in the proof of \eqref{eq:decorr_1} to the terms $\cJ_1$ and $\cJ_2$, yielding \eqref{eq:decorr_2}.

 In the case of  \eqref{eq:decorr_4}, we exploit the gap of size $k$
 between the indices in
$G_{r,s}^{n,m}$ and $\eta^{n, m + k,m + k + \ell } Y_t^{n,m + k }$  by decomposing 
\begin{align*}
&\mu \biggl\{ 
\overline{   \eta^{n, m + k,m + k + \ell }(v, \tau, z)  Y_t^{n,m + k } }   G_{r,s}^{n,m} \biggr\} \\
&= \int_M   \tilde{\eta}^{n, m + k,m + k + \ell }(x,x)  \tilde{Y}_t^{n,m + k }(x,x) G_{r,s}^{n,m}(x) \, d \mu(x) \\
&- \int_M   \tilde{\eta}^{n, m + k,m + k + \ell }(x',x')  \tilde{Y}_t^{n,m + k }(x',x') \, d \mu(x')  \int_M G_{r,s}^{n,m}(x) \, d \mu(x) = \cK_1 + \cK_2 + \cK_3, 
\end{align*}
where 
\begin{align*}
	\cK_1 &= \int_M   \tilde{\eta}^{n, m + k,m + k + \ell }(x,x)  \tilde{Y}_t^{n,m + k }(x,x) G_{r,s}^{n,m}(x) \, d \mu(x) \\
	&- \int_{M} \int_M  \tilde{\eta}^{n, m + k,m + k + \ell }(x,y)  \tilde{Y}_t^{n,m + k }(x,y)  \, d \mu(y)  G_{r,s}^{n,m}(x)   \, d \mu(x), \\
	\cK_2 &= \int_{M} \int_M  \tilde{\eta}^{n, m + k,m + k + \ell }(x,y)  \tilde{Y}_t^{n,m + k }(x,y)  \, d \mu(y)  G_{r,s}^{n,m}(x)   \, d \mu(x) \\
	&- \int_{M} \int_M  \tilde{\eta}^{n, m + k,m + k + \ell }(x',y)  \tilde{Y}_t^{n,m + k }(x',y)  \, d \mu(x') \, d \mu(y) \int_M  G_{r,s}^{n,m}(x)   \, d \mu(x),
	\end{align*}
	\begin{align*}
	\cK_3 &= \int_{M} \int_M  \tilde{\eta}^{n, m + k,m + k + \ell }(x',y)  \tilde{Y}_t^{n,m + k }(x',y)  \, d \mu(x') \, d \mu(y) \int_M  G_{r,s}^{n,m}(x)   \, d \mu(x) \\
	&- \int_M   \tilde{\eta}^{n, m + k,m + k + \ell }(x',x')  \tilde{Y}_t^{n,m + k }(x',x') \, d \mu(x')  \int_M G_{r,s}^{n,m}(x) \, d \mu(x).
\end{align*}
To obtain \eqref{eq:decorr_4}, we can again use the procedure from the proof of \eqref{eq:decorr_1} to control each of these three terms.

\subsection{Estimates on $\int_{\ve^2}^1  | R_i(\tau) | \, d \tau$} Starting from the decompositions 
established in Section \ref{sec:Ri_decomp}, using Proposition \ref{prop:qi_tilde} together 
with Lemmas \ref{lem:decor} and \ref{lem:global_estimates}, 
it is now straightforward to verify that
\begin{align}\label{eq:r_upper}
\int_{\ve^2}^1 |R_i(\tau)| \, d \tau \le 
\bfC (\delta_2 -  \delta_1)^{1 - 3 K_0 / 2} d^3 2^{3K_0 / 2}
C_0^{3/2} L^5  N \Vert b^{-1} \Vert_s^3  
\biggl[  \Vert b^{-1} \Vert_s 
\ve^{-2}  \cE + 
\ve^{-1} \cE  + 1 \biggr]
\end{align}
holds for all $1 \le i \le 7$. To obtain this in the case of $R_4$, recall 
 from \eqref{eq:decomp_R1} that
$$
|R_4| \le |R_4'| + |\bar{Q}_4| + |\tilde{Q}_4| + |S_4|,
$$
where, by Proposition \ref{prop:qi_tilde}, 
$$
\int_{\ve^2}^1 | \tilde{Q}_i(\tau) | \, d \tau \le  \bfC (\delta_2 -  \delta_1)^{1 - 3 K_0 / 2} d^3 L^5  N \Vert b^{-1} \Vert_s^3  
\biggl[  \Vert b^{-1} \Vert_s 
\ve^{-2}  \cE + 
\ve^{-1} \cE  + 1 \biggr].
$$
By \eqref{eq:decorr_4},
\begin{align*}
\int_{\ve^2}^1 |R_4'(\tau)|  \, d \tau &\le  \int_{\ve^2}^1 \int_0^1  \frac{\sqrt{1-\tau} }{2\tau^{3/2}}   \sum_{n=0}^{N-1} 
\sum_{m=1}^{N-1} 
\sum_{k=2m+1}^{N-1}  
\sum_{r,s, t =1}^d  \int_{\bR^d} \biggl|    \mu \biggl\{ 
\overline{ \eta^{n,k, 2k}(v, \tau, z)
	Y_{t}^{n,k}}   G_{r,s}^{n,m}  \biggr\}  \biggr| \\
	&\times | \phi_{rst}(z)|  \, dz  \, dv \, d \tau  \\
&\le \int_{\ve^2}^1 \int_0^1  \frac{\sqrt{1-\tau} }{2\tau^{3/2}}   \sum_{\delta_1 N \le n < \delta_2 N } 
\sum_{m=1}^{N-1} 
\sum_{k=2m+1}^{N-1}  
\sum_{r,s, t =1}^d  \int_{\bR^d}   \bfC L^5  \Vert b^{-1} \Vert_s^4 A_*(m, k - m, k )  \\
&\times q_*^{ (k-m) / 2 } | \phi_{rst}(z) |  \, dz  \, dv  \, d \tau \\
&\le \bfC  ( \delta_2 - \delta_1 )^{  1 - 3K_0 / 2}  d^3 2^{3K_0 / 2}
C_0^{3/2}  L^5 N \Vert b^{-1} \Vert_s^4   \ve^{-2} \cE.
\end{align*}
By \eqref{eq:decorr_1},
\begin{align*}
\int_{\ve^2}^1 | \bar{Q}_4(\tau)|	 \, d \tau &\le \int_{\ve^2}^1  \frac{\sqrt{1-\tau} }{2\tau^{3/2}}   \sum_{n=0}^{N-1} 
	\sum_{m=1}^{N-1} 
	\sum_{k=2m+1}^{N-1} \sum_{ \ell = 2 }^{N - 1}  
	\sum_{r,s, t =1}^d  
	\int_{\bR^d}
	\biggl| \mu \biggl\{  
	\overline{ \eta^{ n, k \ell, k( \ell + 1) }(0, \tau, z) }  F_{r,s,t}^{n,m,k}
	\biggr\} \biggr| \\     
	&\times |\phi_{rst}(z)|  \, dz \, d \tau \\
	&\le  \int_{\ve^2}^1  \frac{\sqrt{1-\tau} }{2\tau^{3/2}}   \sum_{\delta_1 N \le n < \delta_2 N } 
	\sum_{m=1}^{N-1} 
	\sum_{k=2m+1}^{N-1} \sum_{ \ell = 2 }^{N - 1}  
	\sum_{r,s, t =1}^d  
	\int_{\bR^d} \bfC   L^5  \Vert b^{-1} \Vert_s^4  A_*( k , k ( \ell - 1 )  , k) \\
	&\times q_*^{  k (  \ell  - 1) / 2 }  |\phi_{rst}(z)|  \, dz \, d \tau \\
	&\le \bfC  ( \delta_2 - \delta_1 )^{  1 - 3K_0 / 2}  d^3 2^{3K_0 / 2}
	C_0^{3/2}  L^5 N \Vert b^{-1} \Vert_s^4   \ve^{-2} \cE,
\end{align*}
and, by \eqref{eq:decorr_3},
\begin{align*}
\int_{\ve^2}^1 |S_4(\tau)| \, d \tau 
&\le \int_{\ve^2}^1
\frac{\sqrt{1-\tau} }{2\tau^{3/2}}   \sum_{n=0}^{N-1} 
\sum_{m=1}^{N-1} 
\sum_{k=2m+1}^{N-1} \sum_{ \ell = 2 }^{N-1}
\sum_{r,s, t =1}^d  
\int_{\bR^d}
\biggl| \mu \biggl\{   \eta^{n,k\ell, k( \ell + 1)}(0, \tau, z)
Y_{t}^{n,k} \biggr\}  \biggr| \\
&\times | \mu ( G_{r,s}^{n,m} ) |
| \phi_{rst}(z) | \, dz \, d \tau \\
&\le \int_{\ve^2}^1
\frac{\sqrt{1-\tau} }{2\tau^{3/2}}   \sum_{ \delta_1 N \le n < \delta_2 N } 
\sum_{m=1}^{N-1} 
\sum_{k=2m+1}^{N-1} \sum_{ \ell = 2 }^{N-1}
\sum_{r,s, t =1}^d  
\int_{\bR^d}  \bfC  
L^3  \Vert b^{-1} \Vert_s^2
A_*(k , k( \ell - 1), k) \\
&\times q_*^{ k( \ell - 1) / 2 }  \cdot  L^2 \Vert b^{-1} \Vert_s
| \phi_{rst}(z) | \, dz \, d \tau \\
&\le \bfC  ( \delta_2 - \delta_1 )^{  1 - 3K_0 / 2}  d^3 2^{3K_0 / 2}
C_0^{3/2} L^5 N \Vert b^{-1} \Vert_s^4   \ve^{-2} \cE.
\end{align*}
Hence, $R_4$ satisfies \eqref{eq:r_upper}. In a similar way
 we obtain \eqref{eq:r_upper} for 
the remaining terms.

\subsection{Estimate on $\int_{0}^{\ve^2} | R_i(\tau) | \, d \tau$} Let us denote 
\begin{align*}
	&\eta^{n,m,k}_s(v, \tau, z) = h_s(\sqrt{1 - 
		\tau}(W^{n,m} + vY^{n,m}) - 
	\sqrt{\tau}
	z)   - h_s(\sqrt{1 - 
		\tau}W^{n,k}  - 
	\sqrt{\tau}
	z),
\end{align*}
where we recall that $h_s(x) = \partial_s h(x)$ is a partial derivative of $h$.
Using the second equality in \eqref{eq:diffs}, we express 
\begin{align*}
	R_1
	&= - \frac{1}{2  \sqrt{\tau} } \sum_{n=0}^{N-1} 
	\sum_{m=1}^{N-1}  \sum_{r,s=1}^d  \int_{\bR^d} 
	\mu \biggl\{   \biggl( 
	h_s( \sqrt{1 - 
		\tau}(W^{n,m} + uY^{n,m}) - 
	\sqrt{\tau}
	z ) \\
	&- 
	h_s( \sqrt{1 - \tau}W^{n,m} - \sqrt{\tau}
	z )
	\biggr) 
	Y^n_r  Y^{n,m}_s \biggr\} \phi_{r}(z) \, dz \\
	&= - \frac{1}{2  \sqrt{\tau} } \sum_{n=0}^{N-1} 
	\sum_{m=1}^{N-1}  \sum_{r,s=1}^d  \int_{\bR^d} 
	\mu \biggl\{   \eta_s^{n,m,m}(u, \tau, z)
	Y^n_r  Y^{n,m}_s \biggr\} \phi_{r}(z) \, dz.
\end{align*}
Similar representations can be derived for the remaining six terms:
\begin{align*}
	R_2 &= - \frac{1}{2  \sqrt{\tau} } \sum_{n=0}^{N-1}  \sum_{r,s=1}^d  \int_{\bR^d} 
	\mu \biggl\{   \eta_s^{n,0,0}(u, \tau, z)
	Y^n_r  Y^{n}_s \biggr\} \phi_{r}(z) \, dz, \\
	R_3 &=  - \frac{1}{2  \sqrt{\tau} } \sum_{n=0}^{N-1}\sum_{m=1}^{N-1} \sum_{k=m+1}^{2m} \sum_{r,s=1}^d  
	\int_{\bR^d} 
	\mu \biggl\{   \overline{   \eta_s^{n,k,k}(1, \tau, z) }
	Y^n_r  Y^{n,m}_s \biggr\} \phi_{r}(z) \, dz, \\
	R_4 &= - \frac{1}{2  \sqrt{\tau} }   \sum_{n=0}^{N-1}\sum_{m=1}^{N-1} \sum_{k=2m+1}^{N-1} \sum_{r,s=1}^d  
	 \int_{\bR^d} 
	\mu \biggl\{   \overline{\eta_s^{n,k,k}(1, \tau, z)}
	Y^n_r  Y^{n,m}_s \biggr\} \phi_{r}(z) \, dz;
	\end{align*}
	\begin{align*}
	R_5 &= -  \frac{1}{2  \sqrt{\tau} }\sum_{n=0}^{N-1} \sum_{m=1}^{N-1} \sum_{r,s=1}^d  \int_{\bR^d} 
	\mu \biggl\{   \overline{\eta_s^{n,m,m}(1, \tau, z)}
	Y^n_r  Y^{n}_s \biggr\} \phi_{r}(z) \, dz, \\
	R_6 &=   \frac{1}{2  \sqrt{\tau} } \sum_{n=0}^{N-1} \sum_{m=1}^{N-1}  \sum_{k=0}^m  \sum_{r,s=1}^d   \int_{\bR^d} 
	\mu \biggl\{   \eta_s^{n,k,k}(1, \tau, z) \biggr\}
	 \mu(Y^n_r  Y^{n,m}_s) \phi_{r}(z) \, dz, \\
	 R_7 &=  \frac{1}{2  \sqrt{\tau} } \sum_{n=0}^{N-1}   \sum_{r,s=1}^d  \int_{\bR^d} 
	 \mu \biggl\{   \eta_s^{n,0,0}(1, \tau, z) \biggr\}
	 \mu(Y^n_r  Y^{n}_s) \phi_{r}(z) \, dz.
\end{align*}
Considering the properties of $h = h_{C, \ve}$ from Lemma \ref{lem:approx}, we observe that 
by simply replacing $h$ with its partial derivative $h_s$ in the proofs of Lemmas
\ref{lem:global_estimates} and \ref{lem:decor}, 
we can derive the following estimates:
\begin{align}\label{eq:decorr_ve_small}
\begin{split}
&\mu \biggl\{ |
\eta^{n,m, m}_s(u, \tau, z)  | \biggr\} 
\lesssim   \ve^{-1} \cdot  (\delta_2 -  \delta_1)^{- 3 K_0 / 2} L^2  ( m + 1)^2   \Vert b^{-1} \Vert_s  \ve^{-1} \cE, \\ 
&\biggl| \mu \biggl\{   \eta_s^{n,m,m}(u, \tau, z)
Y^n_r  Y^{n,m}_s \biggr\} \biggr|  \le    \ve^{-1} \cdot   \bfC L^4  \Vert b^{-1} \Vert_s^3 B_*(m, 0) q_*^{ m / 2 }, \\
&\biggl| \mu \biggl\{   \overline{\eta_s^{n,m + k,m + k}(1, \tau, z)}
Y^n_r  Y^{n,m}_s \biggr\} \biggr| \le  \ve^{-1} \cdot \bfC L^4 \Vert b^{-1} \Vert_s^3 B_*(m, k) q_*^{ k / 2 },
\end{split}
\end{align}
where 
$$
B_*(m, k) =  (\delta_2 -  \delta_1)^{- 3 K_0 / 2} 2^{3K_0 / 2}
C_0^{3/2}  \ve^{-1} \cE ( m + k  + 1)^2.
$$
The factor $\ve^{-1}$ appears due to the fact
that the Lipschitz constant of $h_s$ is of order $\ve^{-2}$, whereas the Lipschitz constant of $h$ is 
of order $\ve^{-1}$.

From \eqref{eq:decorr_ve_small}, it follows that   
$$
|R_i| \le  \frac{1}{2  \sqrt{\tau} }   (\delta_2 -  \delta_1)^{ 1 - 3 K_0 / 2}  d^{2}  \bfC L^42^{3K_0 / 2}
C_0^{3/2}    N  \Vert b^{-1} \Vert_s^3   \ve^{-2} \cE,
$$
for each $1 \le i \le 7$. Integrating over $\tau$, we obtain
\begin{align}\label{eq:ub_r_lower}
&\int_0^{\ve^2} |R_i(\tau)| \, d \tau \le  \bfC (\delta_2 -  \delta_1)^{ 1 - 3 K_0 / 2}  d^{2} 2^{3K_0 / 2}
C_0^{3/2}   L^4   N  \Vert b^{-1} \Vert_s^3   \ve^{-1} \cE.
\end{align}

\subsection{Completing the proof of Theorem \ref{thm:main}} By Lemma \ref{lem:prelim_bound},
	\begin{align*}
	d_c( \cL(W), \cL(Z) )
	&\le 4 d^{\frac14} \ve
	+ \sup_{ f \in \mathfrak{F}_\ve }
	\sum_{i=1}^7 |E_i(f)| \\
	&\le 4 d^{\frac14} \ve 
	+ \sup_{ f \in \mathfrak{F}_\ve }
	\sum_{i=1}^7  \int_0^{ \ve^2 } |R_i(\tau)| \, d \tau + \sum_{i=1}^7  \int_{\ve^2}^{1} |R_i(\tau)| \, d \tau.
\end{align*}
Since $\bar{b} = \max \{ N \Vert b^{-1} \Vert_s^3, \Vert b^{-1} \Vert_s  \} \ge N \Vert b^{-1} \Vert_s^3$, assembling \eqref{eq:trivial}, \eqref{eq:r_upper} and \eqref{eq:ub_r_lower}, we now obtain
\begin{align*}
& ( \delta_2 - \delta_1 )^{ 3 K_0 / 2 }  \frac{d_c( \cL(W), \cN_d )}{ \bar{b} }  \\ 
&\le \frac{ 4d^{1/4} \ve  }{  \bar{b}  }  + (C_0')^{3/2} +	
\bfC  d^3 2^{3K_0 / 2}
C_0^{3/2}  L^5 
\biggl[  \Vert b^{-1} \Vert_s 
\ve^{-2}  \cE + 
\ve^{-1} \cE  + 1 \biggr]
+
\bfC  d^{2} 2^{3K_0 / 2}
C_0^{3/2}   L^4     \ve^{-1} \cE \\
&\le \frac{ 4d^{1/4} \ve  }{  \bar{b}  }  + (C_0')^{3/2} +	
\bfC  d^3 2^{3K_0 / 2}
C_0^{3/2}  L^5 
\biggl[  \Vert b^{-1} \Vert_s 
\ve^{-2}  \cE + 
\ve^{-1} \cE  + 1 \biggr] \\
&\le \frac{ 4d^{1/4} \ve  }{  \bar{b}  }  + (C_0')^{3/2} +	\bfC d^3  2^{3K_0 / 2}
C_0^{3/2} L^5 
+	\bfC d^{13/4} 2^{3K_0 / 2}
C_0^{3/2}  L^5  \biggl[ 
\Vert b^{-1} \Vert_s \ve^{-2} + \ve^{-1}
\biggr]   \biggl[  \ve +  \Vert b^{-1} \Vert_s   \biggr] \\
&+ \bfC d^3 2^{3K_0 / 2}
C_0^{3/2}  L^5  \biggl[ 
\Vert b^{-1} \Vert_s \ve^{-2} + \ve^{-1}
\biggr] \bar{b}
\fD.
\end{align*}
for arbitrary $\ve > 0$. Choosing $\ve = 4 \bfC d^3 2^{3K_0 / 2}
C_0^{3/2} L^5 \bar{b}$, it follows that
\begin{align*}
&( \delta_2 - \delta_1 )^{ 3 K_0 / 2 }  \frac{d_c( \cL(W), \cN_d )}{ \bar{b} }  
\le  \bfC d^{13/4}  2^{3K_0 / 2}
C_0^{3/2}  L^5  + (C_0')^{3/2} +  \frac12 \fD.
\end{align*}
Now, recalling the definition of $\fD$ from \eqref{eq:iso_D}, we arrive at the estimate
\begin{align*}
	\fD \le    \bfC  d^{13/4}  2^{3K_0 / 2}
	C_0^{3/2}  L^5  + (C_0')^{3/2} +  \frac{1}{2} \fD ,
\end{align*}
i.e.
$$
\fD \le \bfC  d^{13/4}  2^{3K_0 / 2}
C_0^{3/2}  L^5 + 2 (C_0')^{3/2}.
$$
The proof of Theorem \ref{thm:main} is complete.

\appendix

\section{Proof of Theorem \ref{thm:exp_loss} and Corollary \ref{cor:corr_decay}}\label{sec:ml_proof}

In this section, we give a proof of the memory loss estimate in Theorem~\ref{thm:exp_loss}. The proof follows closely the strategy of \cite{korepanov2019explicit} and is included here for completeness. As a consequence, we obtain the correlation decay bounds in Corollary~\ref{cor:corr_decay}. For the model under consideration, alternative methods exist for deriving similar bounds, 
such as those discussed in
 \cite{zweinmuller2004kuzmin, dragicevic2018almost, gupta2013memory, dolgopyat2024rates}.

\subsection{Proof of Theorem \ref{thm:exp_loss}}

\begin{lem}\label{lem:to_ub_1} Let $\psi : M \to (0, + \infty)$. Suppose that $m = kp + \ell$, where 
	$0 \le \ell < p$. Then, for any $j \ge 1$,
	\begin{align}\label{eq:ml_1}
		| \cP_{j,  j + m - 1} ( \psi  \mathbf{1}_a   ) |_{\alpha , \ell} 
		\le K +   | \psi  |_{\alpha , \ell} (K')^\alpha \Lambda^{- \alpha k}.
	\end{align}
	for any $a \in \cA( \cT_{j, j + m - 1} )$. In particular, if $m \ge  ( \lceil  \log(K') / \log(\Lambda) \rceil + 1 )p$, 
	\begin{align}\label{eq:ml_2}
		| \cP_{j, j + m - 1} ( \psi  \mathbf{1}_a   ) |_{\alpha , \ell} 
		\le K +   | \psi  |_{\alpha, \ell} \Lambda^{ - \alpha  }.
	\end{align}
	Moreover, \eqref{eq:ml_1} and \eqref{eq:ml_2} hold with $\psi$ in place of $\psi  \mathbf{1}_a$.
\end{lem}

\begin{proof} The last statement follows by the fact that $|  \sum_n \psi_n  |_{\alpha , \ell} \le \sup_n |\psi_n|_{\alpha , \ell}$ for any 
	countable collection $\{ \psi_n \}$ of maps $\psi_n : M \to (0, + \infty)$.
	
	Let $a \in  \cA( \cT_{j, j + m - 1} )$. Then, 
	\begin{align*}
		\cP_{j, j + m - 1} ( \psi  \mathbf{1}_a   )(y)  = \zeta^{(j, j + m - 1)}_a( y ) \psi ( y_a),
	\end{align*}
	where $y_a$ denotes the unique preimage under $\cT_{j, j + m - 1}$ lying in $a$. Hence,
	\begin{align}\label{eq:transfer_log}
		\begin{split}
		&|  \log  \cP_{j, j + m - 1} ( \psi  \mathbf{1}_a   )(x) -  \log \cP_{j, j + m - 1} ( \psi  \mathbf{1}_a   )(y)  |  \\ 
		&\le |  \log  \zeta^{(j, j + m - 1)}_a( x ) - \log  \zeta^{(j, j + m - 1)}_a( y )   | + | \log \psi (x_a) - \log \psi ( y_a)|  \\
		&\le K d(x,y)^\alpha  + | \psi  |_{\alpha , \ell} d(x_a, y_a)^\alpha  \\ 
		&\le K d(x,y)^\alpha  +   | \psi  |_{\alpha , \ell} (K')^\alpha \Lambda^{- \alpha  k}  d(x, y)^\alpha, 
		\end{split}
	\end{align}
	where (UE:1-3) were used in the last two inequalities.
\end{proof}

\begin{lem}\label{lem:return_to_D} Let $\psi \in \cD_{\alpha, A}$. Then, there exists $\tilde{A} = \tilde{A}( A, K', K, \alpha) \ge A$ such that 
$ \cP_{ j, j + m - 1 }( \cD_{\alpha, A } ) \subset \cD_{\alpha,  \tilde{A} }$ holds for all $j \ge 1$ and all $m \ge 0$.
\end{lem}

\begin{proof} By \eqref{eq:transfer_log}, 
\begin{align}\label{eq:iterate_to}
	\cP_{ j, j + m - 1 }( \cD_{\alpha, A } ) \subset \cD_{\alpha,  \tilde{A} }
\end{align}
holds for $\tilde{A} = K + A(K')^{\alpha}$, whenever $j \ge 1$ and all $m \ge 0$.
\end{proof}

Fix $R > 0$ and $\xi \in (0, e^{-R})$ such that 
$$
R(1 - \xi e^R) \ge K + \Lambda^{- \alpha  } R,
$$
for example,
$
R = 2K / ( 1 - \Lambda^{ - \alpha })$ and 
$\xi  =  e^{-R} ( 1 - \Lambda^{- \alpha  } ) / 2
$.
Set 
$$
\tilde{p} = ( \lceil  \log(K') / \log(\Lambda) \rceil + 1 )p.
$$
\begin{lem}\label{lem:reg_pw} Let $\psi : M \to (0, + \infty)$ satisfy $| \psi |_{\alpha, \ell} \le R$. Then, 
	for any $j \ge 1$ and $m \ge  \tilde{p}$,
	$$
	|  \cP_{j, j + m - 1 } ( \psi  \mathbf{1}_a   ) |_{\alpha , \ell} \le  R,
	$$
	whenever $a \in \cA( \cT_{j, j + m - 1} )$. The inequality continues to hold if $\psi \mathbf{1}_a$ is replaced with $\psi$.
\end{lem}

\begin{proof} By Lemma \ref{lem:to_ub_1},
	$$
	| \cP_{j, j + m - 1} ( \psi  \mathbf{1}_a   ) |_{\alpha , \ell} \le K + \Lambda^{- \alpha  } R \le R.
	$$
\end{proof}

\begin{lem}\label{lem:coupling} Let $\psi^{(1)}, \psi^{(2)} : M \to (0, \infty)$ with $|\psi^{(i)}|_{\alpha , \ell} \le R$ and 
	$\int_M \psi^{(1)} \, d \lambda = \int_M \psi^{(2)} \, d \lambda$. Set 
	$$
	\psi_{j, m}^{(i)} = \cP_{j, j + m - 1} \psi^{(i)} - \xi \int_M \psi ^{(i)} \, d \lambda, \quad i = 1,2.
	$$
	Then, for any $j \ge 1$ and $m \ge \tilde{p}$:
	\begin{itemize}
		\item[(i)] $ |  \psi_{j, m}^{(i)} |_{\alpha , \ell }  \le R$,
		\item[(ii)] $  \cP_{j, j + m - 1} \psi^{(1)} -  \cP_{j, j + m - 1} \psi^{(2)}  = \psi_{j, m}^{(1)} - \psi_{j, m}^{(2)}$,
		\item[(iii)] $\int_M \psi_{j,m}^{(1)} \, d \lambda = \int_M \psi_{j,m}^{(2)} \, d \lambda = (1 - \xi) \int_M \psi^{(1)} \, d \lambda$.
	\end{itemize}
\end{lem}

\begin{proof} (ii) and (iii) are trivial. By \cite[Proposition 3.2]{korepanov2019explicit}
	\begin{align*}
		| \psi_{j}^{(i)} |_{\alpha , \ell}  \le \frac{ |  \cP_{j, j + m - 1} \psi^{(i)}  |_{\alpha , \ell } }{1 - \xi  e^{  |  \cP_{j, j + m - 1} \psi^{(i)}   |_{\alpha , \ell }  } } \le 
		\frac{K +   \Lambda^{- \alpha  }  R   }{1 - \xi e^R }  \le R.
	\end{align*}
\end{proof}

\begin{proof}[Completing the proof of Theorem \ref{thm:exp_loss}] Without loss of generality, we shall assume that $i = 1$.
	Write $n = \tilde{p} k + \ell$, where $0 \le \ell < \tilde{p}$.

	First assume $|u |_\alpha  \le R$ so that $\Vert u \Vert_\infty \le R$, since 
	$\int_M u \, d \lambda = 0$. Decompose $u = \psi_0^+ - \psi_0^{-}$, where 
	\begin{align*}
		\psi_0^+ = 1 + \max \{0, u\} \quad \text{and} \quad \psi_0^{-} = 1 - \min\{ 0 , u \}.
	\end{align*}
	Then, $\psi_0^{ \pm } \ge 1$,
	\begin{align*}
		\int_M \psi_0^{+} \, d \lambda = 	\int_M \psi_0^{-} \, d \lambda  \le 1 + \Vert u \Vert_\infty \le 1 + R,
	\end{align*} 
	and, for all $x,y \in M$,
	\begin{align*}
		| \log  \psi_0^{ \pm }(x) - \log \psi_0^{ \pm }(y)   | \le 	| \psi_0^{ \pm }(x) - \psi_0^{ \pm }(y)   | \le |u(x) - u(y)| \le Rd(x,y)^\alpha.
	\end{align*}
	Hence, $| \psi_0^{ \pm }   |_{\alpha , \ell} \le R$.
	
	Recursively define
	\begin{align*}
		\psi_{1}^{\pm} &=  \cP_{1, \tilde{p} + \ell } \psi_0^{\pm } - \xi \int_M  \psi_0^{\pm } \, d \lambda, \\ 
		\psi_{j+1}^{\pm} &= \cP_{ \ell + j \tilde{p}   + 1, \ell +  (j + 1) \tilde{p}} \psi_j^{\pm } - \xi \int_M  \psi_j^{\pm }, \quad j = 1, \ldots,  k-1.
	\end{align*}
	By Lemma \ref{lem:coupling}-(i), $| \psi_{j}^{\pm} |_{\alpha , \ell} \le R$,
	\begin{align}\label{eq:diff_recursive}
		\cP_n ( u )  = \psi_k^{+} -  \psi_k^{-} ,
	\end{align}
	and 
	\begin{align*}
		\int_M \psi_{k}^{\pm } \, d \lambda  &= \int_M \psi_{k-1}^{ \pm } \, d \lambda - \xi \int_M  \psi_{k-1}^{ \pm } \, d \lambda 
		= (1 - \xi ) \int_M  \psi_{k-1}^{ \pm } \, d \lambda \\
		&= \cdots = (1- \xi)^k \int_M \psi_0^{ \pm} \, d \lambda \le (1- \xi)^k  (1 + R).
	\end{align*}
	Set $q = 1 - \xi$. By \eqref{eq:psi_lb_ub}, we have that 
	\begin{align}\label{eq:ububu}
	\psi_k^{+} \le e^R \int_M \psi_k^{\pm} \, d \lambda \le e^R (1 + R) q^k.
	\end{align}
	The inequality 
	\begin{align*}
		|a - b| \le \max\{ a, b\} |  \log a - \log b | \quad \forall a,b > 0,
	\end{align*}
	combined with $| \psi_{j}^{\pm} |_{\alpha , \ell} \le R$ and \eqref{eq:ububu} yields 
	\begin{align*}
		&| \psi_k^{ \pm }(x) -  \psi_k^{ \pm }(y)  | \le e^R R (1 + R) q^k d(x,y)^\alpha.
	\end{align*}
	Hence, by \eqref{eq:diff_recursive}, 
	$$
	|\cP_n  ( u) |_\alpha  \le 2 e^R R (1 + R) q^k .
	$$
	
	Finally, to remove the restriction $| u |_\alpha  \le R$, it suffices to observe that 
	$
	v = R | u |_\alpha^{-1} u
	$
	satisfies $| v |_\alpha  \le R$ and therefore 
	$$
	|\cP_n  ( u) |_\alpha  =  R^{-1} | u |_\alpha  | \cP_n  ( v) |_\alpha \le   2 e^R  (1 + R) q^k | u |_\alpha .
	$$
	Moreover, $\int_M \cP^n u \, d \lambda = 0$, so that $\Vert \cP^n u   \Vert_\infty \le |  \cP^n u  |_\alpha$. Hence,
	$$
	\Vert \cP_n  ( u) \Vert_\alpha \le 4 e^R  (1 + R) q^k | u |_\alpha \le  4 e^R  (1 + R  )q^{-1} q^{ n / \tilde{p}}  | u |_\alpha.
	$$
\end{proof}

\subsection{Proof of Corollary \ref{cor:corr_decay}}\label{sec:proof_decor} By basic properties of $P_n$, we have that
\begin{align*}
	| \mu(   \bar{\psi}_1^n  \bar{\psi}_2^{n+m}    ) | &= | \mu( \psi_2^{m + n} \psi_1^n -  \psi_2^{n + m}  \mu(  \psi_1^n  ) ) | \\
	&= | \lambda \{    \psi_2 \cP_{n + m} [ (  \psi_1^n  -  \mu ( \psi_1^n )   ) \rho ]    \} | \\
	&\le \Vert \psi_2 \Vert_{ L^1( \lambda ) }  \Vert 
	\cP_{n + m} [ (  \psi_1^n  -  \mu ( \psi_1^n )   ) \rho ] 
	   \Vert_\alpha  \\
	&= \Vert \psi_2 \Vert_{L^1( \lambda )}  \Vert  \cP_{n + 1, n + m} [  \psi_1  \cP_n( \rho ) - \mu ( \psi_1^n )  \cP_n (\rho  )    ] \Vert_\alpha,
\end{align*}
where the basic identity $\cP_n(  f \circ \cT^n  \cdot g   ) = f \cP_n ( g)$ was used in the last equality.
Further,
\begin{align*}
	| \psi_1 \cP_n( \rho ) |_\alpha \le \Vert \psi_1 \Vert_\infty | \cP_n ( \rho ) |_\alpha 
	+ | \psi_1 |_\alpha  \Vert \cP_n ( \rho ) \Vert_\infty.
\end{align*}
Recalling \eqref{eq:iterate_to}, we have $| \cP_{n}(\rho) |_{\alpha, \ell} \le K + A(K')^{\alpha}$.
Hence, by \eqref{eq:from_lip_to_ll},
$$
| \cP_n ( \rho ) |_\alpha \le |  \cP_n(  \rho ) |_{ \alpha, \ell } 
e^{ |  \cP_n(  \rho ) |_{\alpha, \ell}  } \le  ( K + A(K')^{\alpha} ) e^{  K + A(K')^{\alpha} },
$$
and, by \eqref{eq:psi_lb_ub}, $\Vert \cP_n ( \rho ) \Vert_\infty \le e^{ K + A(K')^{\alpha} }$. We conclude that
$$
| \psi_1 \cP_n( \rho ) |_\alpha \le  \Vert \psi_1 \Vert_\alpha  ( K + A (K')^\alpha  + 1 ) e^{  K + A (K')^\alpha  }.
$$
and
$$
| \mu ( \psi_1 \circ \cT_n )  \cP_n (\rho  )   |_\alpha 
\le \Vert \psi_1 \Vert_{\alpha} ( K + A(K')^{\alpha} ) e^{  K + A(K')^{\alpha} }.
$$
It follows by Theorem \ref{thm:exp_loss} that 
\begin{align}\label{eq:second_order_corr}
	\Vert  \cP_{n + 1, n + m} [  \psi_1  \cP_n( \rho ) - \mu ( \psi_1^n )  \cP_n (\rho  )    ] \Vert_\alpha 
	\le  \bfC   \Vert \psi_1 \Vert_\alpha q^m.
\end{align}
Therefore, the first bound in \eqref{eq:correlations} holds.

For the second bound in \eqref{eq:correlations}, it is enough to prove that, whenever $\psi_i \in C^\alpha$ ($i = 1,2,3$) and $n, m, k \ge 0$,
\begin{align}\label{eq:to_show_corr3}
\begin{split}
	&| \mu(  \psi_1^n  \psi_2^{n+m}  \psi_3^{n+m+k} )
	- \mu (
	\psi_1^n   \psi_2^{n+m} ) \mu (   \psi_3^{n+m+k}  ) | 
	\le  \bfC \Vert \psi_1 \Vert_{\alpha} \Vert \psi_2 \Vert_\alpha \Vert \psi_3 \Vert_{\alpha} q^{ k }, \\
	&| \mu(  \psi_1^n  \psi_2^{n+m} \psi_3^{n + m + k}   )
	- \mu(  \psi_1^n ) \mu(  \psi_2^{n+m}  \psi_3^{n+m+k}  ) | 
	\le \bfC \Vert \psi_1 \Vert_{\alpha} \Vert \psi_2 \Vert_\alpha \Vert \psi_3 \Vert_{\alpha} q^{ m }.
\end{split}
\end{align}
To obtain the first bound in \eqref{eq:to_show_corr3}, we write
\begin{align*}
	&| \mu(  \psi_1^n  \psi_2^{n+m}  \psi_3^{n+m+k} )
	- \mu (
	\psi_1^n   \psi_2^{n+m} ) \mu (   \psi_3^{n+m+k}  ) |   \\
	&= |  \lambda \{  \psi_3 \cP_{n + m + k} [
	\rho \psi_1^n \psi_2^{n+m}  - \mu(  \psi_1^n \psi_2^{n+m}  ) \rho ]   \}   | \\
	&=  |  \lambda \{  \psi_3 \cP_{n + m + 1, n + m + k} [ \psi_2 
	\cP_{ n + 1, n + m } \psi_1  \cP_n \rho 
	- \mu(  \psi_1^n \psi_2^{n+m}  )  \cP_{n + m} \rho 
	]   \}   | \\
	&\le \Vert \psi_3 \Vert_\infty  \lambda \{  | \cP_{n + m + 1, n + m + k} [ \psi_2 
	\cP_{ n + 1, n + m } \psi_1  \cP_n \rho 
	- \mu(  \psi_1^n \psi_2^{n+m}  )  \cP_{n + m} \rho 
	]  | \},
\end{align*}
where the convention is that each operator acts on the entire expression to its right.
Let $\tilde{L}_i = \Vert \psi_i \Vert_\alpha + 1$, and define $\tilde{\psi}_i = \psi_i + \tilde{L}_i$. Note that $\tilde{\psi}_i \ge 1$ and, by \eqref{eq:from_ll_to_lip},
$
| \tilde{\psi}_i |_{ \alpha, \ell } \le 1$.
Writing $\tilde \psi_i^n = \tilde \psi_i \circ \cT_n$, we decompose 
\begin{align*}
	&\cP_{n + m + 1, n + m + k} [ \psi_2 
	\cP_{ n + 1, n + m } \psi_1  \cP_n \rho 
	- \mu(  \psi_1^n \psi_2^{n+m}  )  \cP_{n + m} \rho 
	]  \\
	&= \cP_{n + m + 1, n + m + k} [ \tilde{\psi}_2 
	\cP_{ n + 1, n + m } \tilde{\psi}_1  \cP_n \rho 
	- \mu(  \tilde{\psi}_1^n \tilde{\psi}_2^{n+m}  )  \cP_{n + m} \rho 
	] \\
	&+ \cP_{n + m + 1, n + m + k} [ \tilde{\psi}_2 
	\cP_{ n + 1, n + m } \tilde{L}_1  \cP_n \rho 
	- \mu(  \tilde{L}_1 \tilde{\psi}_2^{n+m}  )  \cP_{n + m} \rho 
	] \\
	&+ \cP_{n + m + 1, n + m + k} [ \tilde{L}_2
	\cP_{ n + 1, n + m } \tilde{\psi}_1  \cP_n \rho 
	- \mu(  \tilde{\psi}_1^n  \tilde{L}_2  )  \cP_{n + m} \rho 
	] 
	= I + II + III.
\end{align*}
Using \eqref{eq:iterate_to}, \eqref{eq:psi_lb_ub} and \eqref{eq:from_lip_to_ll}, it is straightforward to verify that each of the functions 
inside the square brackets satisfies $|\cdot |_{\alpha} \le \bfC \Vert \psi_1 \Vert_{\alpha} \Vert \psi_2 \Vert_\alpha$. By an application of Theorem \ref{thm:exp_loss}, we now obtain
the first bound in \eqref{eq:to_show_corr3}.

For the second bound in \eqref{eq:to_show_corr3}, we note that 
\begin{align*}
	&| \mu(  \psi_1^n  \psi_2^{n+m} \psi_3^{n + m + k}   )
	- \mu(  \psi_1^n ) \mu(  \psi_2^{n+m}  \psi_3^{n+m+k}  ) |  \\
	&= | \lambda \{
	\psi_2 \psi_3 \circ \cT_{n+m+1,n+m+k}
	\cP_{n+m}[(\psi_1^n-\mu(\psi_1^n))\rho]
	\} | \\
	&= | \lambda \{
	\psi_2  \psi_3 \circ \cT_{n+m+1,n+m+k}
	\cP_{n+1,n+m}\bigl[(\psi_1-\mu(\psi_1^n))\cP_n ( \rho ) ]
	\} | \\
	&\le \Vert \psi_2 \Vert_\infty \Vert \psi_3 \Vert_\infty
	\|
	\cP_{n+1,n+m}\bigl[(\psi_1-\mu(\psi_1^n))\cP_n ( \rho  ) ]
	\|_\alpha .
\end{align*}
Therefore, the desired bound follows from \eqref{eq:second_order_corr}.

\bibliography{raic}{}

\begin{thebibliography}{10}

\bibitem{aimino2015}
Romain Aimino, Huyi Hu, Matthew Nicol, Andrei T{\"o}r{\"o}k, and Sandro
  Vaienti.
\newblock Polynomial loss of memory for maps of the interval with a neutral
  fixed point.
\newblock {\em Discrete Contin. Dyn. Syst.}, 35(3):793--806, 2015.
\newblock URL: \url{http://dx.doi.org/10.3934/dcds.2015.35.793}.

\bibitem{A15}
Romain Aimino, Matthew Nicol, and Sandro Vaienti.
\newblock Annealed and quenched limit theorems for random expanding dynamical
  systems.
\newblock {\em Probab. Theory Related Fields}, 162(1-2):233--274, 2015.
\newblock \href {https://doi.org/10.1007/s00440-014-0571-y}
  {\path{doi:10.1007/s00440-014-0571-y}}.

\bibitem{antoniou2019rate}
Marios Antoniou and Ian Melbourne.
\newblock Rate of convergence in the weak invariance principle for
  deterministic systems.
\newblock {\em Comm. Math. Phys.}, 369(3):1147--1165, 2019.
\newblock \href {https://doi.org/10.1007/s00220-019-03334-6}
  {\path{doi:10.1007/s00220-019-03334-6}}.

\bibitem{ALS09}
Arvind Ayyer, Carlangelo Liverani, and Mikko Stenlund.
\newblock Quenched {CLT} for random toral automorphism.
\newblock {\em Discrete Contin. Dyn. Syst.}, 24(2):331--348, 2009.
\newblock \href {https://doi.org/10.3934/dcds.2009.24.331}
  {\path{doi:10.3934/dcds.2009.24.331}}.

\bibitem{B94-2}
V.~I. Bakhtin.
\newblock Random processes generated by a hyperbolic sequence of mappings. {I}.
\newblock {\em Izv. Ross. Akad. Nauk Ser. Mat.}, 58(2):40--72, 1994.
\newblock URL: \url{https://doi.org/10.1070/IM1995v044n02ABEH001596}.

\bibitem{B94-1}
V.~I. Bakhtin.
\newblock Random processes generated by a hyperbolic sequence of mappings.
  {II}.
\newblock {\em Izv. Ross. Akad. Nauk Ser. Mat.}, 58(3):184--195, 1994.
\newblock \href {https://doi.org/10.1070/IM1995v044n03ABEH001616}
  {\path{doi:10.1070/IM1995v044n03ABEH001616}}.

\bibitem{ball1993reverse}
Keith Ball.
\newblock The reverse isoperimetric problem for {G}aussian measure.
\newblock {\em Discrete Comput. Geom.}, 10(4):411--420, 1993.
\newblock \href {https://doi.org/10.1007/BF02573986}
  {\path{doi:10.1007/BF02573986}}.

\bibitem{bentkus2003dependence}
V.~Bentkus.
\newblock On the dependence of the {B}erry-{E}sseen bound on dimension.
\newblock {\em J. Statist. Plann. Inference}, 113(2):385--402, 2003.
\newblock \href {https://doi.org/10.1016/S0378-3758(02)00094-0}
  {\path{doi:10.1016/S0378-3758(02)00094-0}}.

\bibitem{B86}
V.~Yu. Bentkus.
\newblock Dependence of the {B}erry-{E}sseen estimate on the dimension.
\newblock {\em Litovsk. Mat. Sb.}, 26(2):205--210, 1986.

\bibitem{berry1941accuracy}
Andrew~C. Berry.
\newblock The accuracy of the {G}aussian approximation to the sum of
  independent variates.
\newblock {\em Trans. Amer. Math. Soc.}, 49:122--136, 1941.
\newblock \href {https://doi.org/10.2307/1990053} {\path{doi:10.2307/1990053}}.

\bibitem{bunimovic1974central}
L.A. Bunimovich.
\newblock Central limit theorem for a class of billiards.
\newblock {\em Theory Probab. Its Appl.}, 19(1):65--85, 1974.

\bibitem{CM08}
Sourav Chatterjee and Elizabeth Meckes.
\newblock Multivariate normal approximation using exchangeable pairs.
\newblock {\em ALEA Lat. Am. J. Probab. Math. Stat.}, 4:257--283, 2008.

\bibitem{CP90}
Zaqueu Coelho and William Parry.
\newblock Central limit asymptotics for shifts of finite type.
\newblock {\em Israel J. Math.}, 69(2):235--249, 1990.
\newblock \href {https://doi.org/10.1007/BF02937307}
  {\path{doi:10.1007/BF02937307}}.

\bibitem{conze2007}
Jean-Pierre Conze and Albert Raugi.
\newblock Limit theorems for sequential expanding dynamical systems on
  {$[0,1]$}.
\newblock In {\em Ergodic theory and related fields}, volume 430 of {\em
  Contemp. Math.}, pages 89--121. Amer. Math. Soc., Providence, RI, 2007.
\newblock \href {https://doi.org/10.1090/conm/430/08253}
  {\path{doi:10.1090/conm/430/08253}}.

\bibitem{dedecker2022rates}
J\'{e}r\^{o}me Dedecker, Florence Merlev\`ede, and Emmanuel Rio.
\newblock Rates of convergence in the central limit theorem for martingales in
  the non stationary setting.
\newblock {\em Ann. Inst. Henri Poincar\'{e} Probab. Stat.}, 58(2):945--966,
  2022.
\newblock \href {https://doi.org/10.1214/21-aihp1182}
  {\path{doi:10.1214/21-aihp1182}}.

\bibitem{dedecker2008mean}
J\'{e}r\^{o}me Dedecker and Emmanuel Rio.
\newblock On mean central limit theorems for stationary sequences.
\newblock {\em Ann. Inst. Henri Poincar\'{e} Probab. Stat.}, 44(4):693--726,
  2008.
\newblock \href {https://doi.org/10.1214/07-AIHP117}
  {\path{doi:10.1214/07-AIHP117}}.

\bibitem{DL25}
Mark Demers and Carlangelo Liverani.
\newblock Central limit theorem for sequential dynamical systems.
\newblock 2025.
\newblock Preprint.
\newblock \href {http://arxiv.org/abs/2502.07765} {\path{arXiv:2502.07765}}.

\bibitem{denker2004poisson}
Manfred Denker, Mikhail Gordin, and Anastasya Sharova.
\newblock A {P}oisson limit theorem for toral automorphisms.
\newblock {\em Illinois J. Math.}, 48(1):1--20, 2004.
\newblock URL: \url{http://projecteuclid.org/euclid.ijm/1258136170}.

\bibitem{dolgopyat2024rates}
Dimitry Dolgopyat and Yeor Hafouta.
\newblock Rates of convergence in {CLT} and {ASIP} for sequences of expanding
  maps.
\newblock 2024.
\newblock Preprint.
\newblock \href {http://arxiv.org/abs/2401.08802} {\path{arXiv:2401.08802}}.

\bibitem{dragicevic2018almost}
D.~Dragi\v{c}evi\'{c}, G.~Froyland, C.~Gonz\'{a}lez-Tokman, and S.~Vaienti.
\newblock Almost sure invariance principle for random piecewise expanding maps.
\newblock {\em Nonlinearity}, 31(5):2252--2280, 2018.
\newblock \href {https://doi.org/10.1088/1361-6544/aaaf4b}
  {\path{doi:10.1088/1361-6544/aaaf4b}}.

\bibitem{dragicevic2020limit}
Davor Dragi\v{c}evi\'{c} and Yeor Hafouta.
\newblock Limit theorems for random expanding or {A}nosov dynamical systems and
  vector-valued observables.
\newblock {\em Ann. Henri Poincar\'{e}}, 21(12):3869--3917, 2020.
\newblock \href {https://doi.org/10.1007/s00023-020-00965-7}
  {\path{doi:10.1007/s00023-020-00965-7}}.

\bibitem{esseen1942liapunov}
Carl-Gustav Esseen.
\newblock On the {L}iapunov limit error in the theory of probability.
\newblock {\em Ark. Mat. Astr. Fys.}, 28:1--19, 1942.

\bibitem{fang2016multivariate}
Xiao Fang.
\newblock A multivariate {CLT} for bounded decomposable random vectors with the
  best known rate.
\newblock {\em J. Theoret. Probab.}, 29(4):1510--1523, 2016.
\newblock \href {https://doi.org/10.1007/s10959-015-0619-7}
  {\path{doi:10.1007/s10959-015-0619-7}}.

\bibitem{fang2015rates}
Xiao Fang and Adrian R\"{o}llin.
\newblock Rates of convergence for multivariate normal approximation with
  applications to dense graphs and doubly indexed permutation statistics.
\newblock {\em Bernoulli}, 21(4):2157--2189, 2015.
\newblock \href {https://doi.org/10.3150/14-BEJ639}
  {\path{doi:10.3150/14-BEJ639}}.

\bibitem{fernando2021edgeworth}
Kasun Fernando and Carlangelo Liverani.
\newblock Edgeworth expansions for weakly dependent random variables.
\newblock {\em Ann. Inst. Henri Poincar\'{e} Probab. Stat.}, 57(1):469--505,
  2021.
\newblock \href {https://doi.org/10.1214/20-aihp1085}
  {\path{doi:10.1214/20-aihp1085}}.

\bibitem{gaunt2016}
Robert~E. Gaunt.
\newblock Rates of convergence in normal approximation under moment conditions
  via new bounds on solutions of the {S}tein equation.
\newblock {\em J. Theoret. Probab.}, 29(1):231--247, 2016.
\newblock \href {https://doi.org/10.1007/s10959-014-0562-z}
  {\path{doi:10.1007/s10959-014-0562-z}}.

\bibitem{gordin2012poisson}
Mikhail Gordin and Manfred Denker.
\newblock The {P}oisson limit for automorphisms of two-dimensional tori driven
  by continued fractions.
\newblock {\em M. J Math Sci}, 199(2):139--149, 2014.
\newblock \href {https://doi.org/10.1007/s10958-014-1841-z}
  {\path{doi:10.1007/s10958-014-1841-z}}.

\bibitem{gotze1991rate}
F.~G\"{o}tze.
\newblock On the rate of convergence in the multivariate {CLT}.
\newblock {\em Ann. Probab.}, 19(2):724--739, 1991.
\newblock URL: \url{https://www.jstor.org/stable/2244370}.

\bibitem{gouezel2004berry}
S\'{e}bastien Gou\"{e}zel.
\newblock Berry-{E}sseen theorem and local limit theorem for non uniformly
  expanding maps.
\newblock {\em Ann. Inst. H. Poincar\'{e} Probab. Statist.}, 41(6):997--1024,
  2005.
\newblock \href {https://doi.org/10.1016/j.anihpb.2004.09.002}
  {\path{doi:10.1016/j.anihpb.2004.09.002}}.

\bibitem{gupta2013memory}
Chinmaya Gupta, William Ott, and Andrei T\"{o}r\"{o}k.
\newblock Memory loss for time-dependent piecewise expanding systems in higher
  dimension.
\newblock {\em Math. Res. Lett.}, 20(1):141--161, 2013.
\newblock \href {https://doi.org/10.4310/MRL.2013.v20.n1.a12}
  {\path{doi:10.4310/MRL.2013.v20.n1.a12}}.

\bibitem{H20}
Yeor Hafouta.
\newblock Limit theorems for some time-dependent expanding dynamical systems.
\newblock {\em Nonlinearity}, 33(12):6421--6460, 2020.
\newblock \href {https://doi.org/10.1088/1361-6544/aba5e7}
  {\path{doi:10.1088/1361-6544/aba5e7}}.

\bibitem{haydn2013entry}
Nicolai Haydn.
\newblock Entry and return times distribution.
\newblock {\em Dyn. Syst.}, 28(3):333--353, 2013.
\newblock \href {https://doi.org/10.1080/14689367.2013.822459}
  {\path{doi:10.1080/14689367.2013.822459}}.

\bibitem{haydn2017}
Nicolai Haydn, Matthew Nicol, Andrew T\"or\"ok, and Sandro Vaienti.
\newblock Almost sure invariance principle for sequential and non-stationary
  dynamical systems.
\newblock {\em Trans. Amer. Math. Soc.}, 369(8):5293--5316, 2017.
\newblock \href {https://doi.org/10.1090/tran/6812}
  {\path{doi:10.1090/tran/6812}}.

\bibitem{haydn2016entry}
Nicolai Haydn and Fan Yang.
\newblock Entry times distribution for mixing systems.
\newblock {\em J. Stat. Phys.}, 163(2):374--392, 2016.
\newblock \href {https://doi.org/10.1007/s10955-016-1487-y}
  {\path{doi:10.1007/s10955-016-1487-y}}.

\bibitem{heinrich1996mixing}
Lothar Heinrich.
\newblock Mixing properties and central limit theorem for a class of
  non-identical piecewise monotonic {$C^2$}-transformations.
\newblock {\em Math. Nachr.}, 181:185--214, 1996.
\newblock \href {https://doi.org/10.1002/mana.3211810107}
  {\path{doi:10.1002/mana.3211810107}}.

\bibitem{hella2020stein}
Olli Hella, Juho Lepp\"{a}nen, and Mikko Stenlund.
\newblock Stein's method of normal approximation for dynamical systems.
\newblock {\em Stoch. Dyn.}, 20(4):2050021, 50, 2020.
\newblock \href {https://doi.org/10.1142/S0219493720500215}
  {\path{doi:10.1142/S0219493720500215}}.

\bibitem{hella2020quenched}
Olli Hella and Mikko Stenlund.
\newblock Quenched normal approximation for random sequences of
  transformations.
\newblock {\em J. Stat. Phys.}, 178(1):1--37, 2020.
\newblock \href {https://doi.org/10.1007/s10955-019-02390-5}
  {\path{doi:10.1007/s10955-019-02390-5}}.

\bibitem{keller1980}
Gerhard Keller.
\newblock Un th\'eor\`eme de la limite centrale pour une classe de
  transformations monotones par morceaux.
\newblock {\em C. R. Acad. Sci. Paris S\'er. A-B}, 291(2):A155--A158, 1980.

\bibitem{K98}
Yuri Kifer.
\newblock Limit theorems for random transformations and processes in random
  environments.
\newblock {\em Trans. Amer. Math. Soc.}, 350(4):1481--1518, 1998.
\newblock \href {https://doi.org/10.1090/S0002-9947-98-02068-6}
  {\path{doi:10.1090/S0002-9947-98-02068-6}}.

\bibitem{K08}
Yuri Kifer.
\newblock Thermodynamic formalism for random transformations revisited.
\newblock {\em Stoch. Dyn.}, 8(1):77--102, 2008.
\newblock \href {https://doi.org/10.1142/S0219493708002238}
  {\path{doi:10.1142/S0219493708002238}}.

\bibitem{korepanov2019explicit}
A.~Korepanov, Z.~Kosloff, and I.~Melbourne.
\newblock Explicit coupling argument for non-uniformly hyperbolic
  transformations.
\newblock {\em Proc. Roy. Soc. Edinburgh Sect. A}, 149(1):101--130, 2019.
\newblock \href {https://doi.org/10.1017/S0308210518000161}
  {\path{doi:10.1017/S0308210518000161}}.

\bibitem{korepanov2021loss}
A.~Korepanov and J.~Lepp\"{a}nen.
\newblock Loss of memory and moment bounds for nonstationary intermittent
  dynamical systems.
\newblock {\em Comm. Math. Phys.}, 385(2):905--935, 2021.
\newblock \href {https://doi.org/10.1007/s00220-021-04071-5}
  {\path{doi:10.1007/s00220-021-04071-5}}.

\bibitem{leppanen2017functional}
Juho Lepp\"{a}nen.
\newblock Functional correlation decay and multivariate normal approximation
  for non-uniformly expanding maps.
\newblock {\em Nonlinearity}, 30(11):4239--4259, 2017.
\newblock \href {https://doi.org/10.1088/1361-6544/aa85d0}
  {\path{doi:10.1088/1361-6544/aa85d0}}.

\bibitem{leppanen2020sunklodas}
Juho Lepp\"{a}nen and Mikko Stenlund.
\newblock Sunklodas' approach to normal approximation for time-dependent
  dynamical systems.
\newblock {\em J. Stat. Phys.}, 181(5):1523--1564, 2020.
\newblock \href {https://doi.org/10.1007/s10955-020-02636-7}
  {\path{doi:10.1007/s10955-020-02636-7}}.

\bibitem{levy1935}
Paul L{\'e}vy.
\newblock Propri{\'e}t{\'e}s asymptotiques des sommes de variables
  ind{\'e}pendantes ou enchain{\'e}es.
\newblock {\em Journal des math{\'e}matiques pures et appliqu{\'e}es. Series},
  9(14):4, 1935.

\bibitem{liu2024wasserstein}
Zhenxin Liu and Zhe Wang.
\newblock Wasserstein convergence rates in the invariance principle for
  deterministic dynamical systems.
\newblock {\em Ergodic Theory Dynam. Systems}, 44(4):1172--1191, 2024.
\newblock \href {https://doi.org/10.1017/etds.2023.40}
  {\path{doi:10.1017/etds.2023.40}}.

\bibitem{liu2023wasserstein}
Zhenxin Liu and Zhe Wang.
\newblock Wasserstein convergence rates in the invariance principle for
  sequential dynamical systems.
\newblock {\em Nonlinearity}, 37(12):Paper No. 125019, 27, 2024.

\bibitem{maume2001projective}
V\'{e}ronique Maume-Deschamps.
\newblock Projective metrics and mixing properties on towers.
\newblock {\em Trans. Amer. Math. Soc.}, 353(8):3371--3389, 2001.
\newblock \href {https://doi.org/10.1090/S0002-9947-01-02786-6}
  {\path{doi:10.1090/S0002-9947-01-02786-6}}.

\bibitem{N76}
S.~V. Nagaev.
\newblock An estimate of the remainder term in the multidimensional central
  limit theorem.
\newblock In {\em Proceedings of the {T}hird {J}apan-{USSR} {S}ymposium on
  {P}robability {T}heory ({T}ashkent, 1975)}, volume Vol. 550 of {\em Lecture
  Notes in Math.}, pages 419--438. Springer, Berlin-New York, 1976.

\bibitem{NTV18}
Matthew Nicol, Andrew T\"or\"ok, and Sandro Vaienti.
\newblock Central limit theorems for sequential and random intermittent
  dynamical systems.
\newblock {\em Ergodic Theory Dynam. Systems}, 38(3):1127--1153, 2018.
\newblock \href {https://doi.org/10.1017/etds.2016.69}
  {\path{doi:10.1017/etds.2016.69}}.

\bibitem{P24}
Nicolo Paviato.
\newblock Rates for maps and flows in a deterministic multidimensional weak
  invariance principle, 2024.
\newblock Preprint.
\newblock \href {http://arxiv.org/abs/2406.06123} {\path{arXiv:2406.06123}}.

\bibitem{pene2005rate}
Fran{\c{c}}oise P{\`e}ne.
\newblock Rate of convergence in the multidimensional central limit theorem for
  stationary processes. {A}pplication to the {K}nudsen gas and to the {S}inai
  billiard.
\newblock {\em Ann. Appl. Probab.}, 15(4):2331--2392, 2005.
\newblock \href {https://doi.org/10.1214/105051605000000476}
  {\path{doi:10.1214/105051605000000476}}.

\bibitem{raic2019multivariate}
Martin Rai{\v{c}}.
\newblock A multivariate {B}erry--{E}sseen theorem with explicit constants.
\newblock {\em Bernoulli}, 25(4A):2824--2853, 2019.

\bibitem{rinott1996multivariate}
Yosef Rinott and Vladimir Rotar.
\newblock A multivariate {CLT} for local dependence with {$n^{-1/2}\log n$}
  rate and applications to multivariate graph related statistics.
\newblock {\em J. Multivariate Anal.}, 56(2):333--350, 1996.
\newblock \href {https://doi.org/10.1006/jmva.1996.0017}
  {\path{doi:10.1006/jmva.1996.0017}}.

\bibitem{rio1996sur}
Emmanuel Rio.
\newblock Sur le th\'eor\`eme de {B}erry-{E}sseen pour les suites faiblement
  d\'ependantes.
\newblock {\em Probab. Theory Related Fields}, 104(2):255--282, 1996.
\newblock \href {https://doi.org/10.1007/BF01247840}
  {\path{doi:10.1007/BF01247840}}.

\bibitem{rosenblatt1956central}
Murray Rosenblatt.
\newblock A central limit theorem and a strong mixing condition.
\newblock {\em Proc. Natl. Acad. Sci. U.S.A.}, 42(1):43--47, 1956.

\bibitem{RE83}
J.~Rousseau-Egele.
\newblock Un th\'{e}or\`eme de la limite locale pour une classe de
  transformations dilatantes et monotones par morceaux.
\newblock {\em Ann. Probab.}, 11(3):772--788, 1983.

\bibitem{sunklodas2007}
J.~Sunklodas.
\newblock On normal approximation for strongly mixing random variables.
\newblock {\em Acta Appl. Math.}, 97(1-3):251--260, 2007.
\newblock \href {https://doi.org/10.1007/s10440-007-9122-1}
  {\path{doi:10.1007/s10440-007-9122-1}}.

\bibitem{xia2011multidimensional}
Hongqiang Xia and Dayao Tan.
\newblock A multidimensional central limit theorem with speed of convergence
  for {A}xiom {A} diffeomorphisms.
\newblock {\em Acta Math. Sci. Ser. B (Engl. Ed.)}, 31(3):1123--1132, 2011.
\newblock \href {https://doi.org/10.1016/S0252-9602(11)60303-2}
  {\path{doi:10.1016/S0252-9602(11)60303-2}}.

\bibitem{zweinmuller2004kuzmin}
Roland Zweim\"{u}ller.
\newblock Kuzmin, coupling, cones, and exponential mixing.
\newblock {\em Forum Math.}, 16(3):447--457, 2004.
\newblock \href {https://doi.org/10.1515/form.2004.021}
  {\path{doi:10.1515/form.2004.021}}.

\end{thebibliography}
\bibliographystyle{plainurl}

\end{document}